\documentclass[11pt]{article}
\usepackage{amsmath}
\usepackage{amssymb}
\usepackage{amsthm}
\usepackage{graphicx}
\usepackage{setspace} 
\usepackage{enumerate}
\usepackage{url} 
\usepackage[numbers,square]{natbib}
\usepackage{xcolor}

\usepackage[colorlinks]{hyperref}
\hypersetup{
    colorlinks,%
    citecolor=blue,%
    filecolor=black,%
    linkcolor=red,%
    urlcolor=blue
}
\newcommand{\blind}{1}

\DeclareMathOperator*{\dsum}{{\sum}^\prime}
\newcommand{\s}{{\mathcal{S}}}

\newcommand{\E}{\mathbb{E}}
\newcommand{\p}{\mathbb{P}}
\newcommand{\R}{\mathbb{R}}

\newcommand{\h}{\mathcal{H}}
\newcommand{\X}{\mathcal{X}}

\newcommand{\Y}{\mathcal{Y}}
\newcommand{\Z}{\mathcal{Z}}

\newcommand{\mcg}{\mathcal{G}}
\newcommand{\emgn}{\mathcal{E}(\mathcal{G}_n)}

\newcommand{\indep}{\perp \!\!\! \perp}

\newtheorem{theorem}{Theorem}

\newtheorem{lemma}{Lemma}
\newtheorem{corollary}{Corollary}
\newtheorem{remark}{Remark}

\newtheorem{proposition}{Proposition}

\newtheorem{defn}{Definition}
\newtheorem{assump}{Assumption}

\begin{document}

\def\spacingset#1{\renewcommand{\baselinestretch}%
{#1}\small\normalsize} \spacingset{1}


\if1\blind
{
  \title{\bf A Kernel Measure of Dissimilarity between $M$ Distributions}
  \author{Zhen Huang\hspace{.2cm}\\
    Department of Statistics, Columbia University\\
    e-mail: \url{zh2395@columbia.edu}\\
    and \\
    Bodhisattva Sen\thanks{Supported by NSF grant DMS-2015376}\\
    Department of Statistics, Columbia University\\
    e-mail: \url{bodhi@stat.columbia.edu}}
  \maketitle
} \fi

\if0\blind
{
  \bigskip
  \bigskip
  \bigskip
  \begin{center}
    {\LARGE A Kernel Measure of Dissimilarity between $M$ Distributions}
\end{center}
  \medskip
} \fi

\begin{abstract}
Given $M\geq 2$ distributions defined on a general measurable space, we introduce a nonparametric (kernel) measure of multi-sample dissimilarity (KMD) --- a parameter that quantifies the difference between the $M$ distributions. The population KMD, which takes values between 0 and 1, is 0 if
and only if all the $M$ distributions are the same, and 1 if and only if all the distributions are mutually singular. Moreover, KMD possesses many properties commonly associated with $f$-divergences such as the data processing inequality and invariance under bijective transformations.
The sample estimate of KMD, based on independent observations from the $M$ distributions, can be computed in near linear time (up to logarithmic factors) using $k$-nearest neighbor graphs (for $k \ge 1$ fixed). We develop an easily implementable test for the equality of $M$ distributions based on the sample KMD that is consistent against all alternatives where at least two distributions are not equal. We prove central limit theorems for the sample KMD, and provide a complete characterization of the asymptotic power of the test, as well as its detection threshold. The usefulness of our measure is demonstrated via real and synthetic data examples; our method is also implemented in an \texttt{R} package. 
\end{abstract}

\noindent%

\spacingset{1.0} 
\section{Introduction}\label{sec:introduction}
Suppose that $\Z$ is a general measurable space, and we have $M$ distributions $P_{1},\ldots,P_{M}$ on $\Z$ from which we observe independent samples.
A natural statistical question to ask here is: ``Are these $M$ distribution the same? If they are not the same, then how different are they?". In this paper, we answer these questions
by proposing a nonparametric measure that quantifies the differences between the multiple samples.
To rephrase this in statistical language, we define a measure $\eta \equiv \eta(P_{1},\ldots,P_{M}) $ such that:
\begin{itemize}
	\item[(i)] $\eta$ is a deterministic number between $[0,1]$; 
	
	\item[(ii)] $\eta = 0$ if and only if $P_{1} =\ldots= P_{M}$ (i.e., all the $M$ distributions are the same), and 
	
	\item[(iii)] $\eta = 1$ if and only if the $M$ distributions are mutually singular, i.e., there exist disjoint measurable sets $\{A_i\}_{i=1}^M$ such that $P_i(A_i) = 1$, for $i=1,\ldots,M$. Thus, $\eta=1$ quantifies that the $M$ distributions are very different.
	
\end{itemize}
Moreover, any value between 0 and 1 of $\eta$ would convey an idea of how different these $M$ distributions are.  We call our proposed $\eta$ as a {\it kernel measure of multi-sample dissimilarity} (KMD) as its definition involves a positive semi-definite kernel matrix.

While there is a rich literature on the multi-sample testing problem:
\begin{equation}\label{eq:hypo}
{\rm H}_0: P_1=\ldots = P_M\quad\qquad {\rm against}\quad\qquad {\rm H}_1:P_i\neq P_j,\ {\rm for\ some\ }1\leq i<j\leq M,
\end{equation}
(see Section~\ref{sec:related} for a detailed review) most of these tests do not quantify to what extent these distributions are different, when the null is violated.
Moreover, many popular distances or similarities --- such as the KL-divergence, Hellinger distance, and other $f$-divergences \cite{Renyi1961entropy,   Hellinger1909, Beran1977Hellinger} --- that quantify the difference between two distributions can be difficult to estimate when $\Z$ is not a Euclidean space. In modern statistical applications, it is quite often necessary to compare more than two distributions, and they can be defined on a general measurable space. 
Two motivating applications in this direction are:

\vspace{0.05in}
{\noindent {\bf Example 1} (Multivariate functional data)}: In speech recognition, data may be viewed as functional inputs.
\begin{figure}
    \centering
    \includegraphics[width = 1\textwidth]{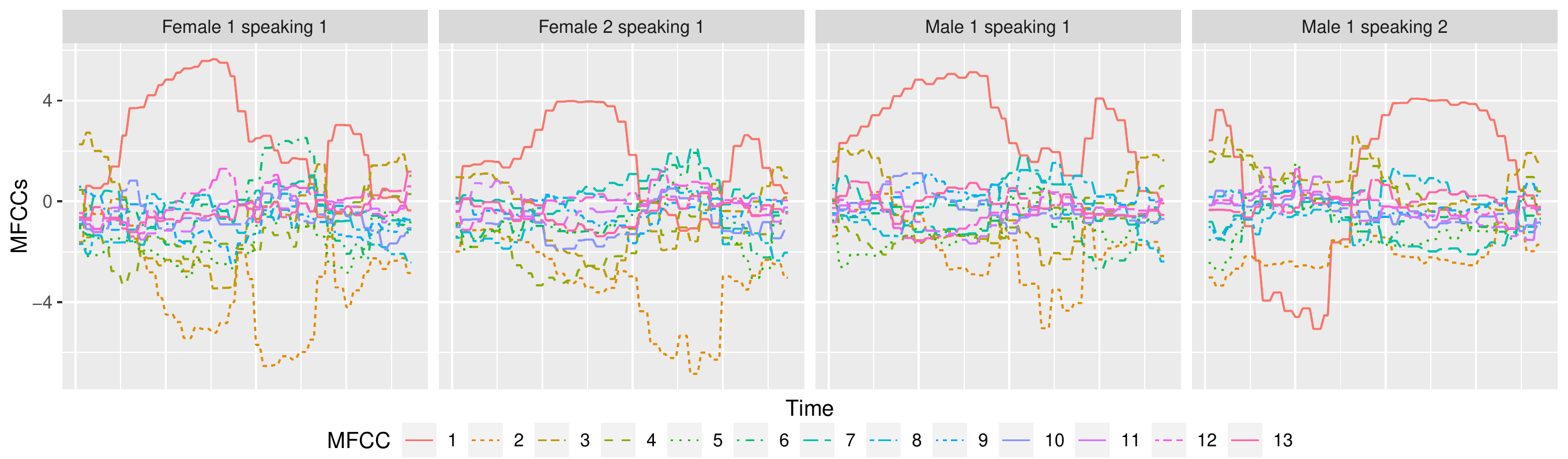}
    \caption{Four instances of spoken Arabic digits, each of which contains time series of 13 Mel Frequency Cepstrum Coefficients (MFCCs).}
    \label{mfccs}
\end{figure}
For example, Figure~\ref{mfccs} shows four instances of spoken Arabic digits from the ArabicDigits data set in \citet{gorecki2017package}, each of which
contains time series of 13 Mel Frequency Cepstrum Coefficients (MFCCs).
Hence, in this scenario, $\Z$ is the space of all 13-dimensional time series.
A natural problem here is to distinguish the 10 spoken digits, which can be seen as $M=10$ different distributions on $\Z$.
Knowing how different these $M$ distributions are on a scale between 0 to 1 provides a sense of how well any machine learning algorithm can perform in distinguishing these $M$ digits. One can also study how the digits spoken may differ by gender. For example, how males speak the number 1 should be different from how males speak the number 2, but this difference should be smaller than that between how females speak 1 and how males speak 2. A statistical measure that quantifies the extent of these differences can be obtained by our procedure.

\vspace{0.05in}

{\noindent {\bf Example 2} (Distribution over documents)}: In sentiment analysis \citep{maas2011learning,text2vec2020},
the goal is to identify the sentiment (e.g., positive, negative, neutral) of a document.
These documents may each contain hundreds of words, and are typically of different lengths.
For example, \citet{maas2011learning} considered 2000 movie reviews, each of which is either positive or negative.
If there are in total $M$ sentiments, then there are naturally $M$ distributions in the space of documents --- each distribution corresponding to a different sentiment.
Let $V$ be the vocabulary, i.e., the set of all possible words.
Then, in this case, $\Z=\{(v_1,\ldots,v_n):v_i \in V, n\in\mathbb{N}\}$ with each element in $\Z$ denoting a generic document.
Given a specific data set with $n_i$ documents, conveying the $i$-th sentiment, for $i=1,\ldots,M$, a data analyst may want to know how different the $M$ sentiments are in this data set.
Our procedure can quantify this difference, so as to suggest whether it is possible to effectively classify the different sentiments.

Often in practice, we have a distance between two objects in $\Z$ which has already been demonstrated to be useful in the domain area.
For the first example above, a useful metric is the {\it dynamic time warping} (DTW) distance between multi-dimensional time series \citep{berndt1994using,gorecki2015multivariate}, a distance which could account for differences in speaking rates between speakers and can be computed even when the two time series are discretized into different lengths.
For the second example above, one can use the {\it Jaccard distance} or other distances between documents \citep{text2vec2020}.
In these situations, applying usual Euclidean methods may require additional embedding efforts (see e.g., \cite{gysel2018neural} on document embedding).
In contrast, it will be seen that our method can directly provide an easy, practical, and interpretable way to quantify the difference between multiple distributions,
as long as a metric is available on the space $\Z$.

In this paper, given $n_i$ independent observations from $P_{i}$, $i=1,\ldots,M$, we propose and study an empirical estimator $\hat \eta$ of $\eta$, constructed using {\it geometric graphs} \citep{deb2020kernel, bhattacharya2019general} (e.g., the $k$-nearest neighbor ($k$-NN) graph for $k\geq 1$).
The main contributions of the paper, and some important properties of $\eta$ and $\hat{\eta}$ are summarized below:

\begin{enumerate}
\item We propose a nonparametric measure of multi-sample dissimilarity $\eta = \eta(P_{1},\ldots,P_{M}) \in[0,1]$  that satisfies properties (i)--(iii) mentioned above (see Theorem~\ref{thm:3prop}).
Moreover, any value of $\eta$, between 0 and 1, conveys an idea about how different these distributions are. For example, in a large class of location and scale distributions, $\eta$ increases as the ``difference" between the parameters gets larger (see Proposition~\ref{prop:monotonicity}).
Our $\eta$ also satisfies many desirable properties commonly associated with $f$-divergences \citep{Csi1967info}, including the data processing inequality and invariance under bijective transformations (Proposition~\ref{prop:Inv}) and joint convexity (Proposition~\ref{prop:convex}). Indeed, $\eta$ is a member of the multi-distribution $f$-divergence proposed in \citet{pmlr-v23-garcia12};
however, no estimation strategy was given in \cite{pmlr-v23-garcia12}.

\item We develop an estimator $\hat{\eta}$ of $\eta$, which is consistent (Theorem~\ref{thm:consist}),  interpretable, easily implementable and computationally efficient. It can be computed in  linear time (up to logarithmic factors); this is an enormous reduction from the $O(n^2)$ complexity of energy and kernel based methods \citep{szekely2013energy,Gretton12}, and other $O(n^3)$ complexity / NP complete approaches \citep{Rosenbaum05nonbi,Biswas2014Hamiltonian,deb2019multivariate,mukherjee2020distribution,hallin2020fully}. We further show that $\hat{\eta}$ is a generalization of the two-sample statistic based on $k$-NN proposed in \citet{Sc86} and \citet{Henze88NN} (see Remark~\ref{rk:connection_graph_est}).
The multi-sample test~\eqref{eq:hypo} based on $\hat{\eta}$ is also consistent against all alternatives for which $P_i\neq P_j$, for some $1\leq i<j\leq M$ (Corollary~\ref{cor:universal_consist}). 


\item The asymptotic distribution of $\hat{\eta}$ is Gaussian, which yields easy-to-use asymptotic tests. Under the null hypothesis \eqref{eq:hypo}, the permutation distribution\footnote{Given the pooled sample without the information on their sample identities, any permutation of the sample identities is equally likely under the null hypothesis \eqref{eq:hypo}. This conditional distribution is referred to as the {\it permutation distribution} in this paper.} and the unconditional distribution of $\hat{\eta}$ are both asymptotically normal (Theorem~\ref{thm:asympnull}). Further, when $\Z$ is a Euclidean space, the asymptotic null distribution is distribution-free if the common distribution has a Lebesgue density 
(see Theorem~\ref{thm:DistFree} for a precise statement).


\item We further provide in Section~\ref{sec:asymp_power_and_threshold} a complete characterization of the asymptotic power of the test for~\eqref{eq:hypo} based on $\hat \eta$ (see Theorem~\ref{thm:detection_threshold0}), and its detection threshold, using the technique of Poissonization (cf., the recent paper~\citet{BB20detection} where the detection threshold for as a class of graph-based two-sample tests is studied). In particular, we show that, under both fixed and shrinking alternatives converging to the null, $\hat{\eta}$ has an asymptotic normal distribution after proper centering; see Appendix~\ref{sec:normal_alternative} (Theorems~\ref{thm:normal_alternative} and \ref{thm:shrinking_alternative}).

\item In Section~\ref{sec:simulations}, we demonstrate the usefulness of the proposed methodology via real and synthetic data examples. Our method is also implemented in an \texttt{R} package\footnote{See \url{https://cran.r-project.org/package=KMD} and \url{https://github.com/zh2395/KMD}.}.
\end{enumerate}
The outline of the paper is as follows. In Section~\ref{sec:def_eta}, we formally define $\eta$ and investigate its properties. Its estimator $\hat{\eta}$ is studied in Section~\ref{sec:eta_hat} along with its basic characteristics. In Section~\ref{sec:asymp_dist} we provide the asymptotic distribution of $\hat{\eta}$, under the null~\eqref{eq:hypo}. A rigorous study of the power behavior of the test based on $\hat{\eta}$ and its detection threshold is given in Section~\ref{sec:asymp_power_and_threshold}. Simulations and real data experiments are provided in Section~\ref{sec:simulations}. All the proofs of our main results, further discussions, additional numerical experiments, and more results on the behavior of $\hat\eta$ under the alternative are given in Appendices~\ref{sec:general_discussion}-\ref{sec:further_simu}.

\subsection{Related Works}\label{sec:related}
The $M$-sample testing problem \eqref{eq:hypo} has been extensively studied in the statistics literature, both in parametric and nonparametric regimes. Parametric tests (e.g., $t$-test, MANOVA, likelihood ratio tests, Wald tests) are provably powerful when the underlying model assumptions hold true (see e.g.,~\cite{Lehmann2005test} and the references therein), but could have poor performance when the model is misspecified. In comparison,
nonparametric methods are usually powerful against general alternatives, under much more relaxed assumptions on the data distributions; this will be the framework adopted in this paper.

There is a long history of nonparametric two-sample tests. In the one-dimensional case, classical well-known distribution-free tests include the Kolmogorov-Smirnov test \citep{KS1939} and the Wald-Wolfowitz run test \citep{Wald1940run}. In the multivariate setting,~\citet{FR79} proposed generalizations of the Wald-Wolfowitz test based on the minimum spanning tree of the pooled sample. Its theoretical properties were further analyzed by \citet{HenzePenrose1999}. Multivariate two-sample tests based on nearest neighbor ideas were proposed in~\cite{Sc86,Henze88NN,hall2002permutation}.~\citet{Chen2017object} extended a number of two-sample tests based on type coincidences in a geometric graph.~\citet{bhattacharya2019general} proposed a general asymptotic framework for studying  graph-based two-sample tests; also see~\cite{BB20detection}.
Apart from the above tests based on geometric graphs, there are two-sample tests based on data depth~\citep{liu1993index,Zuo06depth}.
In the past decade, energy statistics \citep{szekely2013energy} and tests based on the maximum mean discrepancy (MMD) \citep{Gretton12} in the kernel literature have also drawn great attention in two-sample testing, and their equivalence has been established \citep{Sejdinovic2013}.
There are also methods that achieve finite-sample distribution-freeness by using minimum non-bipartite matching \citep{Rosenbaum05nonbi,mukherjee2020distribution}, the shortest Hamiltonian path \citep{Biswas2014Hamiltonian}, or multivariate ranks defined via optimal transport \citep{deb2019multivariate,hallin2020fully}. 

When moving to general $M$-sample testing, \citet{Petrie2016} generalized a number of graph-theoretic tests to the multi-sample scenario.
Recently \citet{mukherjee2020distribution} generalized the test in \cite{Rosenbaum05nonbi} based on minimum non-bipartite matching, retaining the exact distribution-freeness in finite samples.
Energy statistic has also been used in \citet{DISCO2010} for $M$-sample testing, providing a nonparametric extension of ANOVA.
\citet{liu1993index} proposed an index $Q$ between $[0,1]$ measuring the dissimilarity between two distributions, based on data depth, with the null being achieved when $Q=\frac{1}{2}$; however $Q$ does not satisfy properties (i)-(iii). Although some $f$-divergences, such as the Hellinger distance \citep{Hellinger1909} or the total variation distance, satisfy (i)-(iii), they cannot be easily extended beyond $M=2$.



\section{KMD: The Population Version}\label{sec:def_eta}
In this section we define the population version of our measure of dissimilarity $\eta$ between the $M$ distributions $P_1,\ldots, P_M$. Our definition of $\eta$ involves the use of a {\it reproducing kernel Hilbert space} (RKHS) over the finite discrete space $\s := \{1,\ldots,M\}$,
and hence we call our proposal the {\it kernel measure of multi-sample dissimilarity} (KMD). Although the discrete kernel $K(x,y) := I(x=y)$, for $x,y\in\s$, seems to be the most natural choice of the kernel over the discrete space $\s$, our results are applicable to other kernels as well. \vspace{0.08in}

\noindent{\bf Reproducing Kernel Hilbert Space (RKHS)}: While there is a general theory of RKHS on arbitrary spaces, the RKHS over a finite space is much simpler, which will be introduced in the following. For an introduction to the theory of RKHS and its applications in statistics we refer the reader to~\cite{BT-A04,SVM}.
By a {\it kernel} function $K:\s\times\s \to\mathbb{R}$ we mean a symmetric and nonnegative definite function, i.e., the matrix $[K(i,j)]_{i,j=1}^M$ is positive semi-definite.
The kernel $K(\cdot,\cdot)$ is said to be {\it characteristic} if
for any $(\alpha_1,\ldots,\alpha_M) \neq (0,\ldots,0)$ with $\sum_{i=1}^M \alpha_i =0$, we have $\sum_{i,j=1}^M \alpha_i\alpha_jK(i,j) > 0$. Note that the usual definition of a characteristic kernel is through the uniqueness of the {\it kernel mean embedding} \citep{charac-Bharath}, but it is equivalent to this simpler definition when the space $\s$ is finite; see e.g.,~\citep[Section 4.4]{charac-Bharath}.

\subsection{Definition of $\eta$}\label{subsec:def_eta}

Suppose that we have $n_i$ independent observations $X_{i1},\ldots,X_{i n_i}$ from the distribution $P_i$ taking values in $\Z$, for $i=1,\ldots,M$.
Denote the pooled sample $\{X_{11},\ldots,X_{1 n_1},\ldots,X_{M1},\ldots,$ $X_{Mn_M}\}$ as $\{Z_1,\ldots,Z_n\}$ with $n=n_1+\ldots+n_M$,
and the corresponding labels as $\Delta_1,\ldots,\Delta_n$, i.e., $\Delta_j = i\in\{1,\ldots,M\}$ if $Z_j$ comes from distribution $P_i$.
If $\frac{n_i}{n}\approx\pi_i \in (0,1)$ such that $\sum_{i=1}^M \pi_i = 1$, then $\{(\Delta_i,Z_i)\}_{i=1}^n$ can be ``approximately" thought of as an i.i.d.~sample of size $n$ from $(\tilde{\Delta},\tilde{Z})$ whose distribution is specified as follows:
\begin{enumerate}
\item $\mathbb{P}(\tilde{\Delta}=i)=\pi_i$, for $i=1,\ldots,M$.
\item Given $\tilde{\Delta}=i$, $\tilde{Z}$ is drawn from distribution $P_i$.
\end{enumerate}
The following lemma (proved in Appendix~\ref{pf_lem:obs}) is a crucial observation that will allow us to formally define $\eta$.
\begin{lemma}\label{lem:obs}
The $M$ distributions $P_1,\ldots, P_M$ are the same if and only if $\tilde{\Delta}\indep \tilde{Z}$,
and the $M$ distributions are mutually singular if and only if $\tilde{\Delta}$ is a function of $\tilde{Z}$.
\end{lemma}
The connection between $M$ sample testing and measures of association between $\tilde{\Delta}$ and $\tilde{Z}$ has been noted before; see e.g., \cite{Friedman1983asso,panda2019nonpar}.
The above lemma motivates the use of a certain measure of association
\citep{Chatterjee2021new,azadkia2021simple,deb2020kernel,Huang2022KPC} between $\tilde{\Delta}$ and $\tilde{Z}$ to quantify the dissimilarity between the $M$ distributions.
In particular, we adopt the ideas from the {\it kernel measure of association} (KMAc) proposed in \citet{deb2020kernel}. 
Let $\mu$ be the distribution of $(\tilde{Z},\tilde{\Delta})$ as defined above. Here we assume that $\frac{n_i}{n}$ converges to some $\pi_i \in (0,1)$ as $n\to\infty$, such that $\sum_{i=1}^M \pi_i = 1$. Let $(\tilde{Z}_1,\tilde{\Delta}_1)$ and $(\tilde{Z}_2,\tilde{\Delta}_2)$ be i.i.d.~$\mu$. Also, let $(\tilde{Z},\tilde{\Delta}, \tilde{\Delta}')$ be such that $(\tilde{Z},\tilde{\Delta}) \sim \mu$,  $(\tilde{Z},\tilde{\Delta}') \sim \mu$, and $\tilde{\Delta}$ and $\tilde{\Delta}'$ are conditionally independent given $\tilde{Z}$. Define the {\it kernel measure of multi-sample dissimilarity} (KMD) as:
\begin{eqnarray}\label{eq:def_eta}
\eta \equiv \eta(P_1,\ldots,P_M) :=   \frac{\E \big[K(\tilde{\Delta},\tilde{\Delta}{'})\big] - \E \big[ K(\tilde{\Delta}_1,\tilde{\Delta}_2)\big]  }{ \E\big[ K(\tilde{\Delta},\tilde{\Delta}) \big]-  \E \big[K(\tilde{\Delta}_1,\tilde{\Delta}_2)\big] }. \label{eq:eta}
\end{eqnarray}
Note that $\eta$, as defined above, also depends on the mixing proportions $\{\pi_i\}_{i=1}^M$, but for notational simplicity we do not highlight this dependence here.
If we let $\bar{P}:=\sum_{i=1}^M \pi_i P_i$ be the distribution of $\tilde{Z}$.
Given $\tilde{Z}=z$, the conditional probability of $\tilde{Z}$ coming from $P_i$ is $\frac{{\rm d} (\pi_i P_i)}{{\rm d} \bar{P}}(z)$, where $\frac{{\rm d} (\pi_i P_i)}{{\rm d} \bar{P}}$ denotes the Radon-Nikodym derivative of $\pi_i P_i$ with respect to $\bar{P}$. Hence $\eta$ in \eqref{eq:eta} has the alternative expression:
\begin{equation}\label{eq:eta_alternative_Radon}
\eta =   \frac{\int \sum_{i,j=1}^M \pi_i \pi_jK(i,j)\frac{{\rm d} P_i}{{\rm d} \bar{P}}(z)  \frac{{\rm d} P_j}{{\rm d} \bar{P}}(z) {\rm d}\bar{P}(z)  - \sum_{i,j=1}^M \pi_i\pi_j K(i,j) }{ \sum_{i=1}^M \pi_i K(i,i) -  \sum_{i,j=1}^M \pi_i\pi_j K(i,j)}.
\end{equation}
When $P_1,\ldots,P_{M-1}$ are absolutely continuous w.r.t.\ $P_M$, $\eta$ reduces to a member of the multi-dimensional $f$-divergences proposed in \citet{pmlr-v23-garcia12}.

Observe that if all the $M$ distributions $P_1,\ldots,P_M$ are the same, then $\tilde{\Delta}$ is independent of $\tilde{Z}$ (by Lemma~\ref{lem:obs}), and thus $(\tilde{\Delta}, \tilde{\Delta}')$ has the same distribution as
$(\tilde{\Delta}_1,\tilde{\Delta}_2)$; so the numerator of $\eta$ in \eqref{eq:def_eta} equals 0.
If all the $M$ distributions are mutually singular, then $\tilde{\Delta}$ is a function of $\tilde{Z}$ (by Lemma~\ref{lem:obs}), and $\eta=1$ (as $\tilde{\Delta}'=\tilde{\Delta}$).
The converse is also true, i.e., $\eta=0$ (resp.\ $\eta=1$) also implies $P_1=\ldots=P_M$ (resp.\  all the $M$ distributions are mutually singular); thus $\eta$ satisfies properties (i)--(iii) mentioned at the beginning of the Introduction. We formalize this in the following result (proved in Appendix~\ref{sec:pf3prop}).
\begin{theorem}\label{thm:3prop}
Suppose the kernel $K(\cdot,\cdot)$ is characteristic. 
Then $\eta$, defined in \eqref{eq:eta}, satisfies properties (i)--(iii) mentioned in the Introduction.
	
	
\end{theorem}

Theorem~\ref{thm:3prop} describes two extreme cases corresponding to $\eta = 0$ or 1.
The following proposition 
further illustrates that any value of $\eta$ between 0 and 1 indeed conveys an idea of how different the distributions are. More specifically, for common location and scale families, $\eta$ increases as the ``difference" between the distributions becomes ``larger".
\begin{proposition}[Monotonicity of $\eta$]\label{prop:monotonicity}
Consider $M=2$ and $\Z=\R^d$ for $d\geq 1$. Suppose $P_1$ is a log-concave\footnote{A distribution is called {\it log-concave} if its density can be written as $f(x)=\exp\left(-h(x)\right)$ for some convex function $h:\R^d\to \R\cup\{+\infty\}$. Many common probability distributions are log-concave, such as Gaussian distribution, uniform distribution over a convex set, and gamma distribution if the shape parameter is $\geq 1$.} distribution with density $f(\cdot)$, and $\pi_1,\pi_2\in(0,1)$ are fixed.
\begin{enumerate}
\item (Location family) Suppose $P_1$ and $P_2$ have densities $f(\cdot)$ and $f(\cdot -\theta)$ respectively, where $\theta = \lambda h$ for a fixed $h\in \R^d\backslash \{\mathbf{0}\}$ and $\lambda \geq 0$. Then $\eta(P_1,P_2)$ is a function of $\lambda$. Moreover, $\eta(P_1,P_2)$ monotonically increases from 0 to 1 as $\lambda$ increases from 0 to $\infty$.

\item (Scale family) Suppose $P_1$ and $P_2$ have densities $f(\cdot)$ and $\lambda f(\lambda\times \cdot)$ respectively, where $\lambda > 0$. Suppose further that $f$ is twice differentiable in the interior of its support. Then $\eta(P_1,P_2)$ monotonically decreases from 1 to 0 as $\lambda$ increases from 0 to 1, and monotonically increases from 0 to 1 as $\lambda$ grows from 1 to $+\infty$.
\end{enumerate}
Further, the above monotonicity of $\eta$ for scenarios 1 and 2 is strict (i.e., strictly increasing and strictly decreasing) when $\eta(P_1,P_2)< 1$ (i.e., when $P_1$ and $P_2$ are not mutually singular).
\end{proposition}


In the following, we state a few important properties of $\eta$. These properties are commonly associated with $f$-divergences \citep{Liese1987convex}. Our first result, Proposition~\ref{prop:Inv} (proved in Appendix~\ref{sec:pf_Inv}), shows that ``processing" the $M$ distributions makes them  ``less different", as measured by $\eta$.

\begin{proposition}[Data processing inequality and invariance]\label{prop:Inv}
Recall $\tilde{\Delta}$ and $\tilde{Z}\in\Z$ defined at the beginning of Section~\ref{subsec:def_eta}. Let $\Z'$ be another measurable space, and $\kappa(\cdot ,\cdot)$ be a transition kernel from $\Z$ to $\Z'$, i.e., for any $z\in \Z$, $\kappa(z,\cdot)$ specifies a distribution on $\Z'$. Suppose $P_i$ is transitioned to $Q_i$ by $\kappa$, i.e., $Q_i(B):=\int \kappa(z,B){\rm d}P_i(z)$, for every measurable set $B \subset \Z'$, for $i=1,\ldots,M$. If the mixture proportions $\{\pi_i\}_{i=1}^M$ are held fixed, then
$$\eta(P_1,\ldots,P_M) \geq \eta(Q_1,\ldots,Q_M).$$
Further, equality holds in the above display if and only if $\tilde{\Delta} \mid \tilde{Z} \stackrel{d}{=} \tilde{\Delta} \mid \tilde{Z}'$ where $\tilde{Z}'$ is obtained by passing $\tilde Z$ through the transition kernel $\kappa$. In particular, $\eta$ is invariant under any measurable bijective transformation.
\end{proposition}


The above result has an interesting consequence. Suppose that $T(Y_i)\sim Q_i$ where $Y_i\sim P_i$ and $T:\Z\to \Z'$ is a bijection. Then Proposition~\ref{prop:Inv} implies that
$\eta(P_1,\ldots,P_M) = \eta(Q_1,\ldots,Q_M)$, thereby showing that a bijective transformation of the $M$ distributions does not change our measure of dissimilarity $\eta$.

Many distance measures, including $f$-divergences, satisfy a convexity property  \citep{Liese1987convex}. The following result (proved in Appendix~\ref{sec:pf_prop_convex}) shows that $\eta$ is also jointly convex in its inputs, i.e., the $M$ distributions.

\begin{proposition}[Joint convexity of $\eta$]\label{prop:convex}
Let $P_1,\ldots,P_M$, $Q_1,\ldots,Q_M$ be distributions on $\Z$, and $\lambda \in [0,1]$.
If the mixture proportions $\{\pi_i\}_{i=1}^M$ are held fixed, then
$$\eta\left(\lambda P_1 + (1-\lambda)Q_1,\ldots,\lambda P_M + (1-\lambda)Q_M\right)\leq 
\lambda \eta( P_1 ,\ldots, P_M ) + (1-\lambda) \eta(Q_1,\ldots,Q_M).$$
\end{proposition}
Combining the above convexity result with Jensen's inequality yields the following  corollary of Proposition~\ref{prop:convex}.
\begin{corollary}[Conditioning increases $\eta$]
Let $Y_i \sim P_i$, for $i=1,\ldots, M$,  and $W$ be any random variable. With the mixture proportions $\{\pi_i\}_{i=1}^M$ held fixed, we have:
$$\E_W\left[\eta \left(Y_1|W,\ldots,Y_M|W\right) \right]\geq \eta \left(P_1,\ldots,P_M\right).$$
\end{corollary}
The above corollary has an information theoretic interpretation:
If we view $Y_i \sim P_i$ as the output distribution after passing $W$ through a ``channel" $p(Y_i|W)$ \citep{Cover2006info}, the above relation tells us that the average ``difference" between the corresponding channel transitions is at least the ``difference" between the output distributions.

\begin{remark}[Connection to Henze-Penrose dissimilarity  \citep{HenzePenrose1999}]
In fact, $\eta$, in the special case when $\Z=\R^d$ and $M=2$, has a close connection to existing measures of dissimilarity between distributions.
Suppose that $P_1$ and $P_2$ have densities  $f$ and $g$ w.r.t.~the Lebesgue measure on $\R^d$, then, it can be shown that
$
\eta =  1 - \int \frac{f(x) g(x)}{\pi_1 f(x) +\pi_2 g(x)}   \mathrm{d}x
$ (see Appendix~\ref{sec:pf_special_eta} for a proof).
This has a close connection to the Henze-Penrose dissimilarity \citep{HenzePenrose1999} defined as $
\delta(f, g) := 1 - 2 \pi_1 \pi_2 \int \frac{f(x) g(x)}{\pi_1f(x) + \pi_2 g(x)} \mathrm{d}x,$
which belongs to a general class of separation measures between distributions \citep{gyorfi1975f}, and arises as the limit of a large class of graph-based two-sample tests, including the Friedman-Rafsky test \citep{FR79,HenzePenrose1999},
nearest-neighbor based tests \citep{Sc86,Henze88NN}, and the crossmatch test \citep{Arias2016cm,mukherjee2020distribution}.
\end{remark}


%

\section{Estimation}\label{sec:eta_hat}
In this section the sample version $\hat{\eta}$ of $\eta$ will be introduced.
While the definition of $\eta$ does not require $\Z$ to be a metric space, to establish some useful properties of $\hat{\eta}$ we will assume that $\Z$ is a metric space with distance function $\rho_\Z(\cdot,\cdot)$. The notion of a \emph{geometric graph} \citep{deb2020kernel, bhattacharya2019general} is crucial in our construction. Intuitively, in a geometric graph on a set of points $z_1,\ldots,z_m\in\Z$ ($m\ge 1$), an edge $(z_i,z_j)$ will appear if $z_i$ and $z_j$ are ``close" in distance.

Formally, $\mcg$ is said to be a  geometric graph on $\mathcal{Z}$ if, given any finite subset $S$ of $\mathcal{Z}$, $\mcg(S)$ defines a graph with vertex set $S$ and the corresponding edge set $\mathcal{E}(\mcg(S))$. 
The graph can be both directed or undirected, and we will restrict ourselves to simple graphs, i.e., graphs without multiple edges and self loops. Examples of such graphs include {\it minimum spanning trees} (MSTs) and {\it $k$-nearest neighbor ($k$-NN) graphs} (where $k\ge 1$), as described below. 
\begin{enumerate}
    \item \textbf{$k$-NN graph}: The directed $k$-NN graph puts an edge from each node $z_i$ to its $k$-NNs among $z_1,\ldots,z_{i-1},z_{i+1},\ldots,z_m$ (so $z_i$ is excluded from the set of its $k$-NNs). Ties are broken at random if they occur to ensure the out-degree is always $k$. The undirected $k$-NN graph is obtained by ignoring the direction in the directed $k$-NN graph and removing multiple edges if they exist.
    
    \item \textbf{MST}: An MST is a subset of edges of an edge-weighted undirected graph which connects all the vertices in the graph with the least possible sum of edge weights and contains no cycles. For instance, given the set of points $z_1,\ldots,z_m\in\Z$ one can construct an MST for the complete graph with vertices as $z_i$'s and edge weights being the distance $\rho_\Z (z_i,z_j)$ between vertices $z_i$ and $z_j$. 
\end{enumerate}
In practice, a $k$-NN graph is recommended as the primary choice over the MST for its flexibility and computational convenience (see Remark~\ref{rk:computational_efficiency} below).


\subsection{Definition of $\hat{\eta}$}\label{sec:exact_def_eta_hat}
Recall our $M$ sample problem setting (in Section~\ref{subsec:def_eta}) and let $Z_1,\ldots,Z_n$ be the pooled sample and $\Delta_1,\ldots,\Delta_n$ be the corresponding observation labels.
Let $\mathcal{G}_n := \mathcal{G}(Z_1,\ldots ,Z_n)$ where $\mathcal{G}$ is a geometric graph on $\mathcal{Z}$ such that $(Z_i,Z_j)\in\emgn$ implies $Z_i$ and $Z_j$ are ``close". Let $d_i$ be the out-degree of $Z_i$ in $\mathcal{G}_n$.
We consider the following estimator of $\eta$:
\begin{equation}\label{eq:Eta-Hat}
\begin{aligned}
\hat{\eta} := \frac{\frac{1}{n} \sum \limits_{i=1}^n \frac{1}{d_i} \sum \limits_{j:(Z_i,Z_j) \in \emgn} K(\Delta_i,\Delta_j)-\frac{1}{n(n-1)}  \sum \limits_{i \ne j} K(\Delta_i,\Delta_j)}{\frac{1}{n}\sum_{i=1}^n K(\Delta_i,\Delta_i)-\frac{1}{n(n-1)}  \sum_{i \ne j} K(\Delta_i,\Delta_j)}.\end{aligned}
\end{equation} 
The definition of this estimator is intuitive:
first, $\E\big[ K(\tilde{\Delta},\tilde{\Delta}) \big]$ in the denominator of $\eta$ (see \eqref{eq:eta}) is estimated by $\frac{1}{n}\sum_{i=1}^n K(\Delta_i,\Delta_i)$; second, $\E \big[K(\tilde{\Delta}_1,\tilde{\Delta}_2)\big]$ is estimated by 
\begin{equation}\label{eq:2nd-Term}
\frac{\sum_{i \ne j} K(\Delta_i,\Delta_j)}{n(n-1)}=\frac{\sum_{i=1}^M n_i(n_i-1)K(i,i) + \sum_{i\neq j}^M n_in_jK(i,j)}{n(n-1)}.
\end{equation}
For estimating the remaining term $\E\left[\E \big[K(\tilde{\Delta},\tilde{\Delta}{'})|\tilde{Z}\big]\right]$,
ideally we would want two independent observations $\tilde{\Delta},\tilde{\Delta}{'}$ from the conditional distribution $\tilde{\Delta}|\tilde{Z}=Z_i$,
so we take one to be $\Delta_i$, and the other to be the label $\Delta_j$ of an observation $Z_j$ which is `close' to $Z_i$.
The first term in the numerator of \eqref{eq:Eta-Hat} formalizes this intuition via the geometric graph $\mathcal{G}_n$.
Our estimator $\hat{\eta}$ also has a nice interpretation: it can be shown that $\hat{\eta}$ is linearly related to the leave-one-out cross-validation accuracy of a $k$-NN classifier  if the discrete kernel is used; see Appendix~\ref{sec:CVinterpretation} for the details.

With a $k$-NN graph, $\hat{\eta}$ has {\it near linear computational complexity} $O(kn\log n)$. Note that finding the $k$-NN graph has computational complexity $O(kn\log n)$. Further, the first term in the numerator of~\eqref{eq:Eta-Hat} is a sum of $kn$ terms (as $d_i \equiv k$ for all $i$); the second term in the numerator of~\eqref{eq:Eta-Hat} can in fact be computed in $O(1)$ time (see~\eqref{eq:2nd-Term}).



The following theorem (proved in Appendix~\ref{sec:pfconsist}) generalizes the above special case and shows that $\hat{\eta}$ is strongly consistent in estimating $\eta$ under mild assumptions.
\begin{theorem}[Consistency]\label{thm:consist} Recall our setting as described at the beginning of Section~\ref{subsec:def_eta} where we assume that $\frac{n_i}{n}\to \pi_i\in(0,1)$ as $n\to\infty$, for $i=1,\ldots,M$. Suppose that $\tilde{\mathcal{G}}_n := \mathcal{G}(\tilde{Z}_1,\tilde{Z}_2,\ldots, \tilde{Z}_n)$, where $\tilde{Z}_1,\tilde{Z}_2,\ldots, \tilde{Z}_n$ are i.i.d.~from the mixture distribution $\sum_{i=1}^M \pi_i P_i$ on $\Z$, satisfies Assumptions \ref{assump:conv_nn}--\ref{assump:degupbd} (detailed in Appendix~\ref{sec:Assump-Graph}).
Then, $\hat{\eta}\overset{a.s.}{\longrightarrow}\eta$, as $n\to\infty$.
\end{theorem}
Assumptions~\ref{assump:conv_nn}--\ref{assump:degupbd} required on the geometric graph $\tilde{\mathcal{G}}_n$ for the above result were made in~\citet[Theorem 3.1]{deb2020kernel}.
For the $k$-NN graph and the MST, these conditions are satisfied under mild assumptions. For example, in an Euclidean space, they hold for the $k$-NN graph when $\|\tilde{Z}_1-\tilde{Z}_2\|$ has a continuous distribution and $k=o(\frac{n}{\log n})$; for the MST these are satisfied when $\tilde{Z}_1$ has an absolutely continuous distribution \citep[Proposition 3.2]{deb2020kernel}.

The following remark shows that $\hat{\eta}$ in \eqref{eq:Eta-Hat} generalizes, in various directions, many previously known nearest-neighbor type statistics for two-sample testing.
\begin{remark}[Connection to \cite{Sc86} and \cite{Henze88NN}]\label{rk:connection_graph_est}
Under a Euclidean setting with $M=2$, if the discrete kernel and the directed $k$-NN graph (with fixed $k$) are used,
then from \eqref{eq:Eta-Hat},
$$\hat{\eta} =\frac{\frac{1}{n} \sum \limits_{i=1}^n \frac{1}{d_i} \sum \limits_{j:(Z_i,Z_j) \in \emgn} I(\Delta_i=\Delta_j) - \frac{1}{n(n-1)}\sum \limits_{i=1}^M  n_i(n_i-1)   }{1-\frac{1}{n(n-1)}\sum \limits_{i=1}^M  n_i(n_i-1)}.$$
A linearly transformed version of this statistic was first proposed in \citet{Sc86} 
for two-sample testing. Its detailed asymptotic properties were later analyzed by \citet{Henze88NN}.
Our proposed $\hat{\eta}$ in this paper can be viewed as a normalized version of this statistic,
allowing $M>2$ distributions defined on general metric spaces besides the Euclidean space, using more general kernel functions in addition to the discrete kernel, and other geometric graphs besides the $k$-NN graph. Even for the directed $k$-NN graph, $k$ is allowed to be unbounded and grow with $n$ in the analysis of our consistency result (Theorem~\ref{thm:consist}) and the CLT (see Theorem~\ref{thm:asympnull} below), instead of being fixed as in \citet{Sc86,Henze88NN,BB20detection}.

\end{remark}

\section{Asymptotic Behavior of $\hat{\eta}$ under ${\rm H}_0$}\label{sec:asymp_dist}
In this section the asymptotic normality of $\hat{\eta}$ under ${\rm H}_0$ will be derived, which will yield a simple asymptotic test for the equality of the $M$ distributions in \eqref{eq:hypo}.
We show that under the null, the permutation distribution of $\hat{\eta}$, i.e., the distribution of $\hat{\eta}$ given the pooled sample $\{Z_1,\ldots,Z_n\}$, is asymptotically normal, and as a result, the unconditional asymptotic distribution of $\hat{\eta}$ is also normal. 

\subsection{Asymptotic Permutation Distribution}
Let $\mathcal{F}_n$ be the $\sigma$-algebra generated by the unordered pooled sample $\{Z_1,\ldots,Z_n\}$ without the labeling information (note that the number of observations from each distribution $n_i$, for $i=1,\ldots,M$, is known). The following result (proved in Appendix~\ref{sec:pfNormal}) 
states that the permutation distribution, i.e., the conditional distribution of $\hat{\eta}$ given $\mathcal{F}_n$, is asymptotically normal. Furthermore, the asymptotic variance of $\hat{\eta}$ under ${\rm H}_0$ is distribution-free under suitable conditions, in the sense that it does not depend on the underlying common distribution $P_1 = \ldots = P_M$.

\begin{theorem}\label{thm:asympnull}
Suppose Assumptions~\ref{assump:degree} and \ref{assump:degupbd} (in Appendix~\ref{sec:Assump-Graph}) are satisfied, $\frac{n_i}{n}\to \pi_i\in(0,1)$ as $n\to\infty$, for $i=1,\ldots,M$,
the kernel $K(\cdot,\cdot)$ is characteristic,
and the vertex degrees of $\mathcal{G}_n$ (recall that $\mathcal{G}_n$ is the geometric graph constructed on $\{Z_1,\ldots,Z_n\}$) are bounded above by
$t_n$ such that $\frac{t_n^r}{n}\to 0$ for all $r\in \mathbb{N}$ (e.g., $t_n=C(\log n)^\gamma$ for some $C>0,\gamma\geq 0$).
Then, under the null hypothesis \eqref{eq:hypo} that all the $M$ distributions are equal,
$$\frac{\hat{\eta}}{\sqrt{{\rm Var}(\hat{\eta}|\mathcal{F}_n)}}\big| \mathcal{F}_n \overset{d}{\to} N(0,1),\quad {\rm where}$$
$$\frac{\hat{\eta}}{\sqrt{{\rm Var}(\hat{\eta}|\mathcal{F}_n)}} =\frac{\sqrt{n}\left(\frac{1}{n} \sum \limits_{i=1}^n \frac{1}{{d}_i} \sum \limits_{j:(Z_i,Z_j) \in \emgn} K({\Delta}_i,{\Delta}_j) - \frac{1}{n(n-1)}  \sum \limits_{i \ne j} K(\Delta_i,\Delta_j)\right)}{\sqrt{\tilde{a} \left( \tilde{g}_1 + \tilde{g}_3 - \frac{2}{n-1}\right) + \tilde{b} \left( \tilde{g}_2 - 2\tilde{g}_1 -2\tilde{g}_3 - 1 + \frac{4}{n-1} \right) + \tilde{c}\left(\tilde{g}_1-\tilde{g}_2+\tilde{g}_3 +\frac{n-3}{n-1} \right)}},$$
with $\tilde{g}_1 := \frac{1}{n}\sum_{i=1}^n \frac{1}{d_i}$, $\tilde{g}_2 := \frac{1}{n}\sum_{i,j=1}^n \frac{T^{\mathcal{G}_n}(i,j)}{d_id_j}$,
$T^{\mathcal{G}_n}(i,j):=\sum_{k=1}^n I\{(Z_i,Z_k),(Z_j,Z_k)\in\emgn \}$ being the number of common out-neighbors, $\tilde{g}_3:=\frac{1}{n}{\sum_{i,j:(Z_i,Z_j),(Z_j,Z_i)\in\emgn}}\frac{1}{d_id_j}$, and
\begin{equation}\label{eq:tilde_a_b_c}
\begin{aligned}
\tilde{a} &:=\frac{1}{n(n-1)}{\sum}' K^2(\Delta_i,\Delta_j),\qquad \quad
\tilde{b} \;:=\frac{1}{n(n-1)(n-2)}{\sum}'  K(\Delta_i,\Delta_j) K(\Delta_i,\Delta_l),\\
\tilde{c} &:=\frac{1}{n(n-1)(n-2)(n-3)}{\sum}'  K(\Delta_i,\Delta_j) K(\Delta_l,\Delta_m).
\end{aligned}
\end{equation}
Here $\sum'$ means the summation is over distinct indices. Further, ${\rm Var}(\hat{\eta}|\mathcal{F}_n)=O\left(\frac{1}{n}\right)$.
From the above conditional CLT, the unconditional CLT also follows: $\frac{\hat{\eta}}{\sqrt{{\rm Var}(\hat{\eta}|\mathcal{F}_n)}} \overset{d}{\to} N(0,1).$
\end{theorem}

Although the statement of the above theorem is similar to \citet[Theorem 4.1]{deb2020kernel}, the proof technique is quite different, since we are dealing with the conditional distribution given $\mathcal{F}_n$ instead of the unconditional one. 
The main technical tool used here is a modification of the CLT in \citet{Pham1989permu}, based on moment matching.
Such a technique has been used to show the asymptotics for a variety of graph-based statistics \citep{Bloemena1964graph,Henze88NN,HenzePenrose1999,Petrie2016}.
Our general result (see Theorem~\ref{thm:extPham} in Appendix~\ref{sec:pfNormal}), compared to~\cite{Pham1989permu}, can deal with unbounded vertex degrees, which may be of independent interest.

Theorem~\ref{thm:asympnull} has the benefit that if we regard the randomness as coming from random permutations (with the pooled sample fixed), then the variance ${\rm Var}(\hat{\eta}|\mathcal{F}_n)$ can be computed exactly.
This could lead to a better approximation of the sampling distribution of $\hat{\eta}$,
compared to the test that uses the limiting value of ${\rm Var}(\hat{\eta}|\mathcal{F}_n)$ (see Section~\ref{sec:asymp_dist_free} below).
Note that we allow $k$, the number of nearest neighbors, to be unbounded and grow with $n$ instead of being fixed as in some previous relevant works
\citep{Sc86,Henze88NN,Petrie2016}.


\begin{corollary}\label{cor:universal_consist}
Consider the testing problem~\eqref{eq:hypo}. Under the assumptions of Theorem~\ref{thm:asympnull}, if furthermore Assumption~\ref{assump:conv_nn} (in Appendix~\ref{sec:Assump-Graph}) holds, then
the $M$-sample test with rejection region:
\begin{equation}\label{eq:rej_region}
\frac{\hat{\eta}}{\sqrt{{\rm Var}(\hat{\eta}|\mathcal{F}_n)}} \geq z_{1-\alpha},
\end{equation}
where $z_{\alpha}$ is the $\alpha^{\rm th}$ quantile of the standard normal distribution, has asymptotic level $\alpha \in (0,1)$ and is consistent against any alternative where at least two distributions are different.
\end{corollary}

The above result is a direct consequence  of Theorem~\ref{thm:asympnull} and the fact that $\hat{\eta}$ converges in probability to $\eta>0$ under any alternative (by Theorem~\ref{thm:consist}). 

\begin{remark}[Computational efficiency]\label{rk:computational_efficiency}
Another appealing property of our method is that it is computationally efficient and easy to implement.
When a Euclidean $k$-NN graph is used, the computation complexity for $\hat{\eta}$ and $\frac{\hat{\eta}}{\sqrt{{\rm Var}(\hat{\eta}|\mathcal{F}_n)}}$ is $O(kn\log n)$ (see Appendix~\ref{sec:implement_scheme} for more detailed implementation schemes; in particular, $\tilde{a}$, $\tilde{b}$, $\tilde{c}$ can be computed in $O(1)$ time given $n_1,\ldots,n_M$), which is near linear time when $k$ does not grow too fast. 
\end{remark}

\subsection{Asymptotic Distribution-Free Property}\label{sec:asymp_dist_free}
{
We will show in this sub-section that the asymptotic distribution of 
$\hat{\eta}$ under ${\rm H}_0$ (see~\eqref{eq:hypo}) is distribution-free under mild assumptions.
Such a property actually holds true for a variety of graph-based test statistics \citep{Henze88NN,HenzePenrose1999,Rosenbaum05nonbi},
provided that the null distribution is absolutely continuous and the geometric graph $\mathcal{G}$ is ``local" in the sense of a {\it stabilizing} graph \citep{penrose2003weak,Baryshnikov2005Gaussian,BB20detection}, which includes MST and $k$-NN graphs (with fixed $k$) \citep{penrose2003weak}.}

From Theorem~\ref{thm:asympnull} above, it is clear that the variance of $\hat{\eta}$ under ${\rm H}_0$ involves $\tilde{a}$, $\tilde{b}$, $\tilde{c}$, $\tilde{g}_1$, $\tilde{g}_2$, and $\tilde{g}_3$.
From \eqref{eq:tilde_a_b_c}, it is easy to see that $\tilde{a}$, $\tilde{b}$, $\tilde{c}$ converge to some $a$, $b$, $c$ depending only on the kernel and the mixture proportions $\{\pi_i\}_{i=1}^M$. It can also be shown that $\tilde{g}_1$, $\tilde{g}_2$, $\tilde{g}_3$ converge to some $g_1$, $g_2$, $g_3$ respectively. To motivate these limits, take $\tilde{g}_1 = \frac{1}{n}\sum_{i=1}^n \frac{1}{d_i}$ as an example. Under the null, $Z_1,\ldots,Z_n$ are i.i.d.~from a common distribution.
In a ``local" graph, $d_i$, the out-degree of $Z_i$, may only depend on the points near $Z_i$.
If the common distribution has a density continuous at $Z_i$, then the points near $Z_i$ are approximately sampled from a distribution with constant density.
This means that if the geometric graph is {\it translation} and {\it scale invariant} (defined formally later), then $\frac{1}{d_i}$ may be similar in distribution to $\frac{1}{d(\mathbf{0},\mathcal{G}(\mathcal{P}_1^\mathbf{0}))}$, where $\mathcal{P}_1$ is the homogeneous Poisson process\footnote{For a homogeneous Poisson process on $\R^d$ with intensity $\lambda$ (denoted by $\mathcal{P}_\lambda$), the number of points in a set $A\subset\R^d$ follows a
${\rm Poisson}(\int_A \lambda {\rm d}x)$ distribution, and the number of points in disjoint sets are independent.}
of intensity 1 on $\mathbb{R}^d$, $\mathcal{P}_1^\mathbf{0}$ is $\mathcal{P}_1\cup\{\mathbf{0}\}$, and
$d\big(\mathbf{0},\mathcal{G}(\mathcal{P}_1^\mathbf{0})\big)$ is the out-degree of $\mathbf{0}$ in $\mathcal{G}(\mathcal{P}_1^\mathbf{0})$ --- the geometric graph constructed on $\mathcal{P}_1^\mathbf{0}$.
Hence, it is reasonable to expect that $\tilde{g}_1=\frac{1}{n}\sum_{i=1}^n \frac{1}{d_i}$ converges to $g_1 := \mathbb{E}\left[\frac{1}{d(\mathbf{0},\mathcal{G}(\mathcal{P}_1^\mathbf{0}))}\right]$, which is distribution-free.

In the following, we introduce the necessary mathematical concepts in order to state our formal results.
A geometric graph $\mathcal{G}$ on $\R^d$ is said to be {\it translation invariant} if translation by $x$ induces a graph isomorphism\footnote{That is, for $x_1,x_2\in\X$, $(x_1,x_2)$ is an edge in $\mathcal{G}(\X)$ if and only if $(x_1+x,x_2+x)$ is an edge in $\mathcal{G}(\X+x)$.} from $\mathcal{G}(\X)$ to $\mathcal{G}(x+\X)$ for all $x\in\R^d$ and all finite $\X\subset \mathbb{R}^d$. 
Similarly, $\mathcal{G}$ is {\it scale invariant} if scalar multiplication by $a$ induces a graph isomorphism from $\mathcal{G}(\X)$ to $\mathcal{G}(a\X)$ for all finite $\X$ and all $a > 0$.
For $\lambda\geq 0$, denote by $\mathcal{P}_\lambda$ the homogeneous Poisson process
of intensity $\lambda$ in $\mathbb{R}^d$, and define $\mathcal{P}_\lambda^x :=\mathcal{P}_\lambda\cup\{x\}$.
$\mathcal{G}$ is said to be {\it stabilizing} on $\mathcal{P}_\lambda$ if, for almost all realizations $\mathcal{P}_\lambda$, there exists a random variable $R<\infty$ such that the set of edges incident at the origin is not changed by modifying points outside a ball of radius $R$, i.e.,
$\mathcal{E}\left(\mathbf{0},\mathcal{G}\big((\mathcal{P}_\lambda^\mathbf{0}\cap B(\mathbf{0},R))\cup\mathcal{A}\big)\right) = \mathcal{E}\left(\mathbf{0},\mathcal{G}\big(\mathcal{P}_\lambda^\mathbf{0}\cap B(\mathbf{0},R)\big)\right)$
for all finite $\mathcal{A}\subset \mathbb{R}^d\backslash B(\mathbf{0},R)$, where $B(\mathbf{0},R)$ is the closed Euclidean ball of radius $R$ centered at the origin $\mathbf{0}\in\mathbb{R}^d$, and $\mathcal{E}(x,\mathcal{G}(\X))=\{(x,y):(x,y){\rm\ is\ an\ edge\ in\ } \mathcal{G}(\X)\}$ is the set of edges of $\mathcal{G}(\X)$ incident\footnote{Note that in a directed graph, $\mathcal{E}(x,\mathcal{G}(\X))$ only includes the edges starting from $x$.} to $x\in\X$. In such a case, the definition of $\mathcal{G}$ can be extended to the infinite point set $\mathcal{P}_\lambda^\mathbf{0}$, with $
\mathcal{E}\left(\mathbf{0},\mathcal{G}(\mathcal{P}_\lambda^\mathbf{0})\right):=\mathcal{E}\left(\mathbf{0},\mathcal{G}\big(\mathcal{P}_\lambda^\mathbf{0}\cap B(\mathbf{0},R)\big)\right)$.
It is known that both MST and $k$-NN graphs (with fixed $k$) are translation and scale invariant, and stabilizing on $\mathcal{P}_\lambda$ for all $\lambda>0$; see e.g.,~\citep{penrose2003weak}. The following result (proved in Appendix~\ref{sec:pfDistFree}) formally states the asymptotic distribution-free property of $\hat{\eta}$.

\begin{theorem}\label{thm:DistFree}
Under the same assumptions as in Theorem~\ref{thm:asympnull}, if furthermore $\mathcal{G}$ is translation and scale invariant, and stabilizing on $\mathcal{P}_\lambda$ for some $\lambda>0$, then under the null hypothesis \eqref{eq:hypo} where $P_1=\ldots=P_M$ is assumed to have a Lebesgue density on $\R^d$, we have:
$$
\sqrt{n}\hat{\eta}\overset{d}{\to}N\left(0,\sigma^2_{\mathcal{G},K,d,\pi}\right),
$$
where $\sigma^2_{\mathcal{G},K,d,\pi}$ is a positive constant not depending on the common density. More specifically,
\begin{equation}\label{eq:null_variance}
\sigma^2_{\mathcal{G},K,d,\pi} := \frac{{a} \left( {g}_1 + {g}_3 \right) + {b} \left( {g}_2 - 2{g}_1 -2{g}_3 - 1  \right) + {c}\left({g}_1-{g}_2+{g}_3 +1 \right)}{\left(\sum_{i=1}^M \pi_i K(i,i)-\sum_{i,j=1}^M \pi_i\pi_j K(i,j)\right)^2} >0,\quad {\rm where}
\end{equation}
$$
\begin{aligned}
a &:= \sum_{i,j=1}^M \pi_i\pi_j K^2(i,j),\qquad b:=\sum_{i,j,l=1}^M \pi_i\pi_j\pi_l K(i,j)K(i,l),\qquad c:=\left(\sum_{i,j=1}^M \pi_i\pi_j K(i,j)\right)^2,\\
g_1&:=\mathbb{E}\left[\frac{1}{d\big(\mathbf{0},\mathcal{G}(\mathcal{P}_1^\mathbf{0})\big)}\right],\quad \ 
g_2:=\mathbb{E}\left[\sum_{y\neq z:(y,\mathbf{0}),(z,\mathbf{0})\in\mathcal{E}(\mathcal{G}(\mathcal{P}_1^\mathbf{0}))}\frac{1}{d\left(y,\mathcal{G}(\mathcal{P}_1^\mathbf{0})\right) d\left(z,\mathcal{G}(\mathcal{P}_1^\mathbf{0})\right)}\right] + g_1,\\
g_3&:=\mathbb{E}\left[ \sum_{y:(y,\mathbf{0}),(\mathbf{0},y)\in\mathcal{E}(\mathcal{G}(\mathcal{P}_1^\mathbf{0}))}\frac{1}{d\left(\mathbf{0},\mathcal{G}(\mathcal{P}_1^\mathbf{0})\right) d\left(y,\mathcal{G}(\mathcal{P}_1^\mathbf{0})\right)} \right].
\end{aligned}$$
\end{theorem}
The proof of the above result proceeds by showing that the permutation variance in Theorem~\ref{thm:asympnull} converges to the distribution-free limiting variance in Theorem~\ref{thm:DistFree}.
Compared to Theorem~\ref{thm:DistFree},
Theorem~\ref{thm:asympnull} is more general in the sense that it does not require a Euclidean space.
Even in a Euclidean space, Theorem~\ref{thm:asympnull} is based on the exact permutation variance which may yield a more accurate approximation of the sampling distribution of $\hat{\eta}$ than Theorem~\ref{thm:DistFree}, especially when the dimension is high, where the convergence of $\tilde{g}_i$, for $i=1,2,3$, may be slower; see Appendix~\ref{sec:validation_distfree} for empirical evidence.

\section{Asymptotic Power and Detection Threshold}\label{sec:asymp_power_and_threshold}
We have already seen that $\hat{\eta}$ is asymptotically normal under the null.
A natural question to ask now is: ``what is the asymptotic distribution of $\hat{\eta}$ under alternatives?". It turns out that under a fixed alternative or certain shrinking alternatives converging to the null,
$\hat{\eta}$ is also asymptotically normal, if it is properly centered. Using these CLTs, we can provide a complete characterization of the asymptotic power of the test~\eqref{eq:rej_region}. We will focus on $\Z =\R^d$ in this section, and assume that $P_i$ has a density $f_i$ w.r.t.~the Lebesgue measure.

The asymptotic distribution under alternatives and the detection threshold of graph-based statistics were not available until the recent work by \citet{BB20detection},
where the analysis was carried out under a Poissonized setting --- instead of assuming $\frac{n_i}{n}\to\pi_i$ as $n \to \infty$, it is assumed that $n_i\sim {\rm Poisson}(N_i)$,
and $\frac{N_i}{N}\to\pi_i>0$
as $N\to\infty$ where $N := \sum_{i=1}^M N_i$.
In the Poissonized framework, the pooled data $\{Z_1,\ldots,Z_n\}$ follows a non-homogeneous Poisson point process. The spatial independence of this Poisson process facilitates the computation of the variance of $\hat \eta$ under alternatives, and helps us establish CLTs of $\hat{\eta}$ using Stein's method \citep{chen2004normal}. We consider this Poissonized setting in this section.\vspace{0.08in}

\noindent \textbf{CLT under a general fixed alternative.}
Under ${\rm H}_0$, $\hat{\eta}$ is unbiased in estimating $\eta \equiv 0$, so it is reasonable to conjecture that $\sqrt{n}(\hat{\eta}-\eta)$ would converge to a normal distribution (as shown in Theorem~\ref{thm:asympnull}).
However, this is not the case under a fixed  alternative, where the bias of $\hat{\eta}$ can be large and dominating.
Although $\hat{\eta}$ converges to $\eta$ under any fixed alternative (Theorem~\ref{thm:consist}),
the distance between $\hat{\eta}$ and $\eta$ under such an alternative is in general of order $n^{-1/d}$ \citep[Corollary 5.1]{deb2020kernel}, having the same order as the distance between a data point and its nearest neighbor. Hence in such a case, it is unrealistic to expect that $\sqrt{n}(\hat{\eta}-\eta)$ would be asymptotically normal.
Instead, we show in Theorem~\ref{thm:normal_alternative} in Appendix~\ref{sec:normal_alternative} that $\hat{\eta}$ is asymptotically normal after appropriate centering. \vspace{0.08in}

\noindent \textbf{CLT under shrinking alternatives.}
Given that our level $\alpha$ test \eqref{eq:rej_region} is consistent against all fixed alternatives, it is natural to study its power behavior under shrinking alternatives converging to the null. For this purpose, a CLT under shrinking alternatives is needed, which is shown in Theorem~\ref{thm:shrinking_alternative} in Appendix~\ref{sec:normal_alternative} under the setting where the $i$-th distribution has a Lebesgue density $f_i^N$ converging Lebesgue almost everywhere to some density $f$ as $N\to\infty$, for $i=1,\ldots,M$.
\vspace{0.08in}

{\noindent \textbf{Asymptotic power and detection threshold.} Here we answer the following question: ``Can we characterize the exact limiting power of our test \eqref{eq:rej_region} under a sequence of shrinking alternatives?". To answer the above question, we consider $M=2$ and study the power of the test along a parametric sub-model \citep{BB20detection}:
suppose $P_i$ has Lebesgue density $f(\cdot |\theta_i)$ on $\R^d$ belonging to the parametric family $\{f(\cdot| \theta): \theta \in \Theta \subset \R^p\}$ ($p \ge 1$), for $i=1,2$. 

Theorem~\ref{thm:detection_threshold0} below (proved in Appendix~\ref{sec:detection_threshold}) describes the asymptotic power of our test \eqref{eq:rej_region}, whose proof depends on the CLT under shrinking alternatives. As a consequence of Theorem~\ref{thm:detection_threshold0}, we will also be able to answer the closely related question: ``At what rate should $\theta_2$ converge to $\theta_1$ (assumed fixed) so that our test \eqref{eq:rej_region} would be powerless (i.e., power $\leq \alpha$) if the convergence is faster than the rate, and would have asymptotic power 1 if the convergence is slower than the rate, as the sample size increases?".
In such a situation, the order of $|\theta_2-\theta_1|$ as a function of the sample size $N$ is called the {\it detection threshold} of the test~\citep{BB20detection}.}

\begin{theorem}\label{thm:detection_threshold0}
Suppose $\{\mathbb{P}_\theta\}_{\theta\in \Theta}$ is a parametric family of distributions with a convex parameter space $\Theta\subset \R^p$; we further assume that $\mathbb{P}_\theta$ has a Lebesgue density $f(\cdot|\theta)$. Suppose that $M=2$ and the discrete kernel $K$ and the $k$-NN graph (for a fixed $k \ge 1$) are used in defining $\hat{\eta}$. Let $P_1$ and $P_2$ have densities $f(\cdot|\theta_1)$ and $f(\cdot|\theta_2)$ respectively, and  $\theta_2 = \theta_1 + \varepsilon_N$, where $\varepsilon_N \to 0$ as $N \to \infty$. {
Let $\sigma^2_{\mathcal{G},K,d,\pi}$ be as defined in Theorem~\ref{thm:DistFree}, and ${\rm H}_x f(x|\theta_1)$ be the Hessian of $f(x|\theta_1)$ (taken w.r.t.\ $x$)}. Also, define, for $h \in \R^p$,
\begin{equation}\label{eq:a0}
a_{k,\theta_1}(h) := -\frac{ (1-2\pi_2)C_{k,2}}{4\,k \,d\,\sigma_{\mathcal{G},K,d,\pi}}\int h^\top \nabla_{\theta_1}\left(\frac{ {\rm tr}({\rm H}_x f(x|\theta_1))}{f(x|\theta_1)}\right)f^{\frac{d-2}{d}}(x|\theta_1) {\rm d}x,
\end{equation}
$$b_{k,\theta_1}(h):=\frac{\pi_1\pi_2}{\sigma_{\mathcal{G},K,d,\pi}}\E_{X\sim f(\cdot|\theta_1)} \left[\frac{h^\top \nabla_{\theta_1}f(X|\theta_1)}{f(X|\theta_1)}\right]^2,\;{\rm with} \;\;C_{k,2}:=\E \Bigg[ \sum_{x\in\mathcal{P}_1^{\mathbf{0}}: (\mathbf{0},x)\in\mathcal{E}(\mathcal{G}_{k{\rm NN}}(\mathcal{P}_1^{\mathbf{0}}))} \|x\|^2 \Bigg].$$
{Under suitable assumptions on (i) the smoothness of the parametric family, (ii) conditions such that the $k$-NN graph is nicely behaved 
(see Appendix~\ref{sec:detection_threshold_of_our_test} for the detailed list of assumptions), and (iii) $\sqrt{N}\left(\frac{N_i}{N} - \pi_i\right)\to 0$, as $N\to \infty$, for $i=1,\ldots,M$, we have the following result.}
\begin{itemize}
\item[1.] If $d\leq 8$, then the following hold:
\begin{itemize}
\item[(a)] when $\|N^{\frac{1}{4}}\varepsilon_N\|\to 0$: the limiting power of the test \eqref{eq:rej_region} is $\alpha$.

\item[(b)] when $N^{\frac{1}{4}}\varepsilon_N \to h \in \R^d$: if $d\leq 7$, the limiting power of \eqref{eq:rej_region} is
$\Phi \left(z_\alpha + b_{k,\theta_1}(h)\right)$; if $d=8$ the limiting power is
$\Phi \left(z_\alpha + a_{k,\theta_1}(h) + b_{k,\theta_1}(h)\right)$.

\item[(c)] when $\|N^{\frac{1}{4}}\varepsilon_N\|\to \infty$: the limiting power of the test \eqref{eq:rej_region} is 1.
\end{itemize}
\item[2.] If  $d \geq 9$, then the following hold:
\begin{itemize}
\item[(a)] when $\|N^{\frac{1}{2}-\frac{2}{d}}\varepsilon_N\|\to 0$: the limiting power of the test \eqref{eq:rej_region} is $\alpha$.

\item[(b)] when $N^{\frac{1}{2}-\frac{2}{d}}\varepsilon_N \to h \in \R^d$: the limiting power of the test \eqref{eq:rej_region} is 
$\Phi \left(z_\alpha + a_{k,\theta_1}(h) \right)$.

\item[(c)] when $\|N^{\frac{1}{2}-\frac{2}{d}}\varepsilon_N\|\to \infty$ such that $\|N^{\frac{2}{d}}\varepsilon_N\|\to 0$: then depending on whether
\begin{equation}\label{eq:threshold0}
N^{\frac{1}{2}-\frac{2}{d}}(1-2\pi_2)\int \varepsilon_N^\top\nabla_{\theta_1}\left(\frac{{\rm tr}({\rm H}_x f(x|\theta_1))}{f(x|\theta_1)}\right)f^{\frac{d-2}{d}}(x|\theta_1) {\rm d}x\to\left\{\begin{aligned}
&\infty,\\
&-\infty,\\
\end{aligned}\right.\end{equation}
the limiting power of the test \eqref{eq:rej_region} is 0 or 1, respectively.

\item[(d)] when $N^{\frac{2}{d}}\varepsilon_N \to h \in \R^d$: the limiting power of the test \eqref{eq:rej_region} is 0 or 1, depending on whether $a_{k,\theta_1}(h)+ b_{k,\theta_1} (h)$ is negative or positive, respectively.

\item[(e)] $N^{\frac{2}{d}}\varepsilon_N\to\infty$: the limiting power of the test \eqref{eq:rej_region} is 1.
\end{itemize}
\end{itemize}
\end{theorem}
The above result shows that the detection threshold of $\hat{\eta}$ exhibits a ``$d=8$" phenomenon (see~\citet{BB20detection}):
when $d\leq 8$, the detection threshold is $N^{-\frac{1}{4}}$;
while when $d\geq 9$, the detection threshold is somewhere between $N^{-\frac{2}{d}}$ and $N^{-\frac{1}{2}+\frac{2}{d}}$,
depending on the direction of $\varepsilon_N$ and the sign of $a_{k,\theta_1}$; see \eqref{eq:a0} and \eqref{eq:threshold0}.
If $\varepsilon_N=\delta_N h$ for $\delta_N>0$ and some fixed $h\in\R^d$, and
$a_{k,\theta_1}(h) > 0$, then the detection threshold is $\delta_N \sim N^{-\frac{1}{2}+\frac{2}{d}}$;
on the other hand, if $a_{k,\theta_1}(h) < 0$, then the detection threshold is $\delta_N \sim N^{-\frac{2}{d}}$.
When $a_{k,\theta_1}(h) = 0$, the precise location of the detection threshold has to be determined on a case by case basis (see \citep{BB20detection}).

The detection threshold for some particular choices of $P_1$ and $P_2$ are already known: for distinguishing two truncated normal distributions with different location parameters the detection threshold is $N^{-\frac{1}{4}}$ for all $d$, while for distinguishing two truncated normals with different scale parameters the detection threshold can attain both $N^{-\frac{2}{d}}$ and $N^{-\frac{1}{2}+\frac{2}{d}}$ for $d\geq 9$ depending on the sign of $a_{k,\theta_1}(h)$ (see \citep[Section 4.2]{BB20detection}). Thus in certain cases, this nonparametric detection threshold can be very close to the parametric rate $N^{-\frac{1}{2}}$.

In Appendix~\ref{sec:check_detection_threshold}, we empirically illustrate that the same detection threshold also holds for the non-Poissonized setting, i.e., the original setting of Section~\ref{sec:def_eta} where the $n_i$'s are nonrandom constants instead of $n_i\sim {\rm Poisson}(N_i)$.

Our results crucially use the general framework established in \cite{BB20detection}. However, we fill in some gaps in the original proof of \cite[Theorem 4.2]{BB20detection}, where only the CLT under a fixed alternative was shown; the CLT under shrinking alternatives was not explicitly formulated, but was assumed to hold instead. Our CLTs (Theorems~\ref{thm:normal_alternative} and \ref{thm:shrinking_alternative}) are also applicable to $M\geq 2$ distributions.

\section{Numerical Studies}\label{sec:simulations}
In this section, the finite-sample performance of our methods will be investigated on both real and synthetic data. We compare the power behavior of the tests based on $\hat{\eta}$ to other competing methods, based on synthetic data. We also analyze many real data sets and demonstrate the usefulness of KMD. Further simulation experiments are relegated to Appendix~\ref{sec:further_simu}.


\subsection{Power Study on Synthetic Data}
The empirical power behavior of some special cases of our test statistic has been partly investigated in~\citet{Sc86} and~\citet{Petrie2016} as $\hat{\eta}$ is equivalent to the statistics in these papers when $M=2$ with $k$-NN graphs and the discrete kernel.
Thus, in this subsection, we focus on the case $M >2$. We consider the following settings:

\begin{enumerate}
\item {\it Normal location}: $P_1 = N(\mathbf{0},I_d)$, $P_2 = N(0.1\cdot \mathbf{1},I_d)$, $P_3=N(0.2\cdot \mathbf{1},I_d)$, as $d$ varies.

\item {\it Normal scale}: $P_1=N(\mathbf{0},I_d)$, $P_2=N(\mathbf{0},1.5\cdot I_d)$, $P_3=N(\mathbf{0},2\cdot I_d)$, as $d$ varies.

\item $t$-{\it distribution location}: $P_1$ (on $\R^{16}$) has each coordinate drawn i.i.d.~from $t(1)$ with noncentrality parameter $\delta$; $P_2 = P_3$ has each coordinate drawn i.i.d.~from $t(1)$. 

\item {\it U-shaped scale}: $P_1 = P_2$ is a ``U-shaped" distribution which is a mixture of
$$
\scriptsize  N\left(\begin{pmatrix}
0\\
0
\end{pmatrix},
\begin{pmatrix}
2 & 0\\
0 & \frac{1}{8}
\end{pmatrix} \right), \quad
N\left(\begin{pmatrix}
-3\\
1
\end{pmatrix}, 
\begin{pmatrix}
\frac{1}{2} & -\frac{1}{3}\\
-\frac{1}{3} & \frac{1}{2}
\end{pmatrix} \right), \quad
N\left(\begin{pmatrix}
3\\
1
\end{pmatrix},
\begin{pmatrix}
\frac{1}{2} & \frac{1}{3}\\
\frac{1}{3} & \frac{1}{2}
\end{pmatrix} \right)
$$
with mixing weights $\frac{1}{2}, \frac{1}{4}, \frac{1}{4}$. $P_3$ is a scalar multiple (with scale $\in[1,1.4]$) of $P_1$.

\item {\it S-shaped rotation}: $Y_1 \sim P_1 \equiv P_2$ is a ``S-shaped" distribution which is a mixture of
$$
\scriptsize N\left(\begin{pmatrix}
-\frac{9}{2}\\
-\frac{1}{2}
\end{pmatrix},
\begin{pmatrix}
\frac{3}{2} & -\sqrt{\frac{3}{8}}\\
-\sqrt{\frac{3}{8}} &1
\end{pmatrix} \right),
N\left(\begin{pmatrix}
0\\
-\frac{1}{2}
\end{pmatrix},
\begin{pmatrix}
\frac{3}{2} & \sqrt{\frac{3}{8}}\\
\sqrt{\frac{3}{8}} &1
\end{pmatrix} \right),
N\left(\begin{pmatrix}
\frac{9}{2}\\
1
\end{pmatrix},
\begin{pmatrix}
\frac{3}{2} & -\sqrt{\frac{3}{8}}\\
-\sqrt{\frac{3}{8}} &1
\end{pmatrix} \right).
$$
with weights $\frac{1}{3}, \frac{1}{3}, \frac{1}{3}$. $P_3$ is obtained by multiplying $Y_1$ by 
$\tiny \begin{pmatrix}
\cos \theta & \sin \theta \\
-\sin \theta & \cos \theta 
\end{pmatrix}$, for $\theta\in [0,0.1\pi]$.

\item {\it Spherically symmetric}: $P_1, P_2, P_3$ are {\it spherically symmetric distributions}\footnote{$Y_1 \sim P_1$ has a spherically symmetric distribution with radial distribution $L_1$ if $Y_1 = L_1 \times U$, where $U$ is the uniform distribution over the unit sphere $\mathcal{S}^{d-1}$ in $\R^d$.} on $\R^d$ with different radial densities. We assume that $P_1, P_2$ and $P_3$ have Uniform[0,1], ${\rm Beta}(1-\alpha,1+\alpha)$, and ${\rm Beta}(1+\alpha,1-\alpha)$ radial densities respectively (as $\alpha$ varies). 
\end{enumerate}

The first two settings are classical, which were also considered in~\citet{mukherjee2020distribution}.
\begin{figure}
    \centering
    \includegraphics[width = 1\textwidth]{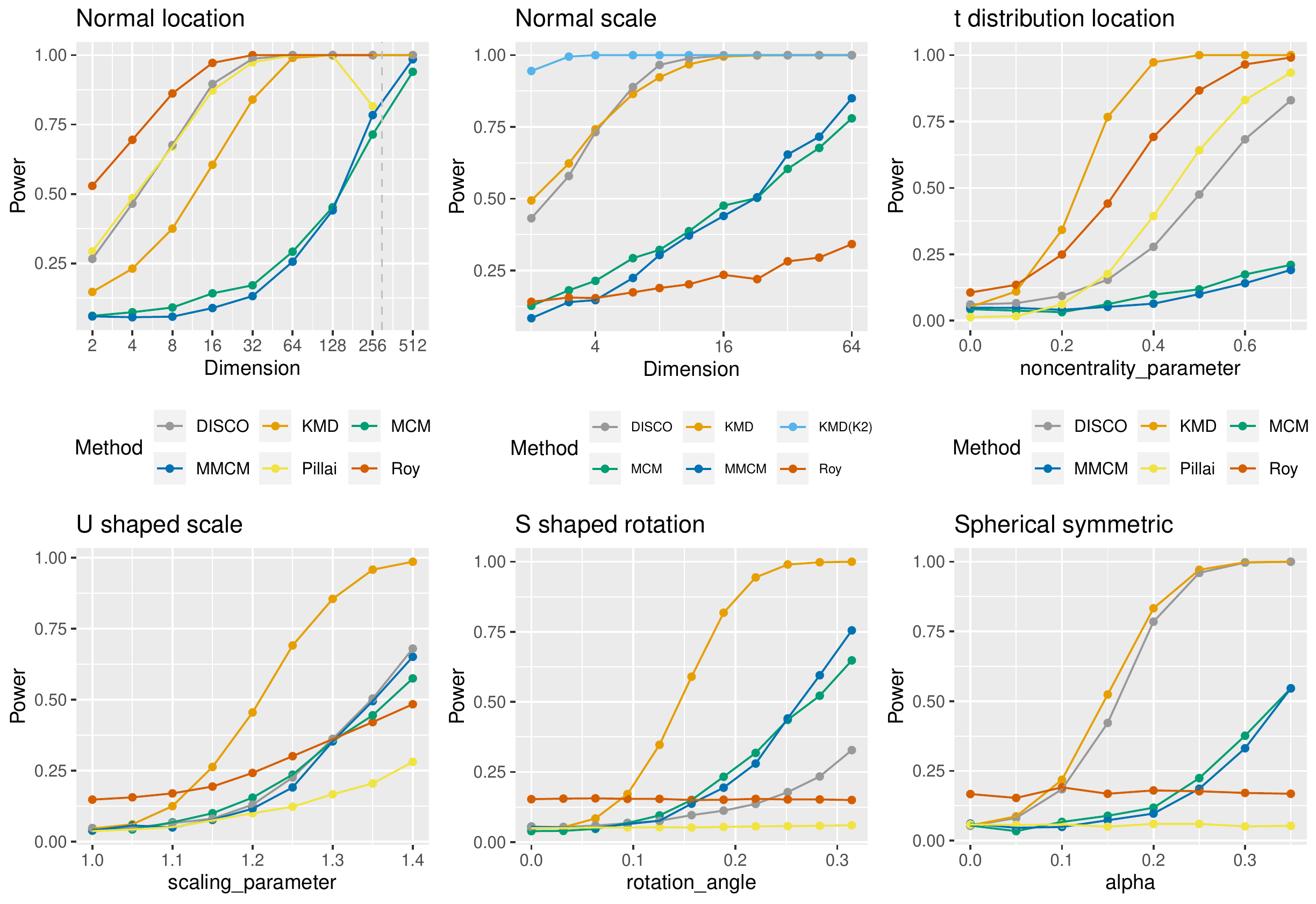}
    \caption{Empirical power curves for different methods under the various setings.}
    \label{tests}
\end{figure}
The rest of the examples go beyond normality, considering general distributions exhibiting different kinds of dissimilarities.
The U-shaped and S-shaped distributions are taken from \citet{hallin2020fully}.
We set $n_i=100$, for $i=1,2,3$.
All tests in the following (except \texttt{MANOVA()} built in \texttt{R}) use 500 random permutations to compute their $p$-values. The null hypothesis is rejected if the observed $p$-value is less than 0.05, and the power is estimated by averaging over 1000 replications.
We compare our methods with other parametric and nonparametric tests:
`Pillai' and `Roy' are the parametric tests implemented by \texttt{MANOVA()} in \texttt{R}.
`MCM', `MMCM' are the nonparametric test statistics proposed by \citet{mukherjee2020distribution}.
`KMD' is our method with $k$-NN graph (see~\eqref{eq:Eta-Hat}).
`DISCO' \citep{DISCO2010} is a nonparametric generalization of ANOVA using the energy statistics between samples, implemented as a permutation test in the \texttt{R} package \texttt{energy} \citep{Rpackenergy} with default index parameter $\alpha =1$. 

It can be seen from the top left plot of Figure~\ref{tests} that DISCO achieves similar power as the parametric method `Pillai' in the normal location problem. For this setting, the power curve of our test `KMD' is sandwiched between these two methods and the nonparametric methods MCM and MMCM 
(which have the attractive property of being fully distribution-free in finite samples but also at a cost of a lower power compared to other methods). Roy's largest root test is known to be powerful in this normal location example~\cite[Section 11.11.1]{denis2021applied}.
When the dimension is larger than the sample size, the error variance matrix is singular and the parametric tests like \texttt{MANOVA()} cannot be applied.


For the normal scale problem in Figure~\ref{tests} we illustrate that the choice of an appropriate kernel for our method can lead to improved performance; see `KMD(K2)'.
For this setting, the observations coming from the distributions with the larger scale parameter lie on the outer layers of the data cloud. When the dimension $d$ is large, although the observations generated with the smallest scale parameter almost always have their nearest neighbors coming from the same distribution,
the points in the outer layers may easily find themselves `closer' to points in the inner layer (generated from the distribution with the smaller scales). 
This can lead to low power for the `KMD' method with the discrete kernel. However, if it is a priori known that $P_1$ may lie in the inner layer, then we can assign $K(1,1)$ a larger value, which in turn would imply that $\hat \eta$ would be large as it would be dominated by the contributions from the observations from $P_1$ (which have their nearest neighbors coming from the same distribution). Figure~\ref{tests} shows that when using the kernel $K_2$, defined as $K_2(1,1) = 10$, $K_2(2,2) = K_2(3,3) = 1$, and $K_2(i,j)=0$, for $i\neq j$, the power of the method is substantially increased.

Note that for the normal scale problem, both Roy's and Pillai's tests have low power. But our method still works due to its generality, and as it is consistent against any alternative.
For the other settings considered here, our KMD exhibits very good performance, achieving highest power  among the competing procedures.

\begin{remark}[Choice of $k$ for the $k$-NN graph]
When testing equality of $M$ distributions, we chose $k= 0.10n$ (in our simulations), for samples with up to a few hundred observations (as advocated in \citet{Petrie2016}). Note that when $\eta = 0$, $\hat \eta$ is unbiased and a larger $k$ reduces variance of $\hat \eta$. 
While, for estimating $\eta$ using its empirical version $\hat{\eta}$, a much smaller $k$ is recommended: often $k=1$ would work the best. A larger $k$ often produces a smaller $\hat{\eta}$ (e.g., $\hat{\eta}=0$ in the extreme case $k=n-1$), inducing a large `bias' especially when $\eta$ is large; see Appendix~\ref{sec:choose_k} for empirical evidence.
\end{remark}

\subsection{Measuring Multi-sample Dissimilarity with Real Data}
In the following we analyse a few real data examples. Here we use a 1-NN graph for constructing $\hat \eta$ and for testing. 
We first consider data sets from the UCI Machine Learning Repository~\citep{Dua:2019}. 
For each of these multi-sample data sets, every observation/instance has a label indicating which sample it came from. However, for these data sets, instances with different labels 
are logically or physically different from each other, and have already been shown to be different in \citet{Petrie2016}. 

To illustrate the usefulness of $\eta$ (and $\hat \eta$) as a measure of dissimilarity between distributions, we construct various $M$ sample scenarios by adapting these real data sets suitably, as described below. For each data set, we compare distributions in the three settings: (i) the distributions to be compared have the same label (so the null hypothesis~\eqref{eq:hypo} may hold here); (ii) two distributions, out of the 3 distributions considered,  have the same label; and (iii) all the three distributions to be compared have distinct labels. It is natural to expect that $\hat \eta$ would increase as we move from scenario (i) to (iii); see Table~\ref{tab:Real-Data} for our results.

\vspace{0.03in}
\noindent {\bf Amazon commerce reviews} \citep{liu2011application}:
Each observation/instance here is a 10000-dimensional vector representing usage of digits, punctuation, words, sentence length, word frequencies, etc., of reviews from one of 50 extremely active customers. The data set has in total $50$ labels --- corresponding to the 50 customers --- with 30 instances per label. For each instance, most variables equal zero while the remaining are integers ranging from one to a few dozen. As explained above we consider comparing: (i) the first 15 reviews from reviewer 4 with her last 15 reviews; (ii) the first 15 reviews from reviewer 4, the last 15 reviews from reviewer 4, and all 30 reviews from reviewer 8; (iii) the first 15 reviews from reviewer 4, the first 15 reviews from reviewer 12, and all 30 reviews from reviewer 8.

\vspace{0.03in}
\noindent {\bf Semeion Handwritten Digits} \citep{Tactile1994semeion}:
In this example we have 10 labels --- corresponding to the 10 digits --- and  each instance/observation is a 256-dimensional vector representing a written digit, with each coordinate being 0 or 1 depending on its underlying grayscale value. Consider comparing: (i) the first 81 instances of digit 6 with the last 80 instances of the same digit; (ii) the first 81 instances of digit 6, the last 80 instances of digit 6, and all instances of digit 8; and (iii) the first 81 instances of digit 6, the first 80 instances of digit 7, and all instances of digit 8.

\vspace{0.03in}
\noindent {\bf ISOLET} \citep{dietterich1994solving}:
Each instance is a 617-dimensional vector representing a spoken letter. 150 speakers spoke each letter in the English alphabet twice, so we have 52  observations from each speaker. Consider comparing: (i) the 7-th letter, i.e., `g', spoken by the first 75 speakers with `g' spoken by the last 75 speakers; (ii) `g' spoken by the first 75 speakers, `g' spoken by the last 75 speakers, and `t' spoken by all speakers; and (iii) `g' spoken by the first 75 speakers, `n' spoken by the first 75 speakers, and `t' spoken by all speakers.

\vspace{0.03in}
\noindent {\bf LRS} \citep{Dua:2019}:
Each instance in this data set is a 93-dimensional vector which describes the fluxes from astronomical objects. There are in total 10 labels. Consider comparing: (i) the first 96 instances with label 2 and the last 177 instances with label 2; (ii) the first 96 instances with label 2, the last 177 instances with label 2, and all the instances with label 1 (NA removed); and (iii) all 96 instances with label 4, the last 177 instances with label 2, and all the instances with label 1 (NA removed).

Table~\ref{tab:Real-Data} gives the value of $\hat \eta$ for the 3 different settings for each  data set discussed above (for simplicity we use the standard Euclidean distance to construct the $k$-NN graphs in these examples). 
It can be seen from the table that the test for~\eqref{eq:hypo} based on $\hat \eta$ (and implemented via a permutation test using 500 random permutations) is rejected (at level 0.05) for both settings (ii) and (iii) for all data sets. Moreover, our measure $\hat{\eta}$ suggests that, for each of the data sets, the distributions in setting (iii) are ``more different" than those in setting (ii), in the sense that $\hat{\eta}$ is closer to 1 for setting (iii). Note that for the data sets ISOLET and LRS, the permutation test based on $\hat \eta$ also reject setting (i) where the samples have the same label.
For the ISOLET data, this suggests that the way the first 75 speakers spoke the 7-th letter is different from the way the last 75 speakers spoke the same letter. A similar conclusion can be drawn for the LRS data set. If we randomly select half of the instances with label 2 in LRS data set to form the first sample, instead of the first half, then the $p$-value will typically not be significant. Note that in the ISOLET data set, although all the hypothesis tests yield the same $p$-value of $1/501$, $\hat \eta$ gives a more meaningful summary of the dissimilarities between the distributions and produces values that progressively increase from settings (i) to (iii).


\begin{table}
\centering
\caption{$\hat{\eta}$ and $p$-values (over 500 random permutations) for real data sets.}\label{tab:Real-Data}
\begin{tabular}{c|cc|cc|cc|}
        Settings & \multicolumn{2}{c}{(i) Same label}  &  \multicolumn{2}{c}{(ii) Mixture} &  \multicolumn{2}{c}{(iii) Different labels} \\
        (Illustration) & \multicolumn{2}{c}{ \colorbox{green!30}{$1$} \colorbox{green!30}{$2$}}  &  \multicolumn{2}{c}{ \colorbox{green!30}{$1$} \colorbox{green!30}{$2$} \colorbox{blue!35}{$3$}} &  \multicolumn{2}{c}{\colorbox{green!30}{$1$} \colorbox{yellow!100}{$2$} \colorbox{blue!35}{$3$}}\\
        \hline
Data set & $\hat{\eta}$ & $p$-value & $\hat{\eta}$ & $p$-value & $\hat{\eta}$ & $p$-value\\
Amazon & 0.033& 0.513& 0.423 &0.004 & 0.607 &  0.002\\
Semeion & 0.099 &  0.160 & 0.604 & 0.002 & 0.975 & 0.002\\
ISOLET & 0.382 & 0.002 & 0.638 & 0.002& 0.886 &  0.002\\
LRS & 0.168 & 0.014 & 0.508 &      0.002 & 0.868 &0.002\\
\end{tabular} 
\end{table}

Next we consider the two non-Euclidean examples introduced in Section~\ref{sec:introduction}.

\vspace{0.03in}
\noindent \textbf{Speech recognition}:
The ArabicDigits data set consists of a total of 8800 instances with 10 labels (numbers 0-9) from native Arabic speakers.
We use the popular DTW distance (as mentioned in the Introduction) between two multi-dimensional time series \citep{berndt1994using,ding2008querying, gorecki2015multivariate} to compute the 1-NN graph. 
The KMD between the 10 spoken digits has an incredibly high value of 0.9976, suggesting the 10 distributions have highly disjoint supports.
This may be the reason why a simple nearest neighbor classifier could achieve 99.8\% prediction accuracy \citep{gorecki2015multivariate}.
We also investigate, using $\hat \eta$, the validity of the following intuitive supposition:
how males speak the number `1' should be different from how they speak `2', but the difference should be smaller than that between how males speak `1' and how females speak `2'. To make the problem harder, we only work with the 13th MFCC (instead of the 13-dimensional time-series), and the corresponding two KMD estimates are 0.5687 and 0.6867 respectively, agreeing with our intuition. 

\vspace{0.05in}
\noindent \textbf{Sentiment analysis}:  The movie review data set from \citet{text2vec2020} consists of 5000 movie reviews, each with a binary sentiment label corresponding to whether the review is positive or negative.
We process the data by transferring texts to lower case, removing non-alphanumeric symbols, collapsing multiple spaces,
and using vocabulary-based vectorization with the vocabulary pruned to have minimum number of occurrences over all documents being 5,
and maximum proportion of documents containing a term in the vocabulary being 10\%. This results in a 5000 $\times$ 12644 sparse {\it document-term matrix} \citep{text2vec2020}, and {\it Jaccard distance} \citep{Jaccard1912} is used as the metric between documents. The data processing is implemented using the \texttt{R} package \texttt{text2vec} \citep{text2vec2020}. The estimated KMD between positive reviews and negative reviews is 0.445, suggesting some overlap in the supports of the two distributions.
Though a simple linear support vector machine classifier based on the document-term matrix could provide 82\% baseline accuracy (evaluated from 10-fold cross-validation),
many later refined approaches, some of which were specifically designed for sentiment analysis on such data sets, cannot improve this accuracy to 90\% \citep{maas2011learning}. This agrees with the observed value of KMD which is much lower than in the previous example.

\numberwithin{equation}{section}
\numberwithin{defn}{section}
\numberwithin{remark}{section}
\numberwithin{lemma}{section}
\numberwithin{proposition}{section}
\numberwithin{theorem}{section}
\appendix 

\section{Appendix}\label{sec:general_discussion}
This Appendix will be organized as follows:
In Appendix~\ref{sec:general_discussion}, we provide some general discussions that were deferred from the main paper.
An analysis of the asymptotic behavior of $\hat{\eta}$ under alternatives is given in Appendix~\ref{sec:normal_alternative}.
We show that both under a fixed alternative and under shrinking alternatives converging to the null,
$\hat{\eta}$ has an asymptotic normal distribution after proper centering (see Theorems~\ref{thm:normal_alternative} and \ref{thm:shrinking_alternative}).
Using the CLT under shrinking alternatives, we provide a complete characterization of the local power of our method for $\Z=\R^d$ (see Theorem~\ref{thm:detection_threshold}).
In particular, this provides the detection threshold of the test based on $\hat{\eta}$.
All proofs are given in Appendix~\ref{sec:all_proofs}. Further simulation results  that were mentioned in the main paper are given in Appendix~\ref{sec:further_simu}.

\subsection{Assumptions on the Geometric Graph}\label{sec:Assump-Graph}
Let $\tilde{Z}_1,\tilde{Z}_2\ldots$ be i.i.d.~from the mixture $\sum_{i=1}^M \pi_i P_i$ on $\Z$. Let $\tilde{\mathcal{G}}_n:= \mathcal{G}(\tilde{Z}_1,\ldots ,\tilde{Z}_n)$ be the geometric graph with vertex set $\{\tilde{Z}_1,\ldots,\tilde{Z}_n\}\subset\Z$. Let $\mathcal{E}(\mathcal{\tilde{G}}_n)$ denote the set of (directed/undirected) edges of $\tilde{\mathcal{G}}_n$, i.e., $(\tilde{Z}_i,\tilde{Z}_j)\in\mathcal{E}(\mathcal{\tilde{G}}_n)$ if and only if there is an edge from $\tilde{Z}_i$ to $\tilde{Z}_j$ in $\tilde{\mathcal{G}}_n$, and $\tilde{d}_i$ denotes the out-degree of $\tilde{Z}_i$ in $\tilde{\mathcal{G}}_n$. To be specific, $\tilde{d}_i:=\sum_{j:(\tilde{Z}_i,\tilde{Z}_j)\in\mathcal{E}(\mathcal{\tilde{G}}_n)} 1$. We assume the following conditions on $\tilde{\mathcal{G}}_n$ (as in \citet{deb2020kernel}):
\begin{assump}\label{assump:conv_nn}
Given the graph $\tilde{\mathcal{G}}_n$, let $N(1),\ldots, N(n)$ be independent random variables where $N(i)$ is a uniformly sampled index from among the (out-)neighbors of $\tilde{Z}_i$ in $\tilde{\mathcal{G}}_n$. Then
$$\rho_\Z(\tilde{Z}_1,\tilde{Z}_{N(1)})\overset{p}{\to}0\quad {\rm as\ }n\to\infty.$$
\end{assump}
\begin{assump}\label{assump:degree}
    Assume that there exists a deterministic positive sequence $r_n\geq 1$ (may or may not be bounded), such that almost surely:
    $$\min_{1\leq i\leq n} \tilde{d}_i\geq r_n.$$
    Let $\tilde{\mathcal{G}}_{n,i}$ denote the graph obtained from $\tilde{\mathcal{G}}_n$ by replacing $\tilde{Z}_i$ with an i.i.d.~random element $\tilde{Z}_i'$. Assume that there exists a deterministic positive sequence $q_n$ (may or may not be bounded), such that
    $$\max_{1\leq i\leq n}\max\{|\mathcal{E}(\tilde{\mathcal{G}}_n)\setminus \mathcal{E}(\tilde{\mathcal{G}}_{n,i})|,|\mathcal{E}(\tilde{\mathcal{G}}_{n,i})\setminus \mathcal{E}(\tilde{\mathcal{G}}_{n})|\}\leq q_n\ \ a.s.\qquad {\rm and}\qquad \frac{q_n}{r_n} = O(1).$$
\end{assump}
\begin{assump}\label{assump:degupbd}
There exists a deterministic sequence $\{t_n\}_{n\ge 1}$ (may or may not be bounded) such that the vertex degree (including both in- and out-degrees for directed graphs, i.e., $\#\{ j:(\tilde{Z}_i,\tilde{Z}_j)\ {\rm or}\ (\tilde{Z}_j,\tilde{Z}_i)\in\mathcal{E}(\tilde{\mathcal{G}}_n)\}$) of every point $\tilde{Z}_i$ (for $i=1,\ldots, n$) is bounded by $t_n$, and
		$\frac{t_n}{r_n}=O(1)$.
\end{assump}
Assumption \ref{assump:conv_nn} formalizes our intuition that the presence of an edge between two points implies that the two points are close.
Assumption \ref{assump:degree} states that the graph is `local' in the sense that replacing one random point will not change too many edges.
Assumption \ref{assump:degupbd} requires that the degree of each vertex should be of the same order.
See~\citet[Section 3]{deb2020kernel} for a detailed discussion on these assumptions.

\subsection{Computational Complexity and Implementation Schemes}\label{sec:implement_scheme}
When a Euclidean $k$-NN graph is used, the computation complexity of $\hat{\eta}$ and $\frac{\hat{\eta}}{\sqrt{{\rm Var}(\hat{\eta}|\mathcal{F}_n)}}$ is $O(kn\log n)$, which is {\it near linear}, when $k$ is bounded. This is due to the fact that 
the Euclidean $k$-NN graph can be computed in $O(kn\log n)$ time (for example, using the k-d tree; see~\cite{bentley1975multidimensional}).

The computation of the $k$-NN graph is well implemented in many computational softwares \citep{scikit-learn,RANN2019FNN}. The computation of $\tilde{a},\tilde{b},\tilde{c}$ (see \eqref{eq:tilde_a_b_c}) and $\tilde{g}_2$ is described below. Observe that
$$\begin{aligned}
\tilde{a}& = \frac{1}{n(n-1)}\left(\sum_{i=1}^M n_i(n_i-1)K^2(i,i) + \sum_{i\neq j}n_in_jK^2(i,j)\right)\\
&=\frac{1}{n(n-1)}\left( \sum_{i,j=1}^M n_in_jK^2(i,j) - \sum_{i=1}^M n_i K^2(i,i) \right)
\end{aligned}$$ which can be computed in $O(M^2) = O(1)$ time since $M$ is a constant. Further, note that,
$$\begin{aligned}
\tilde{b}&=\frac{1}{n(n-1)(n-2)}\Big(  \sum_{i=1}^M K(i,i)K(i,i)n_i(n_i-1)(n_i-2) \\
&\qquad +{\sum}' K(i,i)K(i,j)n_i(n_i-1)n_j +
{\sum}' K(i,j)K(i,i)n_in_j(n_i-1) \\
&\qquad + {\sum}' K(i,j)K(i,j)n_in_j(n_j-1) +
{\sum}' K(i,j)K(i,l)n_in_jn_l\Big)\\
&=\frac{1}{n(n-1)(n-2)}\Big(  
\sum_{i,j,l=1}^M K(i,j)K(i,l)n_in_jn_l  -  \sum_{i=1}^M K(i,i)^2 n_i(3n_i-2) \\
&\qquad - 2\ {\sum}' K(i,i)K(i,j)n_in_j - {\sum}' K(i,j)^2 n_in_j \Big)\\
&=\frac{1}{n(n-1)(n-2)}\Big(  
\sum_{i,j,l=1}^M K(i,j)K(i,l)n_in_jn_l  +  2\sum_{i=1}^M K(i,i)^2 n_i \\
&\qquad - 2\sum_{i,j=1}^M K(i,i)K(i,j)n_in_j - \sum_{i,j=1}^M K(i,j)^2 n_in_j \Big)\\
&=\frac{1}{n(n-1)(n-2)}\Big(  
\sum_{i,j,l=1}^M K(i,j)K(i,l)n_in_jn_l  +  \sum_{i=1}^M K(i,i)^2 n_i \\
&\qquad - 2\sum_{i,j=1}^M K(i,i)K(i,j)n_in_j - n(n-1)\tilde{a} \Big)\\
&=\frac{1}{n(n-1)(n-2)}\Big(  \vec{\textbf{n}}^\top [K(i,j)] {\rm diag}(\vec{\textbf{n}})[K(i,j)]  \vec{\textbf{n}} +  \sum_{i=1}^M K(i,i)^2 n_i\\
&\qquad- 2\sum_{i,j=1}^M K(i,i)K(i,j)n_in_j - n(n-1)\tilde{a} \Big),
\end{aligned}$$
where $ \vec{\textbf{n}} = (n_1,\ldots,n_M)^\top$, ${\rm diag}(\vec{\textbf{n}})$ is the diagonal matrix with $\vec{\textbf{n}}$ being in the diagonal, and $[K(i,j)]$ is the $M \times M$ matrix with the entry in the $i$-th row and $j$-th column being $K(i,j)$.
Hence, $\tilde{b}$ can be computed in $O(M^2)=O(1)$ time as well. Similarly,
$$\begin{aligned}
\tilde{c}&= \frac{\big( \sum_{i,j=1}^M n_in_jK(i,j) - \sum_{i=1}^M n_i K(i,i) \big)^2 - 4n(n-1)(n-2)\tilde{b} - 2n(n-1)\tilde{a}}{n(n-1)(n-2)(n-3)}
\end{aligned}$$
can be computed in $O(1)$ time. Recall that, 
$\tilde{g}_2=\frac{1}{n}\sum_{i,j=1}^n \frac{T^{\mathcal{G}_n}(i,j)}{d_id_j}=\frac{1}{n}\sum_{i=1}^n \frac{1}{d_i} + \frac{1}{n} \sum_{i\neq j}\frac{T^{\mathcal{G}_n}(i,j)}{d_id_j}$.
To compute $\sum_{i\neq j}\frac{T^{\mathcal{G}_n}(i,j)}{d_id_j}$, for each $Z_i$, we find its in-neighbors, i.e., $\{Z_j:(Z_j,Z_i)\in\emgn\}=\{Z_{i_1},\ldots,Z_{i_{k'}}\}$.
Let $S_i := \left(\frac{1}{d_{i_1}}+\ldots+\frac{1}{d_{i_{k'}}} \right)^2 - \sum_{j=1}^{k'} \frac{1}{d_{i_j}^2}$.
Then $\sum_{i\neq j}\frac{T^{\mathcal{G}_n}(i,j)}{d_id_j} = \sum_{i=1}^n S_i$.
The computational complexity for $\tilde{g}_2$ is $O(kn)$ since $k'\leq kC(d) $ for some constant $C(d)$ depending on the dimension $d$ of $\Z$ \citep[Lemma 8.4]{yukich1998Euclidean}. The computations of other terms are straight forward.

\subsection{Interpretation of $\hat{\eta}$ as $k$-NN Cross-Validated Accuracy}\label{sec:CVinterpretation}
Here we show that $\hat{\eta}$ is linearly related to the leave-one-out cross-validation accuracy of a $k$-NN classifier.

A classifier uses existing data $\{(\Delta_i,Z_i)\}_{i\in I}$, where $I\subset\{1,\ldots,n\}$ is an index set, to make prediction for the label $\Delta_j$ of a new observation $Z_j$, $j\notin I$.
The leave-one-out cross-validated accuracy (abbreviated as ``accuracy" in the following) of a possibly random classifier $f:\Z \to\{1,\ldots,M\}$ is defined as:
$$f\ {\rm Accuracy}:=\frac{1}{n}\sum_{j=1}^n \mathbb{P}_f \Big(\Delta_j =f\left(Z_j | \{(\Delta_i,Z_i)\}_{i\in \{1,\ldots,n\}\backslash \{j\}}\right)  \Big),$$
where $f\left(Z_j | \{(\Delta_i,Z_i)\}_{i\in \{1,\ldots,n\}\backslash \{j\}}\right)$ is the prediction for $\Delta_j$ using $Z_j$ given by $f$ learnt on data $\{(\Delta_i,Z_i)\}_{i\in \{1,\ldots,n\}\backslash \{j\}}$, and $ \mathbb{P}_f$ means averaging over possible randomness in $f$.

Consider the accuracy of:
(i) $k$-NN classifier --- the label of a data point is predicted by the label of a random $k$-NN of that data point;
(ii) random guess classifier --- the label of a data point is predicted by a random guess according to the proportion of different labels in the existing data.

When $K(x,y)=I(x=y)$ and a $k$-NN graph is used to define $\hat \eta$ via~\eqref{eq:Eta-Hat}, we have
$$\frac{1}{n} \sum \limits_{i=1}^n \frac{1}{d_i} \sum \limits_{j:(Z_i,Z_j) \in \emgn} K(\Delta_i,\Delta_j)
=\frac{1}{n}\sum_{i=1}^n \frac{1}{k} \sum_{j:(Z_i,Z_j) \in \emgn} I(\Delta_i=\Delta_j)=k\text{-}{\rm NN\ Accuracy}$$ and 
$$\frac{1}{n(n-1)}  \sum \limits_{i \ne j} K(\Delta_i,\Delta_j)=\frac{1}{n}\sum_{i=1}^n \frac{ \sum_{j:j\neq i} I(\Delta_i=\Delta_j)}{n-1} = {\rm Random\ Guess\ Accuracy}.$$
Hence $\eta$ in~\eqref{eq:Eta-Hat} reduces to

\begin{equation}\label{eq:inter_crossknn}
\hat{\eta}=\frac{(k\text{-}{\rm NN\ Accuracy}) - ({\rm Random\ Guess\ Accuracy})}{1-({\rm Random\ Guess\ Accuracy})}.
\end{equation}
Intuitively, if the $M$ distributions are concentrated on different regions of the space $\Z$, then the $k$-NN accuracy is close to 1 and thus $\hat{\eta}$ will be close to 1.
On the other hand, if the $M$ distributions are the same, then the $k$-NN accuracy will be similar to the random guess accuracy and thus $\hat{\eta}$ will be close to 0. Moreover, $\hat{\eta}$ grows linearly with $k$-NN accuracy, which provides a simple and intuitive interpretation of our measure $\hat{\eta}$.

\section{Asymptotic Behavior of $\hat{\eta}$ under Alternatives}\label{sec:normal_alternative}

In this section, we will first establish CLTs for $\hat{\eta}$ under a fixed alternative and under certain shrinking alternatives converging to the null. Using these CLTs, we characterize the asymptotic power of the test based on $\hat \eta$. In particular, this provides the detection threshold of the test based on $\hat{\eta}$. We will focus on $\Z =\R^d$ ($d \ge 1$) in this section, and assume that $P_i$ has a density $f_i$ w.r.t.~the Lebesgue measure, for all $i=1,\ldots,M$.

The asymptotic distribution of geometric graph-based statistics under alternatives and their detection thresholds were not available until the recent work of~\citet{BB20detection},
where the analysis was carried out under a Poissonized setting.
The spatial independence of the Poisson process facilitates the computation of the asymptotic variance of $\hat \eta$, and also helps establish the CLTs for $\hat{\eta}$ using Stein's method \citep{chen2004normal}.
We will also consider this Poissonized framework in this section.
Although de-Poissonization techniques are available in the literature \citep{Penrose2003graphSUPP,Penrose2007measures}, the de-Poissonized version of our theorem (Theorem \ref{thm:normal_alternative}) unfortunately does not reduce to our original setting in the main paper where the sample size $n_i$'s are nonrandom. 


In the Poissonized framework, instead of assuming $\frac{n_i}{n}\to\pi_i$, it is assumed that $n_i\sim {\rm Poisson}(N_i)$, for $i=1,\ldots,M$.
Write $N := \sum_{i=1}^M N_i$ and assume $\frac{N_i}{N}\to\pi_i>0$
as $N\to\infty$.
An equivalent characterization is that we first decide to draw in total $n\sim {\rm Poisson}(N)$ data points, and for each data point, with probability $\frac{N_i}{N}$, we draw an observation from the $i$-th population, for $i=1,\ldots, M$.
In the following, we provide CLTs for $\hat{\eta}$ under fixed and shrinking alternatives, under the Poissonized setting when we let $N\to\infty$.
\subsection{Asymptotic Normality under a General Fixed Alternative}\label{subsec:general_fixed_alt}

In the following, we briefly sketch how a CLT for $\hat \eta$, under a fixed alternative, can be obtained after proper centering; see Theorem~\ref{thm:normal_alternative} for the precise statement.
Since the Poisson distribution is tightly concentrated around its mean: ${\rm Var}\left(\frac{n_i}{N_i}\right) = \frac{1}{N_i}\to 0$ as $N_i \to \infty$. Thus, we have $\frac{n_i}{N_i}\overset{p}{\to}\E\left[\frac{n_i}{N_i}\right]=1$.
Hence the denominator of $\hat{\eta}$ in \eqref{eq:Eta-Hat} converges in probability to $\sum_{i=1}^M \pi_i K(i,i) - \sum_{i,j=1}^M \pi_i\pi_j K(i,j) >0$. So our analysis focuses on the asymptotic behavior of the numerator in \eqref{eq:Eta-Hat}:
\begin{equation*}
\tilde{H}_n:=\frac{1}{n} \sum \limits_{i=1}^n \frac{1}{d_i} \sum \limits_{j:(i,j) \in \emgn} K(\Delta_i,\Delta_j)-\frac{1}{n(n-1)}  \sum \limits_{i \ne j} K(\Delta_i,\Delta_j).
\end{equation*}
When the above quantity cannot be well-defined, i.e., $n=0$ or $1$, or $n\leq k$ when a $k$-NN graph is used, we set $\tilde{H}_n=\hat{\eta}=0$.
Using $U$-statistics projection theory \citep[Lemma D.4]{deb2020kernel}, it can be shown that $n\tilde{H}_n / \sqrt{N}$ has the same asymptotic distribution as $H_n$ defined as
$$H_n :=\frac{1}{\sqrt{N}} \sum \limits_{i=1}^n \frac{1}{d_i} \sum \limits_{j:(Z_i,Z_j) \in \emgn} K(\Delta_i,\Delta_j)-\frac{1}{\sqrt{N}}\sum_{i=1}^n  g_N(\Delta_i),$$
which has a non-degenerate asymptotic normal distribution after centering by its mean.
Here, $$g_N (\Delta_i) := 2\sum_{p=1}^M \frac{N_p}{N}K(\Delta_i,p) - \sum_{p,q=1}^M \frac{N_p N_q}{N^2} K(p,q).$$

Similar to the CLT result under ${\rm H}_0$ (see Theorem~\ref{thm:asympnull}), if we let $\mathcal{F}_n:=\sigma(n,Z_1,\ldots,Z_n)$, the $\sigma$-algebra generated by the unlabelled data and the number of total observations,
then $H_n - \mathbb{E}(H_n|\mathcal{F}_n)$ given $\mathcal{F}_n$ converges in distribution to a normal limit $N(0,\kappa_1^2)$ (see Theorem~\ref{thm:normal_alternative} for the exact expression for $\kappa_1^2$). We can also show that the unconditional distribution of $\mathbb{E}(H_n|\mathcal{F}_n) - \E(H_n)$ converges to another normal limit $N(0,\kappa_2^2)$; see Theorem~\ref{thm:normal_alternative} for the exact expression for $\kappa_2^2$. Finally, a simple argument using characteristic functions shows that $H_n-\E (H_n) \overset{d}{\to} N(0,\kappa_1^2+\kappa_2^2)$.

If ${\rm H}_0$ holds, then it can be shown that $\mathbb{E}(H_n|\mathcal{F}_n) = \mathbb{E}(H_n) = 0$, and $\kappa_1^2$ is exactly the distribution-free variance derived in Theorem~\ref{thm:DistFree}, and $\kappa_2^2 = 0$.
However, in a general situation, $\mathbb{E}(H_n|\mathcal{F}_n)$ is a function of $\{Z_1,\ldots,Z_n\}$, which follows a non-homogeneous Poisson process\footnote{For a non-homogeneous Poisson process with intensity function $f$ (denoted by $\mathcal{P}_f$), the number of points in a set $A\subset\R^d$ follows a
${\rm Poisson}(\int_A f(x) {\rm d}x)$ distribution, and the number of points in disjoint sets are independent.} $\mathcal{P}_{N \phi_N}$ under the Poissonized setting, with
\begin{equation}\label{eq:Phi_N-Phi}
\phi_N(z) :=  \sum_{i=1}^M \frac{N_i}{N} f_i(z) \qquad {\rm and}\qquad  \phi(z):=\sum_{i=1}^M \pi_i f_i(z), \quad \mbox{for }\; z \in \Z,
\end{equation}
being the marginal density of $Z$ given the total number of observations, and its limiting value, respectively.
To show the convergence of the variance of $\mathbb{E}(H_n|\mathcal{F}_n)$ and establish a CLT for $\mathbb{E}(H_n|\mathcal{F}_n) - \E(H_n)$ using Stein's method,
we will need to make another common assumption --- the {\it power-law stabilization} \citep{Penrose2005normal,Penrose2007measures,BB20detection} --- on the non-homogeneous Poisson process $\mathcal{P}_{N \phi_N}$ (in addition to the stabilization on the homogeneous Poisson process $\mathcal{P}_\lambda$ defined in Section~\ref{sec:asymp_dist_free}),
which will be introduced below.

Recall our setup: $\mathcal{G}$ is a geometric graph on $\Z$ and $\X$ is a set of points in $\Z$. $\mathcal{E}\left(x,\mathcal{G}(\X)\right)$ is the set of
edges in $\mathcal{G}(\X)$ that are incident to $x\in\X$.
Let $A$ be a set with $\phi$-probability 1, e.g., $A={\rm supp}(\phi)$, the support of $\phi$. 
Fix $x\in\Z$. A
{\it  radius of stabilization} $R\equiv R(N, x)$ at $x$ exists if the points in $A\backslash B(x,N^{-1/d} R)$ cannot impact the neighborhood of $x$ in $\mathcal{G}\left(\mathcal{P}_{N\phi_N}^x\right)$ where $\mathcal{P}_{N\phi_N}^x:=\mathcal{P}_{N\phi_N}\cup\{x\}$, i.e.,
\begin{equation}\label{eq:R}
\mathcal{E}\left(x,\mathcal{G}\left([\mathcal{P}_{N\phi_N}^x \cap B(x,N^{-1/d} R)]\cup\X \right) \right)=\mathcal{E}\left(x,\mathcal{G}\left(\mathcal{P}_{N\phi_N}^x \cap B(x,N^{-1/d} R) \right)\right),
\end{equation}
for all finite $\X\subset A\backslash B(x,N^{-1/d} R)$. Such a finite $R$ always exists when $A$ is bounded.
For certain graphs such as a $k$-NN graph, the {\it power-law stabilization} states that $R$ is bounded in probability under mild assumptions, and the tail probability decays sufficiently fast, uniformly over $N$ and $x$. To define it formally, let
\begin{equation}\label{eq:power_law_graph}\tau(t):=\sup_{N\geq 1,x\in A}\mathbb{P}[R(N,x) > t], \qquad \mbox{for}\;\;t >0.\end{equation}
Then, $\mathcal{G}$ is said to be:
\begin{enumerate}
\item {\it power-law stabilizing} of order $q$ with respect to $\phi_N$ if $\sup_{t\geq 1} t^q \tau(t) < \infty$,
\item {\it exponentially stabilizing} with respect to $\phi_N$ if $\limsup_{t\to\infty} t^{-1}\log \tau(t) <0$.
\end{enumerate}
Note that if $\mathcal{G}$ is exponentially stabilizing, then it is power-law stabilizing for all $q>0$.
It is known that a $k$-NN graph is exponentially stabilizing under suitable conditions (see Proposition~\ref{prop:knnExpStab} below and Appendix~\ref{pf:knnExpStab} for its proof).
Other exponentially stabilizing graphs include the
Voronoi and Delaunay graphs \citep{Penrose2007measures,BB20detection}.
However, there is no similar tail bounds known for MST, as far as we are aware.

\begin{proposition}[{\citet[Section 6.3]{Penrose2007measures}}]\label{prop:knnExpStab}
Assume $\phi_N$ is bounded away from 0 on its support which is assumed to be convex. Then the $k$-NN graph (with $k$ fixed, either directed or undirected) is exponentially stabilizing.
\end{proposition}


Assuming that the graph $\mathcal{G}$ is power-law stabilizing of sufficient order w.r.t.~our data generating process $\phi_N$, the asymptotic normality of $\hat{\eta}$ can be established, as shown below.

\begin{theorem}[CLT under a fixed alternative]\label{thm:normal_alternative}
Suppose that we have the Poissonized setting where we have $n_i\sim {\rm Poisson}(N_i)$ samples from the $i$-th distribution, for $i=1,\ldots, M$. Suppose the following assumptions hold:
\begin{enumerate}
\item $\sqrt{N}\left(\frac{N_i}{N} - \pi_i\right)\to 0$ as $N\to\infty$, for $i=1,\ldots,M$, where $N = \sum_{i=1}^M N_i$.

\item The $i$-th distribution has a Lebesgue density $f_i$ on $\R^d$ which is Lebesgue almost everywhere continuous, for $i=1,\ldots,M$.

\item $\phi:=\sum_{i=1}^M \pi_i f_i$ is bounded above and has a bounded support.

\item $\mathcal{G}$ is translation and scale invariant, stabilizing on $\mathcal{P}_\lambda$ for some $\lambda>0$.

\item $\mathcal{G}$ is power-law stabilizing w.r.t.~$\phi_N$ (see~\eqref{eq:Phi_N-Phi}) with order $q_0>16d$,
with the corresponding full probability set $A$ and radius of stabilization $R$ as defined in \eqref{eq:R}. 
\item The degree bound (including in- and out-degrees) for $\mathcal{G}(\mathcal{P}_{N \phi_N})$ is $o_p(N^{1/40})$.
\end{enumerate}
Then with $n:=\sum_{i=1}^M n_i$,
\begin{equation}\label{eq:Fixed-Alt}
\frac{1}{\sqrt{n}}\left(  n\tilde{H}_n - \E[n\tilde{H}_n] \right)\overset{d}{\to}N(0,\kappa_1^2+\kappa_2^2),
\end{equation}
where
$$\begin{aligned}
\kappa_1^2&:=(g_1+g_3)\int {\rm Var}[K(\tilde{\Delta},\tilde{\Delta}')|\tilde{Z}=z]\phi(z){\rm d}z\\
&\qquad +(3-2g_1 - 2g_3 + g_2)\int {\rm Cov}(K(\tilde{\Delta},\tilde{\Delta}'),K(\tilde{\Delta},\tilde{\Delta}'')|\tilde{Z}=z) \phi(z){\rm d}z\\
&\qquad -4\int {\rm Cov}(K(\tilde{\Delta},\tilde{\Delta}'),g(\tilde{\Delta})|\tilde{Z}=z) \phi(z){\rm d}z +\E\left[{\rm Var}(g(\tilde{\Delta})|\tilde{Z}) \right],\\
\kappa_2^2&:=\int \left(\sum_{i,j=1}^M K(i,j)\frac{\pi_if_i(z)}{\phi(z)}\frac{\pi_jf_j(z)}{\phi(z)} -  \sum_{i=1}^M g(i)\frac{\pi_if_i(z)}{\phi(z)} \right)^2\phi(z){\rm d}z\\
&\;=\int \left(\E[K(\tilde{\Delta},\tilde{\Delta}')|\tilde{Z}=z] - \E [g(\tilde{\Delta})|\tilde{Z}=z] \right)^2\phi(z){\rm d}z.
\end{aligned}$$
Note that in the above displays, $\tilde{\Delta},\tilde{\Delta}',\tilde{\Delta}''$ are drawn independently from the conditional distribution $\tilde{\Delta} |\tilde{Z}=z$ (recall the definition of $\tilde{\Delta},\tilde{Z}$ in Section~\ref{subsec:def_eta}); $g (\Delta):=2\sum_{i=1}^M \pi_i K(\Delta,i) - \sum_{i,j=1}^M \pi_i \pi_j K(i,j)$, and $g_1,g_2,g_3$ are as defined in Theorem~\ref{thm:DistFree}.
\end{theorem}

Observe that $\kappa_1$ and $\kappa_2$ are not distribution-free in general, unless $\tilde{\Delta}$ is independent of $\tilde{Z}$, i.e., ${\rm H}_0$ holds.
The proof of the above result is given in Appendix~\ref{pf:normal_alternative}, applying similar techniques as in \citet{BB20detection}.
Note that \citet{BB20detection} has applied the results in \citet{Penrose2007measures}, but some of the assumptions in \citet{Penrose2007measures}
such as almost everywhere continuity and boundedness of densities are not explicitly stated in \citet{BB20detection}.
Here, due to the presence of a general kernel and the U-statistic term $\sum_{i\neq j} K(\Delta_i,\Delta_j)$ we cannot directly apply the results in \citet{Penrose2007measures}, and hence we provide a complete proof under the assumptions stated above.


\begin{remark}[Non-degeneracy with characteristic kernel]
Suppose $K$ is characteristic.
When $\tilde{\Delta}$ is independent of $\tilde{Z}$, then $\kappa_2^2=0$, and $\kappa_1^2$ reduces to the distribution-free null variance in Theorem~\ref{thm:DistFree}, which is strictly positive. When $\tilde{\Delta}$ is not independent of $\tilde{Z}$, $\kappa_2^2$ is strictly positive since the integrand is the fourth power of
the maximum mean discrepancy\footnote{
MMD is a distance (provided that the kernel is characteristic) on the space of probability measures defined through the kernel $K(\cdot,\cdot)$ \citep{Gretton12}. For two probabilities $P$ and $Q$, let independent samples $\Delta_1^P,\Delta_2^P$ i.i.d.~follow $P$, and $\Delta_1^Q,\Delta_2^Q$ i.i.d.~follow $Q$. A definition of
MMD between $P,Q$ is given by:
${\rm MMD}^2(P,Q) := \E K(\Delta_1^P,\Delta_2^P) + \E K(\Delta_1^Q,\Delta_2^Q) - 2\E K(\Delta_1^P,\Delta_1^Q).$
} (MMD) between the conditional distribution $\tilde{\Delta}|\tilde{Z}=z$ and the unconditional distribution of $\tilde{\Delta}$
(note that with a characteristic kernel, the MMD is positive whenever the two distributions are not equal).
Hence the asymptotic variance  in~\eqref{eq:Fixed-Alt} is always positive.
\end{remark}

\subsection{A CLT under Shrinking Alternatives}
In the previous subsection, a CLT was derived under a fixed alternative.
In order to study the detection threshold of our statistic $\hat{\eta}$, a CLT under shrinking alternatives is needed, which will be studied in this subsection. Note that \citet{BB20detection} only provided a CLT under a fixed alternative (cf., Theorem~\ref{thm:normal_alternative}); the CLT under shrinking alternatives was not explicitly formulated, but was assumed to hold instead. As will be stated in the following theorem, it turns out that a uniform version of the conditions for the CLT under fixed alternatives (Theorem~\ref{thm:normal_alternative}) can guarantee a CLT under shrinking alternatives.

We assume for each $N$, the $M$ distributions have densities $f_1^N,\ldots,f_M^N$ on $\Z\subset \R^d$ respectively.
We write a $N$ at the top right of each $f_i$ to emphasize that the densities change with the sample size $N$.

Assume $f_i^N(z) \overset{}{\longrightarrow}f(z)$ as $N\to\infty$ for a.e.~$z \in \Z$ (under the Lebesgue measure), for $i=1,\ldots,M$.
Let
$$\phi^N_N (z) := \sum_{i=1}^M \frac{N_i}{N} f_i^N(z)\qquad {\rm and} \qquad \phi^N (z) := \sum_{i=1}^M \pi_i f_i^N(z),\quad {\rm for }\;\; z \in \Z.$$
Then $\phi^N_N (z)$ is the marginal density of $Z$ given the total number of observations.
To extend the CLT result to shrinking alternatives, we need to also extend the notion of stabilization for the non-homogeneous Poisson process $\mathcal{P}_{N\phi^N_N}$ below. We note that the distribution of the set of points $\{Z_1,\ldots,Z_n\}$ is exactly a non-homogeneous Poisson process $\mathcal{P}_{N\phi^N_N}$ \citep[Proposition 1.5]{Penrose2003graph}.

Consider the notation introduced at the start of Appendix~\ref{sec:normal_alternative} and recall the Poissonized setting.
Let $A^N$ be a set with $\phi^N$-probability 1. Fix $x\in\Z$. 
{\it A radius of stabilization} $R\equiv R(N, x)$ at $x$ exists if the points in $A^N \backslash B(x,N^{-1/d} R)$ cannot impact the neighborhood of $x$ in $\mathcal{P}_{N\phi^N_N}^x:=\mathcal{P}_{N\phi^N_N}\cup\{x\}$, i.e.,
\begin{equation}\label{eq:R_extend}
\mathcal{E}\left(x,\mathcal{G}\left([\mathcal{P}_{N\phi^N_N}^x\cap B(x,N^{-1/d} R)]\cup\X  \right)\right)=\mathcal{E}\left(x,\mathcal{G}\left(\mathcal{P}_{N\phi^N_N}\cap B(x,N^{-1/d} R) \right)\right),
\end{equation}
for all finite $\X\subset A^N\backslash B(x,N^{-1/d} R)$.
With
$$\tau(t):=\sup_{N\geq 1,x\in A^N}\mathbb{P}[R(N,x) > t],\qquad \mbox{for }\;\;t>0,$$
$\mathcal{G}$ is said to be:
\begin{enumerate}
\item {\it power-law stabilizing} of order $q$ with respect to $\phi^N_N$ if $\sup_{t\geq 1} t^q \tau(t) < \infty$,
\item {\it exponentially stabilizing} with respect to $\phi^N_N$ if $\limsup_{t\to\infty} t^{-1}\log \tau(t) <0$.
\end{enumerate}
Almost the same proof of Proposition~\ref{prop:knnExpStab} shows that if $\phi^N_N$ has convex support and $\phi^N_N \geq c$ on ${\rm supp}(\phi^N_N)$ for some $c>0$ and all $N$, then $k$-NN graph is exponentially stabilizing with respect to $\phi^N_N$, and thus power-law stabilizing of any order $q>0$. The following theorem (proved in Appendix~\ref{pf:shrinking_alternative}) formally states the CLT under shrinking alternatives.
\begin{theorem}[CLT under shrinking alternatives]\label{thm:shrinking_alternative}
Assume the Poissonized setting where we have $n_i\sim {\rm Poisson}(N_i)$ samples from the $i$-th distribution, for $i=1,\ldots,M$. Suppose the following assumptions hold:
\begin{enumerate}
\item $\sqrt{N}\left(\frac{N_i}{N} - \pi_i\right)\to 0$, as $N\to \infty$, $i=1,\ldots,M$, where $N = \sum_{i=1}^M N_i$.

\item The $i$-th distribution has a Lebesgue density $f_i^N$ on $\R^d$ such that $f_i^N$ converges pointwise Lebesgue a.e.~to a density $f$ as $N\to\infty$, and $f_i^N$ is equicontinuous almost everywhere, i.e., for Lebesgue a.e.\ $x$, for any $\varepsilon > 0$, there exists $\delta>0$ such that $|f_i^N (y)-f_i^N (x)| < \varepsilon$ whenever $|y-x|<\delta$, for $i = 1,\ldots, M$ and all $N \ge 1$. 
\item $f_i^N$ is uniformly bounded above and has uniformly bounded support, i.e., there exists $C>0$ such that $\|f_i^N\|_\infty < C$ for all $N$ and $i=1,\ldots,M$, and there exists a bounded set that contains ${\rm supp}(f_i^N)$ for all $N$ and $i=1,\ldots,M$.

\item $\mathcal{G}$ is translation and scale invariant, stabilizing on $\mathcal{P}_\lambda$ for some $\lambda>0$.

\item $\mathcal{G}$ is power-law stabilizing with respect to $\phi_N^N $ with order $q_0>d$, with the corresponding full probability sets $A^N$ and radius of stabilization $R$ as defined in \eqref{eq:R_extend}. 

\item The degree bound (including in- and out-degrees) for $\mathcal{G}(\mathcal{P}_{N \phi_N^N})$ is $o_p(N^{1/40})$.
\end{enumerate}
Then with $n=\sum_{i=1}^M n_i$,
$$\frac{1}{\sqrt{n}}\left(n\tilde{H}_n - \E[n\tilde{H}_n]\right)\overset{d}{\to}N(0,\kappa_{1,{\rm null}}^2),$$
where
\begin{equation}\label{eq:kappa_null}
\kappa_{1,{\rm null}}^2 :={a} \left( {g}_1 + {g}_3 \right) + {b} \left( {g}_2 - 2{g}_1 -2{g}_3 - 1  \right) + {c}\left({g}_1-{g}_2+{g}_3 +1 \right)>0
\end{equation}
is equal to the numerator of the asymptotic null variance given in \eqref{eq:null_variance}, which is distribution-free, not depending on $\{f_i^N\}_{i=1}^M$ and $f$.
\end{theorem}

\begin{remark}[Conditions 2 and 3 for parametric models] Suppose that $f_i^N(\cdot) = f(\cdot|\theta_i^N)$ is parametrized by $\theta_i^N$ and the parametric family $\{f(\cdot| \theta): \theta \in \Theta \subset \R^p\}$ ($p \ge 1$) with a common compact support $A$. If  $\theta_i^N$ depends on $N$ and converges to $\theta_0$ in the interior of $\Theta$ as $N\to\infty$, $i=1,\ldots,M$, and $f(x|\theta)$ is continuous in $(x,\theta)$, then conditions 2 and 3 above hold, because $f(x|\theta)$ is uniformly continuous in $(x,\theta)\in A\times B(\theta_0,\varepsilon)$ (for some $\varepsilon >0$), a compact set. 
\end{remark}


\subsection{Asymptotic Power and Detection Threshold}\label{sec:detection_threshold_of_our_test}
In this subsection, we consider $M=2$ and study the power of the test based on $\hat \eta$ along a parametric sub-model \citep{BB20detection}:
suppose $P_i$ has Lebesgue density $f(\cdot |\theta_i)$ on $\R^d$ belonging to the parametric family $\{f(\cdot| \theta): \theta \in \Theta \subset \R^p\}$ ($p \ge 1$), for $i=1,2$.



Our analysis depends on Theorem~\ref{thm:shrinking_alternative}, the CLT under shrinking alternatives.
Under the null, $ \E[n\tilde{H}_n]=0$.
Hence, the local power of the test that rejects the null when
\begin{equation}\label{eq:test}\frac{\sqrt{n}\tilde{H}_n}{ \kappa_{1,{\rm null}}}\geq  z_{1-\alpha}\qquad {\rm or\ equivalently\ ( asymptotically})\qquad \frac{\sqrt{n}\hat{\eta}}{\sigma_{\mathcal{G},K,d,\pi}}\geq  z_{1-\alpha}
\end{equation}
is determined by the rate at which $\E[n\tilde{H}_n]/\sqrt{n}$ converges to 0, since
$\frac{n\tilde{H}_n - \E[n\tilde{H}_n]}{\sqrt{n}\kappa_{1,{\rm null}}}\overset{d}{\to}N(0,1)$
(recall that $\sigma_{\mathcal{G},K,d,\pi}$ has been defined in~\eqref{eq:null_variance}, and $\kappa_{1,{\rm null}}$ is defined in~\eqref{eq:kappa_null}). We next state the detection threshold of the test \eqref{eq:test}, which is of the same form as in \citet{BB20detection}.

\begin{theorem}[Detection threshold]\label{thm:detection_threshold}
Suppose $\{\mathbb{P}_\theta\}_{\theta\in \Theta}$ is a parametric family of distributions with an open convex parameter space $\Theta\subset \R^p$ (for $p\geq 1$) and Lebesgue densities $\{f(\cdot|\theta)\}_{\theta \in \Theta}$. Assume:
\begin{enumerate}
\item For all $\theta\in\Theta$,  $f(\cdot |\theta)$ has a compact and convex support $S\subset \R^d$, with a nonempty interior, not depending on $\theta$.
\item The conditions in Theorem~\ref{thm:shrinking_alternative}  (that guarantee the CLT under shrinking alternatives) hold.

\item For all $\theta\in\Theta$, $f(\cdot|\theta)$ and $\nabla_\theta f(\cdot|\theta)$ are three times continuously differentiable in the interior of $S$, and the Fisher information matrix
$\E_{X\sim f(\cdot|\theta)}\left[\left(\frac{ \nabla_\theta f(X|\theta)}{f(X|\theta)}\right)\left(\frac{ \nabla_\theta f(X|\theta)}{f(X|\theta)}\right)^\top \right]$ is positive definite for all $\theta \in \Theta$.
\item For all $x\in S$, $f(x|\cdot)$ is three times continuously differentiable in $\Theta$.
\end{enumerate}
Suppose that $M=2$ and the discrete kernel $K$ and the $k$-NN graph (for a fixed $k \ge 1$) are used in defining $\hat{\eta}$. Let $P_1$ and $P_2$ have densities $f(\cdot|\theta_1)$ and $f(\cdot|\theta_2)$ respectively, and $\theta_2 = \theta_1 + \varepsilon_N$, where $\varepsilon_N \to 0$ as $N \to \infty$. Let $\kappa_{1,{\rm null}}$ be as defined in Theorem~\ref{thm:shrinking_alternative}, and ${\rm H}_x f(x|\theta_1)$ be the Hessian of $f(x|\theta_1)$ taken w.r.t.\ $x$. Also, define, for $h \in \R^p$,
\begin{equation}\label{eq:a}
a_{k,\theta_1}(h) := -\frac{ \pi_1\pi_2(1-2\pi_2)C_{k,2}}{2kd\kappa_{1,{\rm null}}}\int_S h^\top \nabla_{\theta_1}\left(\frac{ {\rm tr}({\rm H}_x f(x|\theta_1))}{f(x|\theta_1)}\right)f^{\frac{d-2}{d}}(x|\theta_1) {\rm d}x,
\end{equation}
$$b_{k,\theta_1}(h):=\frac{\pi_1\pi_2}{\sigma_{\mathcal{G},K,d,\pi}}\E_{X\sim f(\cdot|\theta_1)} \Bigg[\frac{h^\top \nabla_{\theta_1}f(X|\theta_1)}{f(X|\theta_1)}\Bigg]^2,\ {\rm with} \ C_{k,2}:=\E \Bigg[ \sum_{x\in\mathcal{P}_1^{\mathbf{0}}: (\mathbf{0},x)\in\mathcal{E}(\mathcal{G}_{k{\rm NN}}(\mathcal{P}_1^{\mathbf{0}}))} \|x\|^2 \Bigg].$$
\begin{itemize}
\item[1.] If the dimension $d\leq 8$, then the following hold:
\begin{itemize}
\item[(a)] $\|N^{\frac{1}{4}}\varepsilon_N\|\to 0$: The limiting power of the test \eqref{eq:test} is $\alpha$.
\item[(b)] $N^{\frac{1}{4}}\varepsilon_N \to h$:

If $d\leq 7$, limiting power of the test \eqref{eq:test} is
$\Phi \left(z_\alpha + b_{k,\theta_1}(h)\right)$.

If $d=8$ and the limiting power is
$\Phi \left(z_\alpha + a_{k,\theta_1}(h) + b_{k,\theta_1}(h)\right)$.
\item[(c)] $\|N^{\frac{1}{4}}\varepsilon_N\|\to \infty$: The limiting power of the test \eqref{eq:test} is 1.
\end{itemize}
\item[2.] If the dimension $d \geq 9$, then the following hold:
\begin{itemize}
\item[(a)] $\|N^{\frac{1}{2}-\frac{2}{d}}\varepsilon_N\|\to 0$: The limiting power of the test \eqref{eq:test} is $\alpha$.
\item[(b)] $N^{\frac{1}{2}-\frac{2}{d}}\varepsilon_N \to h$: The limiting power of the test \eqref{eq:test} is
$\Phi \left(z_\alpha + a_{k,\theta_1}(h) \right)$.
\item[(c)] $\|N^{\frac{1}{2}-\frac{2}{d}}\varepsilon_N\|\to \infty$ such that $\|N^{\frac{2}{d}}\varepsilon_N\|\to 0$: Then depending on whether
\begin{equation}\label{eq:threshold}
N^{\frac{1}{2}-\frac{2}{d}}(1-2\pi_2)\int_S \varepsilon_N^\top\nabla_{\theta_1}\left(\frac{{\rm tr}({\rm H}_x f(x|\theta_1))}{f(x|\theta_1)}\right)f^{\frac{d-2}{d}}(x|\theta_1) {\rm d}x\to\left\{\begin{aligned}
&\infty,\\
&-\infty,\\
\end{aligned}\right.\end{equation}
the limiting power of the test \eqref{eq:test} is 0 or 1, respectively.
\item[(d)] $N^{\frac{2}{d}}\varepsilon_N \to h$: The limiting power of the test \eqref{eq:test} is 0 or 1, depending on whether $a_{k,\theta_1}(h)+ b_{k,\theta_1} (h)$ is negative or positive, respectively.
\item[(e)] $N^{\frac{2}{d}}\varepsilon_N\to\infty$: The limiting power of the test \eqref{eq:test} is 1.
\end{itemize}
\end{itemize}
\end{theorem}

A proof of the above result is given in Appendix~\ref{sec:detection_threshold}.
In particular,
the above result shows that the detection threshold of $\hat{\eta}$ exhibits a ``$d=8$" phenomenon (see~\citet{BB20detection}):
When $d\leq 8$, the detection threshold is $N^{-\frac{1}{4}}$;
while when $d\geq 9$, the detection threshold is somewhere between $N^{-\frac{2}{d}}$ and $N^{-\frac{1}{2}+\frac{2}{d}}$,
depending on the direction of $\varepsilon_N$ and the sign of $a_{k,\theta_1}$; see \eqref{eq:a} and \eqref{eq:threshold}.
If $\varepsilon_N=\delta_N h$ for $\delta_N>0$ and some fixed $h\in\R^d$, and
$a_{k,\theta_1}(h) > 0$, then the detection threshold is $\delta_N \sim N^{-\frac{1}{2}+\frac{2}{d}}$;
on the other hand, if $a_{k,\theta_1}(h) < 0$, then the detection threshold is $\delta_N \sim N^{-\frac{2}{d}}$.
When $a_{k,\theta_1}(h) = 0$, the precise location of the detection threshold has to be determined on a case by case basis (see \citep{BB20detection}).
In Appendix~\ref{sec:check_detection_threshold}, we empirically show that the same detection threshold also holds for the non-Poissonized setting, i.e., the original setting in Section~\ref{sec:def_eta} of our paper where the $n_i$'s are nonrandom constants (instead of $n_i\sim {\rm Poisson}(N_i)$).

\section{Proofs of the Main Results}\label{sec:all_proofs}

\subsection{Proof of Lemma~\ref{lem:obs}}\label{pf_lem:obs}
If $P_{1} =\ldots= P_{M}=:P$, then $\tilde{Z}$ follows $P$ regardless of $\Delta$, and so $\tilde{\Delta} \indep \tilde{Z}$.
Conversely, if $\tilde{\Delta} \indep \tilde{Z}$, then the conditional distribution $\tilde{Z}|(\tilde{\Delta} = i)$ should be the same for all $i=1,\ldots, M$, which implies $P_{1} =\ldots= P_{M}$. Hence the $M$ distributions are the same if and only if $\tilde{\Delta}\indep \tilde{Z}$.

If there exist disjoint measurable sets $\{A_i\}_{i=1}^M$ such that $P_i(A_i) = 1$,
then we can almost surely determine the label of $\tilde{Z}$, by finding the $A_i$ that $\tilde{Z}$ belongs to, i.e., $\tilde{\Delta} = \sum_{i=1}^M iI(\tilde{Z}\in A_i)$ can be written as a measurable function of $\tilde{Z}$.
Conversely, if $\tilde{\Delta}$ is a measurable function of $\tilde{Z}$,
then the $\mathcal{E}(\tilde{Z})$-measurable\footnote{$\mathcal{E}(\tilde{Z})$ is the smallest $\sigma$-algebra such that $\tilde{Z}$ is measurable.} set $A_i$ defined as $\tilde{\Delta}^{-1}(i)$ satisfies $P_i(A_i) = \p(\tilde{Z}\in A_i|\tilde{\Delta} = i)= 1$.
Hence there exist disjoint measurable sets $\{A_i\}_{i=1}^M$ such that $P_i(A_i) = 1$ if and only $\tilde{\Delta}$ is a measurable function of $\tilde{Z}$. \qed

\subsection{Proof of Theorem~\ref{thm:3prop}}\label{sec:pf3prop}
Note that $\eta$ is the kernel measure of association between $\tilde{Z}$ and $\tilde{\Delta}$ proposed in \cite[Theorem 2.1]{deb2020kernel}. Given the kernel $K(\cdot,\cdot)$ is characteristic, it satisfies:
\begin{itemize}
	\item[(i)] $\eta$ is a deterministic number between $[0,1]$; 
	
	\item[(ii)] $\eta = 0$ if and only if $\tilde{\Delta} \indep \tilde{Z}$, and 
	
	\item[(iii)] $\eta = 1$ if and only if $\tilde{\Delta}$ is a measurable function of $\tilde{Z}$.
\end{itemize}
Hence the results follow from Lemma~\ref{lem:obs}.
\qed

\subsection{Proof of Proposition~\ref{prop:Inv}}\label{sec:pf_Inv}
Define $(\tilde{\Delta},\tilde{\Delta}{'},\tilde{Z}, \tilde{Z}')$ as follows:
\begin{enumerate}
    \item Draw $(\tilde{\Delta},\tilde{\Delta}',\tilde{Z})$ according to the distribution mentioned in Section~\ref{subsec:def_eta} (see~\eqref{eq:def_eta}), i.e.,
    \begin{enumerate}
        \item Draw $\tilde{\Delta} \sim \sum_{i=1}^M \pi_i \delta_i$, where $\delta_i$ is the Dirac measure at $i$.
        \item Draw $\tilde{Z}\mid (\tilde{\Delta}=i) \sim P_i$.
        \item Draw $\tilde{\Delta}'$ independently from the conditional distribution of $\tilde{\Delta}\mid \tilde{Z}$, given the value of $\tilde{Z}$ from (b).
    \end{enumerate}

    \item Draw $\tilde{Z}'$ from the distribution $\kappa(\tilde{Z},\cdot)$. 
\end{enumerate}
Then $\tilde{Z}' \sim \sum_{i=1}^M \pi_i Q_i$, and  $\eta(Q_1,\ldots,Q_M)$ can be defined from $(\tilde{\Delta},\tilde{\Delta}{'}, \tilde{Z}')$, as $\eta(P_1,\ldots,P_M)$ can be defined from $(\tilde{\Delta},\tilde{\Delta}{'}, \tilde{Z})$. 

The denominator and the second term of the numerator of $\eta$ in \eqref{eq:eta} are unchanged when $P_1,\ldots,P_M$ are transitioned to $Q_1,\ldots,Q_M$ as these only depend on the mixture proportion $\{\pi_i\}_{i=1}^M$ which is unchanged.
Hence it suffices to consider the first term in the numerator and show:
\begin{equation}\label{eq:transform_prop}
\E \left[ \E \big[K(\tilde{\Delta},\tilde{\Delta}{'})|\tilde{Z}\big] \right] \geq \E \left[\E \big[K(\tilde{\Delta},\tilde{\Delta}{'})|\tilde{Z}'\big] \right],
\end{equation}
where on the left-hand side, $\tilde{\Delta}$ and $\tilde{\Delta}{'}$ are i.i.d.~drawn from the conditional distribution of  $\tilde{\Delta}\mid \tilde{Z}$,
and on the right-hand side, $\tilde{\Delta}$ and $\tilde{\Delta}{'}$ are i.i.d.~drawn from the conditional distribution of $\tilde{\Delta}\mid\tilde{Z}'$.

We first consider the case where $\tilde{\Delta}$ is not a deterministic function of $\tilde{Z}'$.
Consider the kernel partial correlation \citep{Huang2022KPC} between $\tilde{\Delta}$ and $\tilde{Z}$ given $\tilde{Z}'$, defined by
$$\rho^2 (\tilde{\Delta},\tilde{Z}|\tilde{Z}') = \frac{ \E \left[ \E \big[K(\tilde{\Delta},\tilde{\Delta}{'})|(\tilde{Z},\tilde{Z}')\big] \right] -  \E \left[ \big[K(\tilde{\Delta},\tilde{\Delta}{'})|\tilde{Z}'\big] \right]}{\E \left[K(\tilde{\Delta},\tilde{\Delta})\right] -  \E \left[ \big[K(\tilde{\Delta},\tilde{\Delta}{'})|\tilde{Z}'\big] \right]},$$
where in $\E \left[ \E \big[K(\tilde{\Delta},\tilde{\Delta}{'})|(\tilde{Z},\tilde{Z}')\big] \right]$, $\tilde{\Delta}$ and $\tilde{\Delta}{'}$ are i.i.d.~drawn from the conditional distribution of $\tilde{\Delta}\mid (\tilde{Z},\tilde{Z}')$ and then averaged over $(\tilde{Z},\tilde{Z}')$.
Since $\tilde{Z}'$ is drawn from $\kappa(\tilde{Z},\cdot)$,
$ \E \big[K(\tilde{\Delta},\tilde{\Delta}{'})|(\tilde{Z},\tilde{Z}')\big]  =  \E \big[K(\tilde{\Delta},\tilde{\Delta}{'})|\tilde{Z}\big]$, and therefore
$$\rho^2 (\tilde{\Delta},\tilde{Z}|\tilde{Z}') = \frac{ \E \left[ \E \big[K(\tilde{\Delta},\tilde{\Delta}{'})|\tilde{Z}\big] \right] -  \E \left[ \big[K(\tilde{\Delta},\tilde{\Delta}{'})|\tilde{Z}'\big] \right]}{\E \left[K(\tilde{\Delta},\tilde{\Delta})\right] -  \E \left[ \big[K(\tilde{\Delta},\tilde{\Delta}{'})|\tilde{Z}'\big] \right]}.$$
It is shown in \citet[Lemma 2]{Huang2022KPC} that the
denominator of $\rho^2 (\tilde{\Delta},\tilde{Z}|\tilde{Z}')$ is positive, and the
numerator of $\rho^2 (\tilde{\Delta},\tilde{Z}|\tilde{Z}')$ is nonnegative.
Hence inequality \eqref{eq:transform_prop} holds.
By \citet[Theorem 1]{Huang2022KPC}, $\rho^2 (\tilde{\Delta},\tilde{Z}|\tilde{Z}') =0$ if and only if $\tilde{\Delta}$ is conditionally independent of $\tilde{Z}$ given $\tilde{Z}'$, which is the same as saying $\tilde{\Delta}\mid (\tilde{Z},\tilde{Z}') \stackrel{d}{=}\tilde{\Delta}\mid \tilde{Z}'$.
Since $\tilde{Z}'$ is obtained by passing $\tilde{Z}$ through $\kappa$, this is further equivalent to $\tilde{\Delta}\mid \tilde{Z} \stackrel{d}{=}\tilde{\Delta}\mid \tilde{Z}'$.

If $\tilde{\Delta}$ is a deterministic function of $\tilde{Z}'$, then $\tilde{\Delta}$ must also be a deterministic function of $\tilde{Z}$, and the inequality \eqref{eq:transform_prop} holds with equality, and $\tilde{\Delta}\mid \tilde{Z} \stackrel{d}{=}\tilde{\Delta}\mid \tilde{Z}'$ holds as well.


In the case of a deterministic data processing, a function $T:\Z\to \Z'$ is applied to all the data points from the $M$ distributions.
In such a case, $\kappa(z,\cdot)=\delta_{T(z)}$, the Dirac measure at $T(z)$.
If $T$ is a bijection, then $\tilde{\Delta}\mid\tilde{Z}\stackrel{d}{=}\tilde{\Delta}\mid T(\tilde{Z}) =\tilde{\Delta}\mid\tilde{Z}'$, so the equality in \eqref{eq:transform_prop} is attained, showing that a bijective transformation of the $M$ distributions does not change our measure of dissimilarity $\eta$. \qed

\subsection{Proof of Proposition~\ref{prop:convex}}\label{sec:pf_prop_convex}
From \eqref{eq:eta_alternative_Radon}, it suffices to show the convexity for the first term in the numerator.
Let $\mu = \sum_{i=1}^M (P_i+Q_i)$ be a dominating measure. The first term in the numerator of \eqref{eq:eta_alternative_Radon}
can be re-written as:
$$\int \sum_{i,j=1}^MK(i,j)\frac{\frac{{\rm d} (\pi_i P_i)}{{\rm d}\mu}}{\frac{{\rm d} \bar{P}}{{\rm d}\mu}}  \frac{\frac{{\rm d} (\pi_j P_j)}{{\rm d}\mu}}{\frac{{\rm d} \bar{P}}{{\rm d}\mu}} \frac{{\rm d} \bar{P}}{{\rm d}\mu} {\rm d}\mu(z) = 
\int \sum_{i,j=1}^MK(i,j)\frac{\frac{{\rm d} (\pi_i P_i)}{{\rm d}\mu} \frac{{\rm d} (\pi_j P_j)}{{\rm d}\mu}}{\frac{{\rm d}(\pi_1 P_1)  }{{\rm d}\mu} + \ldots + \frac{{\rm d}( \pi_M P_M)  }{{\rm d}\mu}}    {\rm d}\mu(z).$$

Note that $f:\R^M \to \R$,
$x\mapsto \frac{x^\top [K(i,j)]_{i,j=1}^M x}{1^\top x}$ is a convex function on $1^\top x > 0$ \citep[Exercise 3.23]{Boyd2004convex}.
Hence,
$$\begin{aligned}
&\quad \int \sum_{i,j=1}^MK(i,j)\frac{\frac{{\rm d} (\pi_i (\lambda P_i + (1-\lambda)Q_i))}{{\rm d}\mu} \frac{{\rm d} (\pi_j (\lambda P_j + (1-\lambda)Q_j))}{{\rm d}\mu}}{\frac{{\rm d}(\pi_1 (\lambda P_1 + (1-\lambda)Q_1)) }{{\rm d}\mu} + \ldots + \frac{{\rm d} (\pi_M (\lambda P_M + (1-\lambda)Q_M))  }{{\rm d}\mu}} {\rm d}\mu(z)\\
&=\int f\left(\frac{{\rm d}(\pi_1 (\lambda P_1 + (1-\lambda)Q_1)) }{{\rm d}\mu}(z),\ldots,\frac{{\rm d}(\pi_M (\lambda P_M + (1-\lambda)Q_M)) }{{\rm d}\mu} (z)\right) {\rm d}\mu(z)\\
&\leq \int \left[\lambda f\left(\frac{{\rm d}(\pi_1  P_1) }{{\rm d}\mu},\ldots,\frac{{\rm d}(\pi_M P_M)  }{{\rm d}\mu}\right)
+ (1-\lambda)f\left(\frac{{\rm d}(\pi_1  Q_1) }{{\rm d}\mu},\ldots,\frac{{\rm d}(\pi_M Q_M)  }{{\rm d}\mu}\right)\right] {\rm d}\mu(z)\\
&=\lambda \int \sum_{i,j=1}^MK(i,j)\frac{\frac{{\rm d} (\pi_i P_i)}{{\rm d}\mu} \frac{{\rm d} (\pi_j P_j)}{{\rm d}\mu}}{\frac{{\rm d}(\pi_1 P_1)  }{{\rm d}\mu} + \ldots + \frac{{\rm d}( \pi_M P_M)  }{{\rm d}\mu}}    {\rm d}\mu(z) \\
&\quad + (1-\lambda)\int \sum_{i,j=1}^MK(i,j)\frac{\frac{{\rm d} (\pi_i Q_i)}{{\rm d}\mu} \frac{{\rm d} (\pi_j Q_j)}{{\rm d}\mu}}{\frac{{\rm d}(\pi_1 Q_1)  }{{\rm d}\mu} + \ldots + \frac{{\rm d}( \pi_M Q_M)  }{{\rm d}\mu}}    {\rm d}\mu(z),
\end{aligned}
$$
and the joint convexity of $\eta$ follows. \qed

\subsection{Proof of a Special Case of $\eta$}\label{sec:pf_special_eta}
Suppose $\Z=\R^d$, $M=2$, and that $P_1$ and $P_2$ have densities  $f$ and $g$ w.r.t.~the Lebesgue measure on $\R^d$. For $z \in \Z$, $$\mathbb{P}(\tilde{\Delta} = 2|Z = z) = \frac{\pi_2 g(z)}{\pi_1 f(z) + \pi_2 g(z)}.$$
Further, $\eta$ can be written as $\eta = 1-\frac{\E [\rho(\tilde{\Delta},\tilde{\Delta}')]}{\E [\rho(\tilde{\Delta}_1,\tilde{\Delta}_2)]}$  (see~\citep[Equation 2.1]{deb2020kernel}), where $\rho(s,s') := \| K(s,\cdot)-K(s',\cdot)\|_{\h}^2$,\footnote{$\h$ is the RKHS induced by kernel $K$ on $\s := \{1,\ldots, M\}$ equipped with inner product $\langle K(s,\cdot),K(s',\cdot)\rangle = K(s,s')$, for $s,s' \in \s$.} and
\begin{eqnarray*}
\E [\rho(\tilde{\Delta} ,\tilde{\Delta}')] & = & \E\left[ \E [\rho(\tilde{\Delta} ,\tilde{\Delta}')|Z] \right] \\ 
& = & \int \left\{2 \rho(1,2)\frac{\pi_1 \pi_2 f(x) g(x)}{[\pi_1 f(x) + \pi_2 g(x)]^2}  \right\} [\pi_1 f(x) + \pi_2 g(x)] \mathrm{d}x \\
& = & 2 \rho(1,2) \int \frac{\pi_1 \pi_2 f(x) g(x)}{\pi_1 f(x) + \pi_2 g(x)}   \mathrm{d}x,
\end{eqnarray*}
and $$\E \big[ \rho(\tilde{\Delta}_1, \tilde{\Delta}_2) \big] = 2 \rho(1,2)\, \pi_1 \pi_2.$$ Therefore,
$\eta =  1 - \int \frac{f(x) g(x)}{\pi_1 f(x) +\pi_2 g(x)}   \mathrm{d}x. $
This completes the argument. \qed

\subsection{Proof of Proposition~\ref{prop:monotonicity}}\label{pf:monotonicity}
We will use the expression in Appendix~\ref{sec:pf_special_eta} of $\eta$, i.e., $\eta =  1 - \int \frac{f(x) g(x)}{\pi_1 f(x) +\pi_2 g(x)}   \mathrm{d}x. $. First consider the location family. Here, we have:
$$\eta =  1 - \int \frac{1}{\pi_1 \frac{1}{f(x)} +\pi_2 \frac{1}{f(x-\theta)}  } \mathrm{d}x.$$
Since $f(x)$ is log-concave, 
$ \frac{1}{f(x)}$ is convex as $\exp(\cdot)$ is convex and increasing.
Write $g := \frac{1}{f}$. Then
$$\begin{aligned}
\eta &= 1 - \int \frac{1}{\pi_1 g(x) +\pi_2 g(x-\theta)} \mathrm{d}x\\
&= 1 - \int \frac{1}{\pi_1 g(x+\pi_2 \theta) +\pi_2 g(x-\pi_1\theta)} \mathrm{d}x\\
&= 1 - \int \frac{1}{\pi_1 g(x+\pi_2\lambda h) +\pi_2 g(x-\pi_1\lambda h)} \mathrm{d}x.
\end{aligned}$$
It suffices to show that $\pi_1 g(x+\pi_2\lambda h) +\pi_2 g(x-\pi_1\lambda h)$ is increasing in $\lambda$, i.e.,
for $0\leq \lambda\leq \Lambda$,
\begin{equation}\label{eq:monotone_location}
\pi_1 g(x+\pi_2\lambda h) +\pi_2 g(x-\pi_1\lambda h)\leq \pi_1 g(x+\pi_2\Lambda h) +\pi_2 g(x-\pi_1\Lambda h).
\end{equation}
Write $x-\pi_1\lambda h$ and $x+\pi_2\lambda h$ as convex combinations of $x-\pi_1\Lambda h$ and $x+\pi_2\Lambda h$, i.e.,
$$\begin{aligned}
x-\pi_1\lambda h &=\frac{\pi_1(\Lambda-\lambda)}{\Lambda} (x+\pi_2\Lambda h)  + \frac{\pi_1\lambda+\pi_2\Lambda}{\Lambda}(x-\pi_1\Lambda h),\\
x+\pi_2\lambda h&= \frac{\pi_1\Lambda+\pi_2\lambda}{\Lambda} (x+\pi_2\Lambda h)  + \frac{\pi_2(\Lambda-\lambda)}{\Lambda} (x-\pi_1\Lambda h).
\end{aligned}$$
By the convexity of $g$, we have:
$$\begin{aligned}
\pi_1 g(x+\pi_2\lambda h) +\pi_2 g(x-\pi_1\lambda h)&\leq \pi_1\left(\frac{\pi_1\Lambda+\pi_2\lambda}{\Lambda} g(x+\pi_2\Lambda h)  + \frac{\pi_2(\Lambda-\lambda)}{\Lambda} g(x-\pi_1\Lambda h)\right)\\
&\qquad +\pi_2 \left(\frac{\pi_1(\Lambda-\lambda)}{\Lambda} g(x+\pi_2\Lambda h)  + \frac{\pi_1\lambda+\pi_2\Lambda}{\Lambda}g(x-\pi_1\Lambda h)\right)\\
&=\pi_1 g(x+\pi_2\Lambda h) +\pi_2 g(x-\pi_1\Lambda h).
\end{aligned}$$
Hence the monotonicity is proved.
When $\lambda=0$, $P_1$ and $P_2$ are equal, so $\eta = 0$. Since $f(\cdot)$ is a density,
$\int_\R f(x-\lambda h)\mathrm{d}\lambda < +\infty$ for Lebesgue almost every $x$,
which implies the convex function $\pi_2 g(x-\lambda h)=\pi_2 \frac{1}{f(x-\lambda h)}\to +\infty$ as $\lambda \to + \infty$
for Lebesgue almost every $x$. Hence, by dominated convergence,
$$\eta = 1 -\int \frac{1}{\pi_1 g(x) +\pi_2 g(x-\lambda h)} \mathrm{d}x \to 1\quad \mathrm{as}\quad \lambda \to + \infty.$$
To show that the monotonicity is strict, we can first suppose that $f(x)=0$ on the boundary of its support since the boundary of a convex set has 0 Lebesgue measure, so it will not change the value of $\eta$.
Since $\eta < 1$, $\pi_1 g(x+\pi_2\lambda h) +\pi_2 g(x-\pi_1\lambda h)$ is not $+\infty$ for Lebesgue almost every $x$.
Hence we can find an $x$ such that the left-hand side of \eqref{eq:monotone_location} is finite.
We want to further find an $x$ such that the inequality in \eqref{eq:monotone_location} is strict.
Note that if equality holds in~\eqref{eq:monotone_location}, then by the convexity, $g(\cdot)$ must be a linear function on the line segment $[x-\pi_1\Lambda h,x+\pi_2\Lambda h]$.
Since $g$ is log-convex, it must be a constant on $[x-\pi_1\Lambda h,x+\pi_2\Lambda h]$.
Note that $\lim_{x\to\infty} g(x) = \lim_{x\to\infty} \frac{1}{f(x)} = \infty$.
Hence we can move $x$ along the line passing through $x-\pi_1\Lambda h$ and $x+\pi_2\Lambda h$ until inequality \eqref{eq:monotone_location} becomes strict (and the left-hand side is still finite).
Now, since a convex function is continuous in the interior of its domain, \eqref{eq:monotone_location} with strict inequality can still hold in a neighborhood of $x$, so that after integration we obtain $\eta(P_1,P_2(\lambda)) < \eta(P_1,P_2(\Lambda))$.

Now consider the case of a scale family. Note that
$$\eta =  1 - \int \frac{1}{\pi_1 \frac{1}{f(x)} +\pi_2 \frac{1}{\lambda f(\lambda x)}  } \mathrm{d}x.$$
By a change of variable $y=\lambda^{\pi_2} x$ we get
$$\eta =  1 - \int \frac{1}{\pi_1 \lambda^{\pi_2} g\left(\frac{y}{\lambda^{\pi_2}}\right) +\pi_2 \frac{1}{\lambda^{\pi_1}} g(\lambda^{\pi_1} y)  } \mathrm{d}y,$$
where $g(x) = \frac{1}{f(x)}$ is still a convex function as before.

We will show that $\phi(\lambda):=\pi_1 \lambda^{\pi_2} g\left(\frac{y}{\lambda^{\pi_2}}\right) +\pi_2 \frac{1}{\lambda^{\pi_1}} g(\lambda^{\pi_1} y)$
is decreasing on $(0,1]$, and increasing on $[1,\infty)$. We can suppose both $\frac{y}{\lambda^{\pi_2}}$ and $\lambda^{\pi_1} y$ lie in the interior of ${\rm supp}(f)$,
because if any of them is not in the interior of ${\rm supp}(f)$, then if $\lambda \geq 1$, $\phi(\Lambda) = + \infty$ for any $\Lambda > \lambda$ by the convexity of $g$, and similarly
if $\lambda < 1$, then $\phi(\delta) = +\infty$ for any $0 < \delta < \lambda$. Given that both the points are in the interior of ${\rm supp}(f)$, we can take the derivative of $\phi$ to obtain
$$\begin{aligned}\phi'(\lambda) &= \pi_1\left(\pi_2 \lambda^{\pi_2-1}g\left(\frac{y}{\lambda^{\pi_2}}\right)  -  \lambda^{\pi_2} \nabla g\left(\frac{y}{\lambda^{\pi_2}}\right)^\top  \frac{\pi_2 y}{\lambda^{\pi_2+1}} \right)\\
&\qquad \qquad+\pi_2\left(-\frac{\pi_1}{\lambda^{\pi_1+1}} g(\lambda^{\pi_1} y) +  \frac{1}{\lambda^{\pi_1}} \nabla g(\lambda^{\pi_1} y)^\top \pi_1\lambda^{\pi_1-1}y  \right)\\
&=\frac{\pi_1\pi_2}{\lambda} \left( \lambda^{\pi_2 }g\left(\frac{y}{\lambda^{\pi_2}}\right)  -  \nabla g\left(\frac{y}{\lambda^{\pi_2}}\right)^\top  y -\frac{1}{\lambda^{\pi_1}} g(\lambda^{\pi_1} y) +   \nabla g(\lambda^{\pi_1} y)^\top y  \right).
\end{aligned}$$
Note that $\phi'(1) = 0$. Let $\psi(\lambda):= \lambda^{\pi_2 }g\left(\frac{y}{\lambda^{\pi_2}}\right)  -  \nabla g\left(\frac{y}{\lambda^{\pi_2}}\right)^\top  y -\frac{1}{\lambda^{\pi_1}} g(\lambda^{\pi_1} y) +   \nabla g(\lambda^{\pi_1} y)^\top y$. It suffices to show that $\psi'(\lambda)> 0$. Observe that
$$\begin{aligned}
\psi'(\lambda)&=\pi_2\lambda^{\pi_2 -1}g\left(\frac{y}{\lambda^{\pi_2}}\right) - \lambda^{\pi_2 } \nabla g\left(\frac{y}{\lambda^{\pi_2}}\right)^\top \frac{\pi_2 y}{\lambda^{\pi_2 + 1}}
+ y^\top \nabla^2 g\left(\frac{y}{\lambda^{\pi_2}}\right) y \cdot \frac{\pi_2}{\lambda^{\pi_2+1}}\\
&\qquad \qquad +\frac{\pi_1}{\lambda^{\pi_1+1}} g(\lambda^{\pi_1} y) - \frac{1}{\lambda^{\pi_1}} \nabla g(\lambda^{\pi_1} y)^\top \pi_1\lambda^{\pi_1-1} y
+ \pi_1\lambda^{\pi_1-1}y^\top  \nabla^2 g(\lambda^{\pi_1} y) y\\
&= \pi_2\lambda^{\pi_2 -1} L\left( \frac{y}{\lambda^{\pi_2}} \right) + \frac{\pi_1}{\lambda^{\pi_1+1}}L\left( \lambda^{\pi_1} y\right),
\end{aligned}$$
where $L(x) := g(x) - \nabla g(x)^\top x + x^\top \nabla^2 g(x) x$.
Since $f(x)$ is log-concave, $g(x) = \frac{1}{f(x)}=\exp(h(x))$ for some convex function $h(x)$. Hence
$$\begin{aligned}
L(x)&=e^{h(x)} - e^{h(x)}\nabla h(x)^\top x + x^\top \left(e^{h(x)}\nabla h(x)\nabla h(x)^\top +  e^{h(x)}\nabla^2 h(x) \right) x\\
&=e^{h(x)}\left(1 - \nabla h(x)^\top x + (\nabla h(x)^\top x)^2 \right) + e^{h(x)} x^\top \nabla^2 h(x) x\\
& = \frac{3}{4}e^{h(x)} + e^{h(x)}\left(\frac{1}{2} - \nabla h(x)^\top x \right)^2 + e^{h(x)} x^\top \nabla^2 h(x) x >0.
\end{aligned}$$
Hence $\psi'(\lambda)> 0$ and the monotonicity is proved.

When $\lambda = 1$, $P_1=P_2$, and so $\eta =0$.
As $\lambda \to 0^+$, no matter $f(0)=0$ or $f(0)>0$, we have $\pi_1 \frac{1}{f(x)} +\pi_2 \frac{1}{\lambda f(\lambda x)}\to +\infty$;
so dominated convergence implies
$$\eta =  1 - \int \frac{1}{\pi_1 \frac{1}{f(x)} +\pi_2 \frac{1}{\lambda f(\lambda x)}  } \mathrm{d}x \to 1.$$
As $\lambda \to +\infty$, by a change of variable $y=\lambda x$ and again using dominated convergence theorem, we get
$$\eta =  1 - \int \frac{1}{\pi_1 \frac{\lambda}{f\left(\frac{y}{\lambda}\right)} +\pi_2 \frac{1}{f(y)}  } \mathrm{d}y \to 1.$$ This completes the proof of the result.
\qed

\subsection{Proof of Theorem~\ref{thm:consist}}\label{sec:pfconsist}
Here we provide a convenient proof of Theorem~\ref{thm:consist} using the consistency result in \citet{deb2020kernel} and a coupling of our data to a process where $\Delta_i$'s  are i.i.d.~Multinoulli($\pi$) with $\pi = (\pi_1,\ldots, \pi_M)$.

Suppose $X^{(m)}_1,X^{(m)}_2 \ldots \overset{i.i.d.}{\sim} P_m$, $m=1,\ldots,M$.
Let $\tilde{\Delta}_i\overset{i.i.d.}{\sim} {\rm Multinoulli}(\pi)$.
If $\tilde{\Delta}_i = m$, we draw $\tilde{Z}_i$ according to $P_m$. More specifically, for $i=1,\ldots, n$,
$$\tilde{Z}_i := \sum_{m=1}^M I (\tilde{\Delta}_i = m) X^{(m)}_{\sum_{j=1}^i I(\tilde{\Delta}_j=m)}.$$
Let $\tilde{n}_m := \sum_{i=1}^n I(\tilde{\Delta}_i=m)$, and  $\tilde{\mathcal{G}}_n$ be the graph constructed on the pooled sample
$(\tilde{Z}_1,\ldots,\tilde{Z}_n)$, which is 
$X^{(1)}_1,\ldots,X^{(1)}_{\tilde{n}_1},X^{(2)}_1,\ldots,X^{(2)}_{\tilde{n}_2},\ldots,X_1^{(M)},\ldots,X^{(M)}_{\tilde{n}_M}$,
after proper permutation.
Define:
$$\tilde{\eta} := \frac{\frac{1}{n} \sum \limits_{i=1}^n \frac{1}{\tilde{d}_i} \sum \limits_{j:(\tilde{Z}_i,\tilde{Z}_j) \in \tilde{\mathcal{G}}_n} K(\tilde{\Delta}_i,\tilde{\Delta}_j)-\frac{1}{n(n-1)}  \sum \limits_{i \ne j} K(\tilde{\Delta}_i,\tilde{\Delta}_j)}{\frac{1}{n}\sum_{i=1}^n K(\tilde{\Delta}_i,\tilde{\Delta}_i)-\frac{1}{n(n-1)}  \sum \limits_{i \ne j} K(\tilde{\Delta}_i,\tilde{\Delta}_j)},$$
where $\tilde{d}_i$ is the out-degree of vertex $i$ in the graph $\tilde{\mathcal{G}}_n$.

From the proof of Theorem~\ref{thm:3prop} we see that $\eta$ is the kernel measure of association \cite{deb2020kernel} between $\tilde{Z}$ and $\tilde{\Delta}$. Here $\tilde{\eta}$ is the empirical estimator of $\eta$ proposed in \cite{deb2020kernel}. We can check the conditions required for the consistency of $\tilde{\eta}$ (i.e., $\tilde{\eta} \stackrel{p}{\to} \eta$) as follows:
Note that any function $f\in \h$\footnote{$\h$ is the RKHS induced by kernel $K$ on $\s := \{1,\ldots, M\}$ equipped with inner product $\langle K(s,\cdot),K(s',\cdot)\rangle = K(s,s')$, for $s,s' \in \s$.} can be written as $f(s) = f(1)I(s=1)+\ldots+f(M)I(s=M)$, for $s \in \s$,
which implies $\h$ is finite-dimensional and hence separable.
As any kernel on the finite set $\s$ is bounded, we have
$\E_{\tilde{Z}\sim Q} [K^{4+\varepsilon} (\tilde{Z},\tilde{Z})]<+\infty$ for any probability distribution $Q$.
Now from \cite[Theorem 3.1]{deb2020kernel} it follows that $\tilde{\eta}\to \eta$ almost surely.

Recall that ${\mathcal{G}}_n$ is the graph constructed on the pooled samples $$(Z_1,\ldots,Z_n)=\left(X^{(1)}_1,\ldots,X^{(1)}_{{n}_1},X^{(2)}_1,\ldots,X^{(2)}_{{n}_2},\ldots,X_1^{(M)},\ldots,X^{(M)}_{{n}_M}\right),$$
with $n_1,\ldots,n_M$ being nonrandom, and
$$\hat{\eta} = \frac{\frac{1}{n} \sum \limits_{i=1}^n \frac{1}{d_i} \sum \limits_{j:(Z_i,Z_j) \in {\mathcal{G}}_n} K({\Delta}_i,{\Delta}_j)-\frac{1}{n(n-1)}  \sum \limits_{i \ne j} K({\Delta}_i,{\Delta}_j)}{\frac{1}{n}\sum_{i=1}^n K({\Delta}_i,{\Delta}_i)-\frac{1}{n(n-1)}  \sum \limits_{i \ne j} K({\Delta}_i,{\Delta}_j)},$$
with $(\Delta_1,\ldots,\Delta_{n})= (1,\ldots,1,\ldots,M,\ldots,M)$ being the class labels such that $n_m$ of them is $m$, i.e., $\sum_{i=1}^n I(\Delta_i = m) = n_m$, for $m=1,\ldots,M$.
The goal is to show that $\hat{\eta}-\tilde{\eta}\overset{a.s.}{\longrightarrow} 0$.

By the strong law of large numbers,
$$\frac{1}{n}\sum_{i=1}^n K(\tilde{\Delta}_i,\tilde{\Delta}_i) \overset{a.s.}{\longrightarrow}\E [K(\tilde{\Delta}_1,\tilde{\Delta}_1)]=
\sum_{i=1}^M \pi_i K(i,i)=\lim_{\frac{n_i}{n}\to\pi_i} \frac{1}{n}\sum_{i=1}^n K({\Delta}_i,{\Delta}_i).$$

By the strong law of large numbers for U-statistics, we have
$$\frac{1}{n(n-1)}  \sum \limits_{i \ne j} K(\tilde{\Delta}_i,\tilde{\Delta}_j)\overset{a.s.}{\longrightarrow}\E [K(\tilde{\Delta}_1,\tilde{\Delta}_2)]=\sum_{i,j=1}^M\pi_i\pi_j K(i,j)=\lim_{\frac{n_i}{n}\to\pi_i} \frac{1}{n(n-1)}  \sum \limits_{i \ne j} K({\Delta}_i,{\Delta}_j).$$

Since $K$ is characteristic,
whenever $\tilde{\Delta}_1\neq \tilde{\Delta}_2$,
$K(\tilde{\Delta}_1,\tilde{\Delta}_1) + K(\tilde{\Delta}_2,\tilde{\Delta}_2) - 2K(\tilde{\Delta}_1,\tilde{\Delta}_2) = \|K(\tilde{\Delta}_1,\cdot) - K(\tilde{\Delta}_2,\cdot)\|_\h^2 > 0$.
This implies $\mathbb{E}K(\tilde{\Delta}_1,\tilde{\Delta}_1) - \mathbb{E}K(\tilde{\Delta}_1,\tilde{\Delta}_2)>0$, so the denominator of $\hat \eta$ has a nonzero deterministic limit. Hence we only need to show
$$\frac{1}{n} \sum \limits_{i=1}^n \frac{1}{d_i} \sum \limits_{j:(Z_i,Z_j) \in {\mathcal{G}}_n} K({\Delta}_i,{\Delta}_j)-
\frac{1}{n} \sum \limits_{i=1}^n \frac{1}{\tilde{d}_i} \sum \limits_{j:(\tilde{Z}_i,\tilde{Z}_j) \in \tilde{\mathcal{G}}_n} K(\tilde{\Delta}_i,\tilde{\Delta}_j)
\overset{a.s.}{\longrightarrow}0.$$
Recall that ${\mathcal{G}}_n$ is the graph constructed on $X^{(1)}_1,\ldots,X^{(1)}_{{n}_1},X^{(2)}_1,\ldots,X^{(2)}_{{n}_2},\ldots,X_1^{(M)},\ldots,X^{(M)}_{{n}_M}$, and $\tilde{\mathcal{G}}_n$ is the graph constructed on
$X^{(1)}_1,\ldots,X^{(1)}_{\tilde{n}_1},X^{(2)}_1,\ldots,X^{(2)}_{\tilde{n}_2},\ldots,X_1^{(M)},\ldots,X^{(M)}_{\tilde{n}_M}$.
Hence $\tilde{\mathcal{G}}_n$ can be obtained from ${\mathcal{G}}_n$ be replacing $\sum_{i=1}^M |n_i - \tilde{n}_i|/2$ points.
Note that replacing one point changes at most $q_n +t_n$ edges.
More specifically, if $\tilde{\mathcal{G}}_{n}$ can be obtained from $\mathcal{G}_n$ by replacing one point with another point (possibly elsewhere), then
$$\#\{(j,k):(j,k)\in \left(\mathcal{E}(\mathcal{G}_n)\setminus \mathcal{E}(\tilde{\mathcal{G}}_n)\right)\cup\left( \mathcal{E}(\tilde{\mathcal{G}}_n)\setminus \mathcal{E}(\mathcal{G}_{n})\right) {\rm\ or\ } i\in \{j,k\} \} \leq q_n + t_n.$$
Hence:
$$\left|\frac{1}{n} \sum \limits_{i=1}^n \frac{1}{d_i} \sum \limits_{j:(Z_i,Z_j) \in \emgn} K(\Delta_i,\Delta_j) - \frac{1}{n} \sum \limits_{i=1}^n \frac{1}{\tilde{d}_i} \sum \limits_{j:(\tilde{Z}_i,\tilde{Z}_j) \in \tilde{\mathcal{G}}_n} K(\tilde{\Delta}_i,\tilde{\Delta}_j) \right|\leq \frac{2(q_n+t_n)\|K\|_{\infty}}{nr_n}\leq \frac{C}{n},$$
for some $C>0$. Therefore, for general $\tilde{\mathcal{G}}_{n}$, by replacing $\sum_{i=1}^M |n_i - \tilde{n}_i|/2$ points, 
$$\left|\frac{1}{n} \sum \limits_{i=1}^n \frac{1}{d_i} \sum \limits_{j:(Z_i,Z_j) \in \emgn} K(\Delta_i,\Delta_j) -  \frac{1}{n} \sum \limits_{i=1}^n \frac{1}{\tilde{d}_i} \sum \limits_{j:(\tilde{Z}_i,\tilde{Z}_j) \in \tilde{\mathcal{G}}_n} K(\tilde{\Delta}_i,\tilde{\Delta}_j)\right|\leq \frac{C\sum_{i=1}^M |n_i - \tilde{n}_i|}{2n}\overset{a.s.}{\longrightarrow}0.$$ This completes the proof of the theorem. \qed

\subsection{Proof of Theorem~\ref{thm:asympnull}}\label{sec:pfNormal}
We will use the following result (proved in Appendix~\ref{sec:pfextPham}) which is an extension of \citet[Theorem 3.2]{Pham1989permu}.
\begin{theorem}\label{thm:extPham}
Suppose we have real numbers $\{a_{ij}\}_{i,j=1}^n$, $\{b_{ij}\}_{i,j=1}^n$ ($a_{ij}$, $b_{ij}$ depend on $n$) satisfying $a_{ii}=b_{ii}=0$, $a_{ij}=a_{ji}$, $b_{ij}=b_{ji}$ for all $i,j$, and $\sum a_{ij}=\sum b_{ij} = 0$.
Let $(R_1,\ldots,R_n)$ be a random permutation of $(1,\ldots,n)$, with all permutations being equally likely.
Set $B_{ij} := b_{R_i R_j}$.
Suppose $t_n$ is a deterministic sequence such that $\frac{t_n^r}{n}\to 0$, as $n \to \infty$, for all $r\in \mathbb{N}$, and
\begin{itemize}
\item[(A1)] $\max_i\sum_j |a_{ij}| = O(\max|a_{ij}|\cdot t_n)$,

\item[(A2)] $\liminf_n \sum a_{ij}^2/(n\max a_{ij}^2 )>0$,

\item[(B1)] $\limsup_n \left(\frac{2\sum b_{i+}^2}{n}\right) / \sum b_{ij}^2 < 1$, $\;\;$ where $b_{i+} :=\sum_{j=1}^n b_{ij}$,

\item[(B2)] $\sum |b_{ij}|^r /n^2 = O\left((\sum b_{ij}^2/n^2)^{\frac{r}{2}} \right)$, \;\; for  $r=3,4,\ldots$.
\end{itemize}
Then $$\frac{\sum a_{ij}B_{ij}}{w_n}\overset{d}{\to} N(0,1)$$ where $w_n^2 := 4\sum' a_{ij}a_{ik} \sum' b_{ij}b_{ik}/n^3 + 2\sum a_{ij}^2 \sum b_{ij}^2 /n^2$, and $\sum'$ means summing over all possible distinct indices (so here it means summing over all distinct triplets $(i,j,k)$).
\end{theorem}

\begin{remark}[On the normalizing constant $w_n$] Note that
$w_n^2$ is not the variance of $\sum a_{ij}B_{ij}$ but they are asymptotically equivalent in the sense that
$w_n^2/{\rm Var}(\sum a_{ij}B_{ij}) \to 1$. We write $w_n$ here because it arises naturally in the proof, and it can be replaced by ${\rm Var}(\sum a_{ij}B_{ij})$ in the statement of Theorem~\ref{thm:extPham}.
\end{remark}

We now show how Theorem~\ref{thm:extPham} can be used to prove Theorem~\ref{thm:asympnull}.
Note that given $\mathcal{F}_n$, the denominator of $\hat{\eta}$ is deterministic, so we only need to focus on the numerator of $\hat{\eta}$. For $i\neq j$, let
\begin{equation}\label{eq:a_ijb_ij}
\begin{aligned}
a_{ij} &:= \frac{1}{d_i}I\{(Z_i,Z_j)\in\emgn\} + \frac{1}{d_j}I\{(Z_j,Z_i)\in\emgn\} - \frac{2}{n-1},\\
b_{ij} &= K(\Delta_i,\Delta_j) - \frac{1}{n(n-1)} \sum_{p\neq q} K(\Delta_p,\Delta_q).
\end{aligned}\end{equation}
Then, $a_{ii}=b_{ii}=0$ for all $i$, and $\sum a_{ij}=\sum b_{ij} = 0$. Note that $\sum a_{ij}B_{ij}$ has the same distribution as the permutation distribution of (the numerator of) $\hat{\eta}$, up to scaling by a constant. Recall Assumption~\ref{assump:degupbd} on the boundedness of $\frac{t_n}{r_n}$, where $t_n$ and $r_n$ are the upper bound and lower bound on vertex degrees in $\mathcal{G}_n$. To verify (A1), note that
$$\begin{aligned}
\sum_j |a_{ij}| &= \sum_{j:(Z_i,Z_j)\ {\rm or\ }(Z_j,Z_i) \in\emgn}|a_{ij}| +  \sum_{j:(Z_i,Z_j), (Z_j,Z_i)\notin\emgn}|a_{ij}|\\
&\leq 2+ \sum_{j:(Z_i,Z_j)\ {\rm or\ }(Z_j,Z_i) \in\emgn}\frac{2}{r_n} \\
&\leq 2+\frac{2t_n}{r_n} \;=  \; O(1)\\
&=O(\max|a_{ij}|\cdot t_n),
\end{aligned}$$
where in the last line we have used the fact that $\max |a_{ij}|\cdot t_n \geq \left( \frac{1}{r_n} - \frac{2}{n-1}\right)\cdot t_n\geq 1-\frac{2t_n}{n-1}\to 1$ as $n\to\infty$.
To check (A2), observe that
\begin{equation}\label{eq:suma_ij}
\begin{aligned}
\sum a_{ij}^2 &\geq \sum_i \sum_{j:(Z_i,Z_j)\in\emgn} a_{ij}^2\\
&\geq \sum_i \sum_{j:(Z_i,Z_j)\in\emgn} \left(\frac{1}{d_i} - \frac{2}{n-1}\right)^2\; = \;\sum_i  \left(\frac{1}{d_i} - \frac{2}{n-1}\right)^2 d_i\\
&\gtrsim \sum_i \frac{1}{d_i^2} \cdot d_i\; \geq \;\frac{n}{t_n} \; \geq \;  \frac{n}{t_n} \left(\frac{\max |a_{ij}|}{\frac{2}{r_n}} \right)^2\\
&\gtrsim \frac{r_n}{4}n \,\max a_{ij}^2 \; \geq \;\frac{1}{4}n \max a_{ij}^2,
\end{aligned}
\end{equation}
where $x \gtrsim y $ means $x \geq Cy$ for some fixed $C>0$.
The third line follows as $\max |a_{ij}| \leq \frac{2}{r_n}$, and the second $\gtrsim$ follows from $\frac{t_n}{r_n}$ being bounded (by Assumption \ref{assump:degupbd}).

To check (B1), we first show that $\frac{1}{n^2}\sum b_{ij}^2 \overset{a.s.}{\longrightarrow} {\rm Var}\left[K(\tilde{\Delta}_1,\tilde{\Delta}_2)\right]$, where $\tilde{\Delta}_i\overset{i.i.d.}{\sim}{\rm Multinoulli}(\pi)$.
$$\begin{aligned}
\frac{1}{n^2}\sum b_{ij}^2=&\sum_{i,j=1,i\neq j}^M \frac{n_in_j}{n^2}\left(K(i,j)- \frac{1}{n(n-1)} \sum_{p\neq q} K(\Delta_p,\Delta_q) \right)^2\\
&+\sum_{i=1}^M \frac{n_i(n_i-1)}{n^2} \left(K(i,i) - \frac{1}{n(n-1)} \sum_{p\neq q} K(\Delta_p,\Delta_q)\right)^2.
\end{aligned}$$
Since $\frac{n_i}{n}\to \pi_i$ and $\frac{1}{n(n-1)} \sum_{p\neq q} K(\Delta_p,\Delta_q) \overset{a.s.}{\longrightarrow} \E [K(\tilde{\Delta}_1,\tilde{\Delta}_2)]$, we have:
\begin{equation}\label{eq:sumb_ij}
\frac{1}{n^2}\sum b_{ij}^2 \overset{a.s.}{\longrightarrow} \sum_{i,j=1}^M \pi_i\pi_j \left(K(i,j) - \E [K(\tilde{\Delta}_1,\tilde{\Delta}_2)]\right)^2 =   {\rm Var}\left[K(\tilde{\Delta}_1,\tilde{\Delta}_2)\right].
\end{equation}
Next, we show 
$\frac{\sum b_{i+}^2}{n^3}\overset{a.s.}{\longrightarrow} \E K(\tilde{\Delta}_1,\tilde{\Delta}_2)K(\tilde{\Delta}_1,\tilde{\Delta}_3) - \left(\mathbb{E}K(\tilde{\Delta}_1,\tilde{\Delta}_2)\right)^2$. Since $b_{ii}=0$,
$$\begin{aligned}
\frac{\sum b_{i+}^2}{n^3} &=\frac{\sum' b_{ij}b_{ik} + \sum b_{ij}^2}{n^3}.
\end{aligned}$$
Similar to the previous argument, we have
$$
\begin{aligned}
\frac{\sum' b_{ij}b_{ik}}{n^3}&\overset{a.s.}{\longrightarrow} \sum_{i,j,k=1}^M \pi_i\pi_j\pi_k \left(K(i,j) - \E [K(\tilde{\Delta}_1,\tilde{\Delta}_2)]\right)\left(K(i,k) - \E [K(\tilde{\Delta}_1,\tilde{\Delta}_2)]\right)\\
&=\E\left[\left(K(\tilde{\Delta}_3,\tilde{\Delta}_4) - \E [K(\tilde{\Delta}_1,\tilde{\Delta}_2)]\right)\left(K(\tilde{\Delta}_3,\tilde{\Delta}_5) - \E [K(\tilde{\Delta}_1,\tilde{\Delta}_2)]\right) \right]\\
&=\E K(\tilde{\Delta}_1,\tilde{\Delta}_2)K(\tilde{\Delta}_1,\tilde{\Delta}_3) - \left(\mathbb{E}K(\tilde{\Delta}_1,\tilde{\Delta}_2)\right)^2.
\end{aligned}$$
Together with $\frac{1}{n^2}\sum b_{ij}^2 \overset{a.s.}{\longrightarrow} {\rm Var}\left[K(\tilde{\Delta}_1,\tilde{\Delta}_2)\right]$ which implies $\frac{1}{n^3}\sum b_{ij}^2 \overset{a.s.}{\longrightarrow}0$, we have:
\begin{equation}\label{eq:convb_iplus}
\frac{\sum b_{i+}^2}{n^3}\overset{a.s.}{\longrightarrow} \E K(\tilde{\Delta}_1,\tilde{\Delta}_2)K(\tilde{\Delta}_1,\tilde{\Delta}_3) - \left(\mathbb{E}K(\tilde{\Delta}_1,\tilde{\Delta}_2)\right)^2.
\end{equation}
To verify (B1), it remains to check that
\begin{equation}\label{eq:a_2b_c}
2\cdot \frac{\E K(\tilde{\Delta}_1,\tilde{\Delta}_2)K(\tilde{\Delta}_1,\tilde{\Delta}_3) - \left(\mathbb{E}K(\tilde{\Delta}_1,\tilde{\Delta}_2)\right)^2}{{\rm Var}\left[K(\tilde{\Delta}_1,\tilde{\Delta}_2)\right]} < 1.
\end{equation}
This follows from the fact
$$\E\left(K(\tilde{\Delta}_1,\tilde{\Delta}_2) - \E_3 K(\tilde{\Delta}_1,\tilde{\Delta}_3) - \E_3 K(\tilde{\Delta}_3,\tilde{\Delta}_2) + \E_{3,4} K(\tilde{\Delta}_3,\tilde{\Delta}_4)\right)^2 >0.$$
The inequality is strict because otherwise
$K(\tilde{\Delta}_1,\tilde{\Delta}_2) = f(\tilde{\Delta}_1) + g(\tilde{\Delta}_2)$ almost surely, which implies
$K(p,q)=f(p)+g(q)$ for all $p,q\in\{1,\ldots,M\}$, and consequently $\sum_{p,q=1}^M K(p,q)x_px_q = 0$ as long as $x_1+\ldots+x_M =0$, contradicting the fact that $K(\cdot,\cdot)$ is characteristic.

To verify (B2), for any $r = 3,4,\ldots$,
since the kernel is bounded, $|b_{ij}|^r$ is also bounded. Together with the fact that
$\frac{1}{n^2}\sum b_{ij}^2 \overset{a.s.}{\longrightarrow} {\rm Var}\left[K(\tilde{\Delta}_1,\tilde{\Delta}_2)\right]>0$, we have
$$\sum |b_{ij}|^r /n^2 = O(1) = O\left((\sum b_{ij}^2/n^2)^{\frac{r}{2}} \right).$$
Hence all the conditions needed for Theorem~\ref{thm:extPham} are satisfied and we have the asymptotic normality of the numerator of $\hat{\eta}$.

Next, we simplify the permutation variance ${\rm Var}(\hat{\eta}|\mathcal{F}_n)$.
Since if the pooled sample is given, the graph $\mathcal{G}_n$ is known up to a permutation.
We can arbitrarily fix a labeling of the vertices as $1,\ldots,n$ with out-degree $d_i$ (which is a slight abuse of notation, since $d_i$ was originally defined as the degree of $Z_i$ whose label is $\Delta_i$).
Let $(\check{\Delta}_1,\ldots,\check{\Delta}_n)$ be a uniformly random permutation of $({\Delta}_1,\ldots,{\Delta}_n)$.
Since the denominator of $\hat{\eta}$ is constant, it suffices to compute the variance of 
$$N_n = \sqrt{n}\left(\frac{1}{n} \sum \limits_{i=1}^n \frac{1}{{d}_i} \sum \limits_{j:(Z_i,Z_j) \in \emgn} K(\check{\Delta}_i,\check{\Delta}_j) - \frac{1}{n(n-1)}  \sum \limits_{i \ne j} K(\Delta_i,\Delta_j)\right).$$
Since $\mathcal{G}_n$ does not contain any self-loop, for any $i\neq j\in\{1,\ldots,n\}$, $\E [K(\check{\Delta}_i,\check{\Delta}_j)] = \frac{1}{n(n-1)}  \sum \limits_{l \ne m} K(\Delta_l,\Delta_m)$. This implies $\E [N_n]=0$ and
$${\rm Var}(N_n)=\frac{1}{n} \E\left[\left(\sum \limits_{i=1}^n \frac{1}{{d}_i} \sum \limits_{j:(Z_i,Z_j) \in \emgn} K(\check{\Delta}_i,\check{\Delta}_j)\right)^2 \right]- \frac{1}{n(n-1)^2}\left( \sum \limits_{i \ne j} K(\Delta_i,\Delta_j)\right)^2.$$
Let
$$\begin{aligned}
\tilde{a} &:= \mathbb{E}[K^2(\check{\Delta}_1,\check{\Delta}_2)]=\frac{1}{n(n-1)}{\sum}' K^2(\Delta_i,\Delta_j),\\
\tilde{b} &:=\mathbb{E}[K(\check{\Delta}_1,\check{\Delta}_2)K(\check{\Delta}_1,\check{\Delta}_3)]=\frac{1}{n(n-1)(n-2)}{\sum}'  K(\Delta_i,\Delta_j) K(\Delta_i,\Delta_l),\\
\tilde{c} &:=\mathbb{E}[K(\check{\Delta}_1,\check{\Delta}_2)K(\check{\Delta}_3,\check{\Delta}_4)]=\frac{1}{n(n-1)(n-2)(n-3)}{\sum}'  K(\Delta_i,\Delta_j) K(\Delta_l,\Delta_m),
\end{aligned}$$
where $\sum'$ means the summation indices are required to be distinct. Then
$$\begin{aligned}
&\quad  \E\left[\left(\sum \limits_{i=1}^n \frac{1}{{d}_i} \sum \limits_{j:(Z_i,Z_j) \in \emgn} K(\check{\Delta}_i,\check{\Delta}_j)\right)^2 \right]\\
 &=\tilde{a} \left(\sum_{i=1}^n \frac{d_i}{d_i^2} + \dsum_{(Z_i,Z_j),(Z_j,Z_i)\in \emgn}\frac{1}{d_i d_j} \right)\\
 &+\tilde{b} \left(\sum_{i=1}^n \frac{d_i(d_i-1)}{d_i^2} + \dsum_{(Z_i,Z_j),(Z_s,Z_i)\in \emgn} \frac{1}{d_id_s}+ \dsum_{(Z_i,Z_s),(Z_s,Z_j)\in \emgn} \frac{1}{d_id_s}+ \dsum_{(Z_i,Z_j),(Z_s,Z_j)\in \emgn} \frac{1}{d_id_s}  \right)\\
 &+\tilde{c} \cdot {\sum}' \frac{\left(d_i-1_{(Z_i,Z_j)\in\emgn}\right)\left(d_j-1_{(Z_j,Z_i)\in\emgn } \right)-T^{\mathcal{G}_n}(i,j)}{d_id_j},
\end{aligned}$$
where $T^{\mathcal{G}_n}(Z_i,Z_j)=\sum_{k=1}^n I\{(Z_i,Z_k),(Z_j,Z_k)\in\emgn \}$ is the number of common out-neighbors of $Z_i$ and $Z_j$.
Note that
$$\begin{aligned}
\dsum_{(Z_i,Z_j),(Z_s,Z_i)\in \emgn} \frac{1}{d_id_s} &= \dsum_{(Z_s,Z_i)\in \emgn} \frac{d_i-1_{(Z_i,Z_s)\in\emgn}}{d_id_s} ,\\
\dsum_{(Z_i,Z_s),(Z_s,Z_j)\in \emgn} \frac{1}{d_id_s}&=\dsum_{(Z_i,Z_s)\in \emgn} \frac{d_s-1_{(Z_s,Z_i)\in\emgn}}{d_id_s},\\
\dsum_{(Z_i,Z_j),(Z_s,Z_j)\in \emgn} \frac{1}{d_id_s} &= {\sum}' \frac{T^{\mathcal{G}_n}(Z_i,Z_s)}{d_id_s}.
\end{aligned}$$
Recall that $\tilde{g}_1 = \frac{1}{n}\sum_{i=1}^n \frac{1}{d_i}$, $\tilde{g}_2 = \frac{1}{n}\sum_{i,j=1}^n \frac{T^{\mathcal{G}_n}(i,j)}{d_id_j}= \frac{1}{n}\sum' \frac{T^{\mathcal{G}_n}(i,j)}{d_id_j} +  \frac{1}{n} \sum_{i=1}^n \frac{1}{d_i}$, and 
$\tilde{g}_3=\frac{1}{n}\dsum_{(Z_i,Z_j),(Z_j,Z_i)\in\emgn}\frac{1}{d_id_j}$. Hence we have:
$$\frac{1}{n} \E\Bigg[\Bigg(\sum \limits_{i=1}^n \frac{1}{{d}_i} \sum \limits_{j:(Z_i,Z_j) \in \emgn} K(\check{\Delta}_i,\check{\Delta}_j)\Bigg)^2 \Bigg] =\tilde{a}(\tilde{g}_1 + \tilde{g}_3)+\tilde{b}(3 - 2\tilde{g}_1 + \tilde{g}_2 -2\tilde{g}_3) + \tilde{c}(n-3+\tilde{g}_1-\tilde{g}_2+\tilde{g}_3).$$
The second term of ${\rm Var}(N_n)$ is easier to handle:
$$\left( \sum \limits_{i \ne j} K(\Delta_i,\Delta_j)\right)^2 = n(n-1)(n-2)(n-3)\cdot \tilde{c} + 4n(n-1)(n-2)\cdot \tilde{b} + 2n(n-1)\cdot \tilde{a}.$$
Combining the two, we get
\begin{equation}\label{eq:cond_var}
{\rm Var}(N_n) = \tilde{a} \Bigg( \tilde{g}_1 + \tilde{g}_3 - \frac{2}{n-1}\Bigg) + \tilde{b} \Bigg( \tilde{g}_2 - 2\tilde{g}_1 -2\tilde{g}_3 - 1 + \frac{4}{n-1} \Bigg) + \tilde{c}\Bigg(\tilde{g}_1-\tilde{g}_2+\tilde{g}_3 +\frac{n-3}{n-1} \Bigg).
\end{equation}
Hence
$$\frac{\hat{\eta}}{{\rm Var}(\hat{\eta}|\mathcal{F}_n)} =\frac{\sqrt{n}\left(\frac{1}{n} \sum \limits_{i=1}^n \frac{1}{{d}_i} \sum \limits_{j:(Z_i,Z_j) \in \emgn} K({\Delta}_i,{\Delta}_j) - \frac{1}{n(n-1)}  \sum \limits_{i \ne j} K(\Delta_i,\Delta_j)\right)}{\tilde{a} \left( \tilde{g}_1 + \tilde{g}_3 - \frac{2}{n-1}\right) + \tilde{b} \left( \tilde{g}_2 - 2\tilde{g}_1 -2\tilde{g}_3 - 1 + \frac{4}{n-1} \right) + \tilde{c}\left(\tilde{g}_1-\tilde{g}_2+\tilde{g}_3 +\frac{n-3}{n-1} \right)}.$$

From the conditional CLT, the unconditional CLT also follows:
$$\p\left(\frac{\hat{\eta}}{{\rm Var}(\hat{\eta}|\mathcal{F}_n)} \leq x \right)=\E\left[\p\left(\frac{\hat{\eta}}{{\rm Var}(\hat{\eta}|\mathcal{F}_n)} \leq x| \mathcal{F}_n\right) \right]\to \Phi(x),$$
by the dominated convergence, where $\Phi(x)$ is the cumulative distribution function of $N(0,1)$. \qed

\subsection{Proof of Corollary~\ref{cor:universal_consist}}\label{sec:universal_consist}
Note that $\tilde{a},\tilde{b},\tilde{c}$ are bounded as the kernel is bounded.
Further,
$\tilde{g}_1\leq \frac{1}{r_n}\leq 1$,
$$
\tilde{g}_2 =  \frac{1}{n}\dsum \frac{T^{\mathcal{G}_n}(i,j)}{d_id_j} +  \frac{1}{n} \sum_{i=1}^n \frac{1}{d_i} \leq \frac{t_n^2}{r_n^2} + \frac{1}{r_n} = O(1)$$ since $t_n$ is an upper bound for the in-degree, and
$$\tilde{g}_3\leq \frac{1}{n}\sum_{i\neq j:(Z_i,Z_j)\in\emgn}\frac{1}{d_id_j}\leq \frac{2t_n}{r_n^2}=O(1).$$
Hence ${\rm Var}(N_n) =O(1)$ and 
therefore ${\rm Var}(\hat{\eta}|\mathcal{F}_n)=O\left(\frac{1}{n}\right)$.
If the $M$ distributions are not equal, then $\eta >0$, and $\hat{\eta}\overset{a.s.}{\longrightarrow}\eta$ since Assumption \ref{assump:conv_nn} holds. Hence:
$$\frac{\hat{\eta}}{\sqrt{{\rm Var}(\hat{\eta}|\mathcal{F}_n)}}\overset{a.s.}{\longrightarrow}+\infty,$$
which implies the universal consistency of the test based on $\hat{\eta}$.

\subsection{Proof of Theorem~\ref{thm:DistFree}}\label{sec:pfDistFree}
The convergence of the denominator of $\hat{\eta}$ (in \eqref{eq:Eta-Hat}) is easy to see, as $n\to\infty$,
$$\frac{1}{n}\sum_{i=1}^n K(\Delta_i,\Delta_i)-\frac{1}{n(n-1)}  \sum \limits_{i \ne j} K(\Delta_i,\Delta_j) \overset{a.s.}{\longrightarrow} \sum_{i=1}^M \pi_i K(i,i)-\sum_{i,j=1}^M \pi_i\pi_j K(i,j)>0.$$
To see why the limit is positive, we write $\tilde{\Delta}_i\overset{i.i.d.}{\sim}{\rm Multinoulli}(\pi)$, for $i=1,\ldots,n$, and it is equivalent to showing that
$\mathbb{E}K(\tilde{\Delta}_1,\tilde{\Delta}_1) - \mathbb{E}K(\tilde{\Delta}_1,\tilde{\Delta}_2)>0$.
Since $K(\cdot,\cdot)$ is characteristic, the distance between Dirac measures
${\rm MMD}(\delta_i,\delta_j)>0$ whenever $i\neq j\in \{1,\ldots,M\}$, so equivalently
$K(\tilde{\Delta}_1,\tilde{\Delta}_1) + K(\tilde{\Delta}_2,\tilde{\Delta}_2) - 2K(\tilde{\Delta}_1,\tilde{\Delta}_2) > 0$ whenever $\tilde{\Delta}_1\neq \tilde{\Delta}_2$. After taking the expectation, we have $\mathbb{E}K(\tilde{\Delta}_1,\tilde{\Delta}_1) - \mathbb{E}K(\tilde{\Delta}_1,\tilde{\Delta}_2)>0$.

The (conditional) variance of $N_n$ (in Equation \ref{eq:cond_var}), i.e., the numerator of $\hat{\eta}$, involves $\tilde{a},\tilde{b},\tilde{c}$ and $\tilde{g}_1,\tilde{g}_2,\tilde{g}_3$. It is clear that:
$$\begin{aligned}
\tilde{a} &=\frac{1}{n(n-1)}{\sum}' K^2(\Delta_i,\Delta_j)\overset{a.s.}{\longrightarrow} \sum_{i,j=1}^M \pi_i\pi_j K^2(i,j)=a,\\
\tilde{b} &=\frac{1}{n(n-1)(n-2)}{\sum}'  K(\Delta_i,\Delta_j) K(\Delta_i,\Delta_l)\overset{a.s.}{\longrightarrow} \sum_{i,j,l=1}^M \pi_i\pi_j\pi_l K(i,j)K(i,l)=b,\\
\tilde{c} &=\frac{1}{n(n-1)(n-2)(n-3)}{\sum}'  K(\Delta_i,\Delta_j) K(\Delta_l,\Delta_m)\\
	&\overset{a.s.}{\longrightarrow} \sum_{i,j,l,m=1}^M \pi_i\pi_j\pi_l\pi_m K(i,j)K(l,m) = \left(\sum_{i,j=1}^M \pi_i\pi_j K(i,j)\right)^2=c,\ {\rm as\ }n\to\infty.
\end{aligned}$$
To show the convergence of $\tilde{g}_1,\tilde{g}_2,\tilde{g}_3$,
note that they are functions of $\mathcal{G}(Z_1,\ldots,Z_n)$, with $\{Z_1,\ldots,Z_n\}$ being the pooled sample, and $\mathcal{G}$ being a stabilizing graph as defined in Section \ref{sec:asymp_dist_free}. Further, they can all be expressed in terms of $\sum_{i=1}^n\xi(Z_i,\{Z_1,\ldots,Z_n\})$, and thus can be analyzed using the tool of {\it stabilizing functions} to be introduced next.

\begin{defn}[Stabilizing functions \citep{penrose2003weak}]\label{def:stablefunc}
Let $\xi(x,\X)$ be a measurable $\R^+$-valued function defined for all finite set $\X\subset\R^d$ and
$x\in\X$. If $x\notin \X$, define $\xi(x,\X):=\xi(x,\X\cup\{x\})$.
$\xi$ is said to be {\it translation invariant}, if
$\xi(x+y,\X+y) = \xi(x,\X)$ for
all finite set $\X\subset\R^d$ and $x,y\in\R^d$.
For a locally finite\footnote{$\X\subset\R^d$ is said to be locally finite if its intersection with any compact set is finite.} set $\X$, if
$$\begin{aligned}
&\liminf_{m\to\infty} \underbrace{ \inf_{n\in\mathbb{N}}\left(\mathop{\rm ess\ inf}_{\substack{\mathcal{A}\subset \R^d\backslash B(\mathbf{0},m) \\ |\mathcal{A}|=n}}\xi\left(\mathbf{0},\big(\X\cap B(\mathbf{0},m)\big) \cup\mathcal{A}\right) \right)}_{{\rm denoted\ by\ }\underline{\xi}(\X,m)}\\
=&\limsup_{m\to\infty} \underbrace{  \sup_{n\in\mathbb{N}}\left(\mathop{\rm ess\ sup}_{\substack{\mathcal{A}\subset \R^d\backslash B(\mathbf{0},m) \\ |\mathcal{A}|=n}}\xi\left(\mathbf{0},\big(\X\cap B(\mathbf{0},m)\big) \cup\mathcal{A}\right)\right)}_{{\rm denoted\ by\ }\bar{\xi}(\X,m)},
\end{aligned}$$
where the essential supremum/infimum is taken with respect to the Lebesgue measure on $\R^{dn}$,
then $\xi$ is said to {\it stabilize} on $\X$.
In such a case, we define $\xi(\mathbf{0},\X)$ as the limit of the above quantity.
\end{defn}

We will be interested in functions that stabilize on the homogeneous Poisson process $\mathcal{P}_\lambda$, as the pooled sample $\{Z_1,\ldots,Z_n\}$ is locally close to a homogeneous Poisson process.
These functions arise naturally from stabilizing graphs. For example, if $\mathcal{G}$ is stabilizing on $\mathcal{P}_\lambda$, and $\xi$ satisfies, for any finite set $\X$ containing $\mathbf{0}$, $\xi(\mathbf{0},\X)$ only depends on the edge set $\mathcal{E}\left(\mathbf{0},\mathcal{G}(\X) \right)$, then $\xi$ stabilizes on $\mathcal{P}_{\lambda}$.

\begin{theorem}[{\citep[Theorem 2.1]{penrose2003weak}}]
Suppose $q=1$ or $q=2$. Let $X_1,X_2,\ldots$ be i.i.d.~$d$-dimensional random variables with common density $f$ and $\X_n:=\{X_1,\ldots,X_n\}$. Suppose $\xi$ is translation invariant and almost surely stabilizing on the homogeneous Poisson process $\mathcal{P}_\lambda$ for all $\lambda>0$.
If $\xi$ satisfies the moment condition
$$\sup_{n\in\mathbb{N}}\mathbb{E}\left[\xi(n^{1/d}X_1,n^{1/d}\X_n)^p\right]<\infty$$
for some $p>q$, then as $n\to\infty$,
$$\frac{1}{n}\sum_{x\in\X_n}\xi(n^{1/d}x,n^{1/d}\X_n)\to\int_{\R^d}\mathbb{E}[\xi(\mathbf{0},\mathcal{P}_{f(x)})]f(x){\rm d}x,\quad in\ L^q$$
and the right-hand side above is finite.
Here $\mathcal{P}_{f(x)}$ is a homogeneous Poisson process with intensity $f(x)$.

In particular, if $\xi(x,\X)$ is also scale invariant (i.e., $\xi(ax,a\X)=\xi(x,\X)$ for all $a>0$), then
$$\frac{1}{n}\sum_{x\in\X_n}\xi(x,\X_n)\to\mathbb{E}[\xi(\mathbf{0},\mathcal{P}_{1})],\quad in\ L^q.$$
\end{theorem}

Let us now get back to the proof of our Theorem \ref{thm:DistFree}. Since the graph $\mathcal{G}$ is stabilizing, $\xi(x,\X) := \frac{1}{d(x,\mathcal{G}(\X))}$ is a bounded stabilizing function.
Using the above theorem, by the same coupling used in Appendix~\ref{sec:pfconsist} (the proof of the consistency theorem for $\hat{\eta}$), the limit of $\tilde{g}_1$ is unchanged when we replace $Z_1,\ldots,Z_n$ by i.i.d.~observations from $\sum_{i=1}^M \pi_i P_i$:
$$\tilde{g}_1 = \frac{1}{n}\sum_{i=1}^n \frac{1}{d_i}\overset{L^2}{\longrightarrow}\mathbb{E}\left[\frac{1}{d\big(\mathbf{0},\mathcal{G}(\mathcal{P}_1^\mathbf{0})\big)}\right]=g_1.$$
By \citet[Lemma 3.3]{penrose2003weak}, $\xi(x,\X) := \sum_{y\neq z:(y,x),(z,x)\in\mathcal{E}(\mathcal{G}(\X))}\frac{1}{d\left(y,\mathcal{G}(\X)\right) d\left(z,\mathcal{G}(\X)\right)}$ is also bounded and stabilizing on $\mathcal{P}_\lambda$ for any $\lambda>0$. Hence
$$\begin{aligned}
\tilde{g}_2-\tilde{g}_1&= \frac{1}{n}\sum_{i\neq j}\frac{T^{\mathcal{G}_n}(i,j)}{d_id_j} = \frac{1}{n}\sum_{i=1}^n\sum_{j\neq k:(Z_j,Z_i),(Z_k,Z_i)\in\emgn} \frac{1}{d_jd_k}\\
&\overset{L^2}{\longrightarrow}\mathbb{E}\left[\sum_{y\neq z:(y,\mathbf{0}),(z,\mathbf{0})\in\mathcal{E}(\mathcal{G}(\mathcal{P}_1^\mathbf{0}))}\frac{1}{d\left(y,\mathcal{G}(\mathcal{P}_1^\mathbf{0})\right) d\left(z,\mathcal{G}(\mathcal{P}_1^\mathbf{0})\right)}\right]=g_2-g_1.
\end{aligned}$$
Similarly, by considering $\xi(x,\X) := \sum_{y:(y,x),(x,y)\in\mathcal{E}(\mathcal{G}(\X))}\frac{1}{d\left(x,\mathcal{G}(\X)\right) d\left(y,\mathcal{G}(\X)\right)}$, we have
$$\tilde{g}_3 = \frac{1}{n}{\sum_{(Z_i,Z_j),(Z_j,Z_i)\in\emgn}}\frac{1}{d_id_j}\overset{L^2}{\longrightarrow}\mathbb{E}\left[ \sum_{y:(y,\mathbf{0}),(\mathbf{0},y)\in\mathcal{E}(\mathcal{G}(\mathcal{P}_1^\mathbf{0}))}\frac{1}{d\left(\mathbf{0},\mathcal{G}(\mathcal{P}_1^\mathbf{0})\right) d\left(y,\mathcal{G}(\mathcal{P}_1^\mathbf{0})\right)} \right]=g_3.$$
Hence the (conditional) variance of the numerator of $\hat{\eta}$, scaled by $\sqrt{n}$, converges in $L^2$ to
$$\begin{aligned}
&\quad {a} \left( {g}_1 + {g}_3 \right) + {b} \left( {g}_2 - 2{g}_1 -2{g}_3 - 1  \right) + {c}\left({g}_1-{g}_2+{g}_3 +1 \right)\\
&= (a-2b+c)(g_1+g_3)+(g_2-1)(b-c).
\end{aligned}$$

To show the limiting variance is positive, note that from \eqref{eq:a_2b_c}, $a-2b+c>0$, and from \eqref{eq:convb_iplus}, $b\geq c$.
Hence it remains to show $g_2\geq 1$. Note that
$$\begin{aligned}
\tilde{g}_2 &= \frac{1}{n}\sum_{i\neq j}\frac{T^{\mathcal{G}_n}(i,j)}{d_id_j}  + \frac{1}{n}\sum_{i=1}^n\frac{1}{d_i}\\
&=\frac{1}{n}\sum_{i=1}^n \sum_{j\neq l:(Z_j,Z_i),(Z_l,Z_i)\in\emgn} \frac{1}{d_jd_l}+ \frac{1}{n}\sum_{i=1}^n\frac{1}{d_i}\\
&=\frac{1}{n}\sum_{i=1}^n \left(\sum_{j:(Z_j,Z_i)\in\emgn} \frac{1}{d_j}\right)^2 - \frac{1}{n}\sum_{i=1}^n \sum_{j:(Z_j,Z_i)\in\emgn} \frac{1}{d_j^2}+ \frac{1}{n}\sum_{i=1}^n\frac{1}{d_i}\\
&=\frac{1}{n}\sum_{i=1}^n \left(\sum_{j:(Z_j,Z_i)\in\emgn} \frac{1}{d_j}\right)^2\\
&\geq \frac{1}{n^2}\left(\sum_{i=1}^n \sum_{j:(Z_j,Z_i)\in\emgn} \frac{1}{d_j}\right)^2=1,
\end{aligned}$$
where in the last step we have used the Cauchy-Schwarz inequality.
Hence $g_2$, as the $L^2$ limit of $\tilde{g}_2$, is greater than or equal to 1. \qed

\subsection{Proof of Proposition~\ref{prop:knnExpStab}}\label{pf:knnExpStab}
Recall $\phi=\sum_{i=1}^M\pi_i f_i$. Let $\mathcal{O}$ be the interior of ${\rm supp}(\phi)$ --- a set with probability 1, under $\phi$. Let $\mathcal{C}_i$, $1\leq i\leq I$, be a finite collection of infinite open cones covering $\R^d$ with angular radius $\pi/12$ and apex at $\mathbf{0}$.
For $x\in\mathcal{O} $ and $1\leq i\leq I$, let $\mathcal{C}_i(x)$ be the translate of $\mathcal{C}_i$ with apex at $x$. Let $\mathcal{C}_i^+(x)$ be the open cone concentric to $\mathcal{C}_i$ with apex $x$ and angular radius $\pi/6$.
For $1\leq i\leq I$, let $N^{-1/d}R_i(x)$ be the distance from $x$ to its $k$-th nearest neighbor in $\mathcal{P}_{N\phi_N}\cap \mathcal{C}_i^+(x)$, if this $k$-th nearest neighbor exists at a distance less than ${\rm diam}(\mathcal{C}_i(x)\cap \mathcal{O})$, and otherwise set $N^{-1/d}R_i(x) = {\rm diam}(\mathcal{C}_i(x)\cap \mathcal{O})$.
Then $\max_{1\leq i\leq I}R_i(x)$ is a radius of stabilization.
Note that $\mathbb{P}\left[ R_i(x) > t\right]$ is 0 unless ${\rm diam}(\mathcal{C}_i(x)\cap \mathcal{O})> N^{-1/d}t$.
Now, consider ${\rm diam}(\mathcal{C}_i(x)\cap \mathcal{O})\geq N^{-1/d}t$. Then there exists a $y\in\mathcal{C}_i(x)\cap \mathcal{O}$ such that $|y-x|=N^{-1/d}t$. By convexity, $\frac{x+y}{2}\in \mathcal{C}_i(x)\cap \mathcal{O}$.
With $\gamma:=\frac{1}{2}\sin(\pi/12)$, we have
$B(\frac{x+y}{2},\gamma N^{-1/d}t)\subset \mathcal{C}_i^+(x)\cap \mathcal{O}$.
$R_i(x) > t$ implies that $B(\frac{x+y}{2},\gamma N^{-1/d}t)$ contains less than $k$ points.
Since $\phi_N$ is bounded below,
the number of points in $B(\frac{x+y}{2},\gamma N^{-1/d}t)$ is a Poisson random variable with mean at least
$N \int_{B(\frac{x+y}{2},\gamma N^{-1/d}t)} \inf \{\phi_N\}{\rm d}z =  \delta t^d$ for some $\delta>0$ independent of $i,x,t,N$. Hence
$$\begin{aligned}
\mathbb{P}[R_i>t]&\leq \mathbb{P}\left[{\rm Poisson}\left(\delta t^d \right)< k\right]\\
&=\sum_{i=0}^{k-1} \frac{(\delta t^d)^i}{i!}e^{-\delta t^d}\leq C_1e^{-C_2 t^d},
\end{aligned}$$
where $C_1,C_2>0$ are independent of $i,x,t,N$.
Hence,
$$\begin{aligned}
\tau(t) &= \sup_{N\geq 1,x\in \mathcal{O}}\mathbb{P}\left[\max_{1\leq i\leq I} R_i > t \right] \leq  IC_1e^{-C_2 t^d}.
\end{aligned}$$ which yields the desired result. \qed

\subsection{Proof of Theorem~\ref{thm:normal_alternative}} \label{pf:normal_alternative}

We first provide a roadmap for the proof. Recall the definitions of $H_n$, $\mathcal{F}_n$, $\{\kappa_i\}_{i=1}^2$ in Appendix~\ref{subsec:general_fixed_alt}.
In the following, we first provide some lemmas that describe the local behavior of $\mathcal{P}_{N\phi_N}$.
Next, in Step 1 of the main proof, we show $H_n - \mathbb{E}(H_n|\mathcal{F}_n)\mid\mathcal{F}_n \overset{d}{\to}N(0,\kappa_1^2)$, where the provided lemmas are crucial in establishing the convergence of the variance, and the CLT is established using Stein's method.
Next, in Step 2,  we show that
$\mathbb{E}(H_n|\mathcal{F}_n) - \mathbb{E}(H_n)\overset{d}{\to}N(0,\kappa_2^2)$.
This is the step where we use $\sqrt{N}\left(\frac{N_p}{N} - \pi_p\right)\to 0$ and the power-law stabilization. To establish the convergence of variance, we need power-law stabilization of order $q>d$ (recall that $\Z \subset \R^d$), and to establish the CLT using Stein's method, power-law stabilization of order $q>16d$ is needed.
Finally, combining Step 1 and Step 2, we can show that $H_n - \mathbb{E}(H_n) \overset{d}{\to}N(0,\kappa_1^2+\kappa_2^2)$.

\subsubsection{Preliminaries}
The following four lemmas are modifications of Lemma 3.1, Lemma 3.2, and Proposition 3.1 from \citet{penrose2003weak}, which describe that the local behavior of $\mathcal{P}_{N\phi_N}$ at $z$ is similar to a homogeneous Poisson process $\mathcal{P}_{\phi(z)}$, where $\phi_N(z) := \sum_{i=1}^M \frac{N_i}{N} f_i(z)$, for $z\in\Z\subset\R^d$, with $\frac{N_i}{N}\to\pi_i$ as $N\to\infty$, and $\phi(z):=\sum_{i=1}^M \pi_i f_i(z)$.

We say that $x\in\R^d$ is a {\it Lebesgue point} of a function $f:\R^d\to \R$, if
$$\lim_{r\to 0^+}\frac{1}{|B(x,r)|}\int_{B(x,r)} |f(y)-f(x)|{\rm d}y=0.$$ If $x$ is a continuity point of $f$, then $x$ is a Lebesgue point of $f$. More generally, the {\it Lebesgue differentiation theorem} states that given $f\in L^1(\R^d)$, a.e.~$x$ is a Lebesgue point of $f$.

\begin{lemma}\label{lem:coupl}
Recall that $\phi(z)=\sum_{i=1}^M \pi_i f_i(z)$, for $z \in \Z \subset \R^d$.
Suppose $x_0$ is a Lebesgue point of $f_1,\ldots,f_M$. Then there exists a homogeneous Poisson process $\mathcal{P}_{\phi(x_0)}$, coupled with $\mathcal{P}_{N\phi_N}$, such that for all $K>0$,
$$\mathbb{P}\left[ N^{1/d}(\mathcal{P}_{N\phi_N} - x_0)\cap B(\mathbf{0},K) = \mathcal{P}_{\phi(x_0)} \cap B(\mathbf{0},K)\right]\to 1,\quad {\rm as\ }N \to\infty.$$
\end{lemma}
\begin{proof}
Let $\mathcal{P}$ be a homogeneous Poisson process of rate 1 on $\R^d \times [0,\infty)$.
Let $\mathcal{P}_{N\phi_N}$ be the projection of the set
$$\left\{ (x,t)\in \mathcal{P}:t \leq N \phi_N(x)\right\}$$
onto the $x$ space, i.e., $(x,t)\mapsto x$.
Then $\mathcal{P}_{N\phi_N}$ is a non-homogeneous Poisson process with intensity function $N\phi_N$.
Let $\mathcal{P}_{\phi(x_0)}$ be the image of the point set
$$\{(x,t)\in \mathcal{P}:t\leq N\phi(x_0)\}$$
under the mapping
$$(x,t)\mapsto N^{1/d}(x-x_0).$$
Note that $\mathcal{P}_{\phi(x_0)}$ is a homogeneous Poisson process on $\R^d$ with intensity $\phi(x_0)$ for all $N$.
The number of points in 
$\left( N^{1/d}(\mathcal{P}_{N\phi_N} - x_0) \triangle \mathcal{P}_{\phi(x_0)}\right)\cap B(\mathbf{0},K)$
equals the number of points $(x,t)\in \mathcal{P}$ such that $x\in B(x_0,N^{-1/d}K)$ and $t$ is between $\phi_N(x)$ and $\phi(x_0)$, and follows a Poisson distribution with mean
\begin{equation}\label{eq:coupl}
N \int_{B(x_0,N^{-1/d}K)}|\phi_N(x)-\phi(x_0)|{\rm d}x\leq N \int_{B(x_0,N^{-1/d}K)}(|\phi_N(x)-\phi(x)|+|\phi(x)-\phi(x_0)|){\rm d}x,
\end{equation}
which converges to 0 as $x_0$ is also a Lebesgue point of $\phi$. 
\end{proof}

\begin{lemma}\label{lem:local_prop}
Suppose $x_0$ is a Lebesgue point of $f_1,\ldots,f_M$ and $\phi(x_0)>0$. Let $\xi(x,\X)$ be a translation invariant function that almost surely stabilizes on $\mathcal{P}_{\phi(x_0)}$ (see Definition~\ref{def:stablefunc}). Suppose the moment condition
$$\sup_{N \geq 1}\mathbb{E}\left[\left|\xi\left(\mathbf{0},N^{1/d}(\mathcal{P}_{N\phi_N} - x_0)\right)\right|^p\right]<\infty$$
is satisfied for some $p>1$. Then
$$\xi\left(\mathbf{0},N^{1/d}(\mathcal{P}_{N\phi_N} - x_0)\right) \to \xi\left(\mathbf{0},\mathcal{P}_{\phi(x_0)}\right)$$
in distribution and in expectation.
\end{lemma}
\begin{proof}
Consider the coupling of $\mathcal{P}_{N\phi_N}$ and $\mathcal{P}_{\phi(x_0)}$ in the previous lemma. We have
$$\begin{aligned}
&\quad \mathbb{P}\left[\left|\xi\left(\mathbf{0},N^{1/d}(\mathcal{P}_{N\phi_N} - x_0)\right) - \xi\left(\mathbf{0},\mathcal{P}_{\phi(x_0)}\right) \right| > \varepsilon\right]\\
&\leq \mathbb{P}\big[ N^{1/d}(\mathcal{P}_{N\phi_N} - x_0)\cap B(\mathbf{0},K) \neq \mathcal{P}_{\phi(x_0)}\cap B(\mathbf{0},K)  \big] + \mathbb{P}\big[\bar{\xi}(\mathcal{P}_{\phi(x_0)},K) - \underline{\xi}(\mathcal{P}_{\phi(x_0)},K)\geq \varepsilon \big],
\end{aligned}$$
where $\bar{\xi}$ and $\underline{\xi}$ are defined in Definition \ref{def:stablefunc}.
By the stabilization assumption, we can choose $K > 0$ so that second term is less than $\delta$, for any $\delta > 0$, by taking $K$ large. By the previous lemma, the first term converges to 0 as $N\to\infty$. Since $\delta>0$ is arbitrary, it follows that
$$\xi\left(\mathbf{0},N^{1/d}(\mathcal{P}_{N\phi_N} - x_0)\right)\overset{d}{\longrightarrow} \xi\left(\mathbf{0},\mathcal{P}_{\phi(x_0)}\right).$$
The assumption on the boundedness of the $p$-th moment implies that $\{\xi\left(\mathbf{0},N^{1/d}(\mathcal{P}_{N\phi_N} - x_0)\right)\}_{N \ge 1}$ is uniformly integrable; hence its expectation converges to the expectation of $ \xi\left(\mathbf{0},\mathcal{P}_{\phi(x_0)}\right)$.
\end{proof}

\begin{lemma}\label{lem:coupl2}
Suppose $x_0\neq y_0$ are Lebesgue points of $f_1,\ldots,f_M$.
There exist homogeneous Poisson processes $\mathcal{P}_{\phi(x_0)}$ and $\mathcal{P}_{\phi(y_0)}$ coupled with $\mathcal{P}_{N\phi_N}$ such that $\mathcal{P}_{\phi(x_0)}$ is independent of $\mathcal{P}_{\phi(y_0)}$ and for all $K>0$:
$$\lim_{N \to\infty }\mathbb{P}\left[ N^{1/d}(\mathcal{P}_{N\phi_N}^{y_0} - x_0)\cap B(\mathbf{0},K) = \mathcal{P}_{\phi(x_0)} \cap B(\mathbf{0},K)\right]\ = 1,$$
$$\lim_{N \to\infty }\mathbb{P}\left[ N^{1/d}(\mathcal{P}_{N\phi_N}^{x_0} - y_0)\cap B(\mathbf{0},K) = \mathcal{P}_{\phi(y_0)} \cap B(\mathbf{0},K)\right]\ = 1.$$
\end{lemma}

\begin{proof}
Let $\mathcal{P}$ be a homogeneous Poisson processes of rate 1 on $\R^d \times [0,\infty)$.
Let $\mathcal{P}_{N\phi_N}$ be the projection of the set
$$\left\{ (x,t)\in \mathcal{P}:t \leq N \phi_N(x)\right\}$$
onto the $x$ space, i.e., $(x,t)\mapsto x$.
Let $\mathcal{Q}$ be an independent copy of $\mathcal{P}$.
Let $F_{x_0} :=\{z\in\R^d:\|x_0-z\|< \|y_0-z\|\}$ be the half-space of points in $\R^d$ closer to $x_0$ than to $y_0$ and let $F_{y_0} :=\{z\in\R^d:\|y_0-z\|< \|x_0-z\|\}$ be the half-space of points in $\R^d$ closer to $y_0$ than to $x_0$.
Let $\mathcal{P}_{\phi(x_0)}$ be the image of the point set
$$\{(x,t)\in \mathcal{P}\cap F_{x_0}\times [0, N\phi(x_0)]\}\cup \{(x,t)\in \mathcal{Q}\cap F_{y_0}\times [0,N\phi(x_0)]\}$$
under the mapping
$$(x,t)\mapsto N^{1/d}(x-x_0).$$
Note that $\mathcal{P}_{\phi(x_0)}$ is a homogeneous Poisson process on $\R^d$ with intensity $\phi(x_0)$ for all $N$.
Note that when $N$ is large, $B(x_0,N^{-1/d}K)\subset F_{x_0}$.
Hence, the number of points in 
$\left( N^{1/d}(\mathcal{P}_{N\phi_N}^{y_0} - x_0) \triangle \mathcal{P}_{\phi(x_0)}\right)\cap B(\mathbf{0},K)$
equals the number of points $(x,t)\in \mathcal{P}$ such that $x\in B(x_0,N^{-1/d}K)$ and $t$ is between $\phi_N(x)$ and $\phi(x_0)$, and follows a Poisson distribution with mean
$$N \int_{B(x_0,N^{-1/d}K)}|\phi_N(x)-\phi(x_0)|{\rm d}x\leq N \int_{B(x_0,N^{-1/d}K)}(|\phi_N(x)-\phi(x)|+|\phi(x)-\phi(x_0)|){\rm d}x,$$
which converges to 0 as $x_0$ is also a Lebesgue point of $\phi$. 

Let $\mathcal{P}_{\phi(y_0)}$ be the image of the point set
$$\{(x,t)\in \mathcal{P}\cap F_{y_0}\times [0, N\phi(y_0)]\}\cup \{(x,t)\in \mathcal{Q}\cap F_{x_0}\times [0,N\phi(y_0)]\}$$
under the mapping
$$(x,t)\mapsto N^{1/d}(x-y_0).$$
$\mathcal{P}_{\phi(x_0)}$ and $\mathcal{P}_{\phi(y_0)}$ are independent as they are constructed from Poisson processes on disjoint regions of space. The number of points in 
$\left( N^{1/d}(\mathcal{P}_{N\phi_N}^{x_0} - y_0) \triangle \mathcal{P}_{\phi(y_0)}\right)\cap B(\mathbf{0},K)$ also converges in mean to 0 since it follows a Poisson distribution with mean
$$N \int_{B(y_0,N^{-1/d}K)}|\phi_N(x)-\phi(y_0)|{\rm d}x\leq N \int_{B(y_0,N^{-1/d}K)}(|\phi_N(x)-\phi(x)|+|\phi(x)-\phi(y_0)|){\rm d}x,$$
which converges to 0 as $y_0$ is also a Lebesgue point of $\phi$.
\end{proof}

\begin{lemma}\label{lem:second_order_conv}
Suppose $x_0\neq y_0$ are Lebesgue points of $f_1,\ldots,f_M$ and $\phi(x_0),\phi(y_0)>0$. Let $\mathcal{P}_{\phi(x_0)}$ and $\mathcal{P}_{\phi(y_0)}$ be independent homogeneous Poisson processes with intensity $\phi(x_0)$ and $\phi(y_0)$ respectively.
Suppose $\xi$ is translation invariant, almost surely stabilizing on $\mathcal{P}_{\phi(x_0)}$ and $\mathcal{P}_{\phi(y_0)}$.
Suppose the moment condition:
$$\sup_{N \geq 1}\mathbb{E}\left[\left|\xi\left(x_0,x_0+ N^{1/d}(\mathcal{P}_{N\phi_N}^{y_0} - x_0)\right)\xi\left(y_0,y_0+ N^{1/d}(\mathcal{P}_{N\phi_N}^{x_0} - y_0)\right)\right|^p\right]<\infty$$
is satisfied for some $p>1$. Then:
$$\xi\left(x_0,x_0+ N^{1/d}(\mathcal{P}_{N\phi_N}^{y_0} - x_0)\right)\xi\left(y_0,y_0+ N^{1/d}(\mathcal{P}_{N\phi_N}^{x_0} - y_0)\right) \to \xi\left(\mathbf{0},\mathcal{P}_{\phi(x_0)}\right)\xi\left(\mathbf{0},\mathcal{P}_{\phi(y_0)}\right)$$
in distribution and in expectation.
\end{lemma}
\begin{proof}
Consider the coupling of $\mathcal{P}_{N\phi_N}$ and $\mathcal{P}_{\phi(x_0)},\mathcal{P}_{\phi(y_0)}$ in the previous lemma. The same argument in Lemma~\ref{lem:local_prop} shows
$$\begin{aligned}
\mathbb{P}\left[\left|\xi\left(x_0,x_0+ N^{1/d}(\mathcal{P}_{N\phi_N}^{y_0} - x_0)\right) - \xi\left(\mathbf{0},\mathcal{P}_{\phi(x_0)}\right) \right| > \varepsilon\right]\to 0,
\end{aligned}$$
$$\begin{aligned}
\mathbb{P}\left[\left|\xi\left(y_0,y_0+ N^{1/d}(\mathcal{P}_{N\phi_N}^{x_0} - y_0)\right) - \xi\left(\mathbf{0},\mathcal{P}_{\phi(y_0)}\right) \right| > \varepsilon\right]\to 0.
\end{aligned}$$
Hence
$$\xi\left(x_0,x_0+ N^{1/d}(\mathcal{P}_{N\phi_N}^{y_0} - x_0)\right)\xi\left(y_0,y_0+ N^{1/d}(\mathcal{P}_{N\phi_N}^{x_0} - y_0)\right) \overset{d}{\longrightarrow}
\xi\left(\mathbf{0},\mathcal{P}_{\phi(x_0)}\right)\xi\left(\mathbf{0},\mathcal{P}_{\phi(y_0)}\right).$$
Convergence in expectation again follows from uniform integrability.
\end{proof}

\begin{lemma}[{Palm theory for Poisson processes \citep[Theorem 1.6]{Penrose2003graph}}]\label{lem:palm}
Suppose $N >0$, $j\in\mathbb{N}$, and $h(\Y,\X)$ is a bounded measurable function defined on all pairs of the form $(\Y,\X)$ where $\X\subset \R^d$ is finite, and $\Y\subset\X$, satisfying $h(\Y,\X)=0$ when $\Y$ does not contain $j$ elements. Then
$$\E \sum_{\Y\subset \mathcal{P}_{N f}} h(\Y,\mathcal{P}_{N f}) = \frac{N^j}{j!}\E h(x_j',x_j'\cup \mathcal{P}_{N f}),$$
where the sum on the left-hand side is over all subsets $\Y\subset \mathcal{P}_{N f}$, and on the right-hand side $x_j'$ is a set of $j$ i.i.d.~observations from $f$, independent of $\mathcal{P}_{N f}$.
\end{lemma}

\subsubsection{Proof of Theorem \ref{thm:normal_alternative}}

Write $g_N (\Delta_i) :=2\sum_{p=1}^M \frac{N_p}{N}K(\Delta_i,p) - \sum_{p,q=1}^M \frac{N_pN_q}{N^2} K(p,q)$. By standard U-statistics projection theory \citep[Lemma D.4]{deb2020kernel},
there exists a constant $C$ such that for $n\geq 2$,
$$\E\left[\left(\frac{1}{\sqrt{n}(n-1)}\sum_{i\neq j }K(\Delta_i,\Delta_j) - \frac{1}{\sqrt{n}}\sum_{i=1}^n g_N(\Delta_i) \right)^2\Big| n\right] \leq \frac{C}{n}.$$
Since $n\sim {\rm Poisson}(N)$, $\mathbb{P}(n<2)$ decays exponentially as $N\to\infty$, and recall that we set $\frac{1}{\sqrt{n}(n-1)}\sum_{i\neq j }K(\Delta_i,\Delta_j)$ to 0 when $n<2$. This implies
$$\E\left[\left(\frac{1}{\sqrt{n}(n-1)}\sum_{i\neq j }K(\Delta_i,\Delta_j) - \frac{1}{\sqrt{n}}\sum_{i=1}^n  g_N(\Delta_i) \right)^2\right] \to0,\ {\rm as\ }N\to\infty.$$
Hence $\frac{1}{n-1}\sum_{i\neq j }K(\Delta_i,\Delta_j) =\sum_{i=1}^n  g_N(\Delta_i)  + o_p(\sqrt{n})$.
It suffices to derive the asymptotic distribution of 
\begin{equation}\label{eq:Uproj}
H_n :=\frac{1}{\sqrt{N}} \sum \limits_{i=1}^n \frac{1}{d_i} \sum \limits_{j:(Z_i,Z_j) \in \emgn} K(\Delta_i,\Delta_j)-\frac{1}{\sqrt{N}}\sum_{i=1}^n  g_N(\Delta_i).
\end{equation}

Let $\mathcal{F}_n=\sigma(n,Z_1,\ldots,Z_n)$ be the $\sigma$-algebra generated by the unlabelled data and the number of total observations. The strategy of the proof goes as follows:
First, we show that given $\mathcal{F}_n$, $H_n$ --- centered by its conditional mean --- converges conditionally to a normal distribution with constant variance, i.e., $H_n - \mathbb{E}(H_n|\mathcal{F}_n)\mid\mathcal{F}_n \overset{d}{\to}N(0,\kappa_1^2)$.
Second, we show that $\mathbb{E}(H_n|\mathcal{F}_n) - \mathbb{E}(H_n)\overset{d}{\to}N(0,\kappa_2^2)$.
Finally, a simple argument using the characteristic function yields that $H_n - \mathbb{E}(H_n) \overset{d}{\to}N(0,\kappa_1^2+\kappa_2^2)$. Observe that, for $t \in \R$,
$$\begin{aligned}
\mathbb{E}\left[e^{it\left(H_n - \mathbb{E}(H_n)\right)} \right] &=\mathbb{E}\left[e^{it\left(H_n - \mathbb{E}(H_n|\mathcal{F}_n)\right)} e^{it\left( \mathbb{E}(H_n|\mathcal{F}_n) - \mathbb{E}(H_n)\right)} \right]\\
&=\mathbb{E} \left[\E \left[ e^{it\left(H_n - \mathbb{E}(H_n|\mathcal{F}_n)\right)} \big| \mathcal{F}_n\right] e^{it\left( \mathbb{E}(H_n|\mathcal{F}_n) - \mathbb{E}(H_n)\right)}\right].
\end{aligned}$$
Note that for any $t\in\R$, by continuous mapping,
$$\E \left[ e^{it\left(H_n - \mathbb{E}(H_n|\mathcal{F}_n)\right)} \big| \mathcal{F}_n\right] e^{it\left( \mathbb{E}(H_n|\mathcal{F}_n) - \mathbb{E}(H_n)\right)}\overset{d}{\longrightarrow}e^{-\frac{1}{2}\kappa_1^2t }\cdot  e^{it \cdot N(0,\kappa_2^2)}.$$
Hence, by choosing an almost sure representative and applying the dominated convergence theorem, we have
$$\mathbb{E} \left[\E \left[ e^{it\left(H_n - \mathbb{E}(H_n|\mathcal{F}_n)\right)} \big| \mathcal{F}_n\right] e^{it\left( \mathbb{E}(H_n|\mathcal{F}_n) - \mathbb{E}(H_n)\right)}\right]\to
\mathbb{E}\left[e^{-\frac{1}{2}\kappa_1^2t }\cdot  e^{it \cdot N(0,\kappa_2^2)} \right]=e^{-\frac{1}{2}(\kappa_1^2+\kappa_2^2)t}.$$
This implies $H_n - \mathbb{E}(H_n) \overset{d}{\to}N(0,\kappa_1^2+\kappa_2^2)$.

\noindent \textbf{Step 1.} We wil first show that $H_n - \mathbb{E}(H_n|\mathcal{F}_n)\mid \mathcal{F}_n \overset{d}{\to}N(0,\kappa_1^2)$.

We first compute the limiting variance of $H_n$ given $\mathcal{F}_n$. Observe that

\begin{subequations}
\begin{align}
{\rm Var}(H_n|\mathcal{F}_n)&=\frac{1}{N}\sum_{i\neq j:(Z_i,Z_j)\in\emgn}\frac{1}{d_i^2}{\rm Var}(K(\Delta_i,\Delta_j)|\mathcal{F}_n)\label{eq:subeq1}\\
&\quad +\frac{1}{N}\sum_{i\neq j:(Z_i,Z_j),(Z_j,Z_i)\in\emgn}\frac{1}{d_id_j} {\rm Var}(K(\Delta_i,\Delta_j)|\mathcal{F}_n)\label{eq:subeq2}\\
&\quad +\frac{1}{N} \sum_{\substack{i,j,l{\rm\ distinct:}\\ (Z_i,Z_j),(Z_i,Z_l)\in\emgn}} \frac{1}{d_i^2}{\rm Cov}(K(\Delta_i,\Delta_j),K(\Delta_i,\Delta_l)|\mathcal{F}_n\label{eq:subeq3})\\
&\quad +\frac{1}{N} \sum_{\substack{i,j,l{\rm\ distinct:}\\ (Z_i,Z_j),(Z_j,Z_l)\in\emgn}} \frac{1}{d_id_j}{\rm Cov}(K(\Delta_i,\Delta_j),K(\Delta_j,\Delta_l)|\mathcal{F}_n)\label{eq:subeq4}\\
&\quad +\frac{1}{N} \sum_{\substack{i,j,l{\rm\ distinct:}\\ (Z_i,Z_j),(Z_l,Z_i)\in\emgn}} \frac{1}{d_id_l}{\rm Cov}(K(\Delta_i,\Delta_j),K(\Delta_l,\Delta_i)|\mathcal{F}_n)\label{eq:subeq5}\\
&\quad +\frac{1}{N} \sum_{\substack{i,j,l{\rm\ distinct:}\\ (Z_i,Z_j),(Z_l,Z_j)\in\emgn}} \frac{1}{d_id_l}{\rm Cov}(K(\Delta_i,\Delta_j),K(\Delta_l,\Delta_j)|\mathcal{F}_n)\label{eq:subeq6}\\
&\quad -\frac{2}{N}\sum_{i\neq j:(Z_i,Z_j)\in\emgn} \frac{1}{d_i}{\rm Cov}(K(\Delta_i,\Delta_j),g_N(\Delta_i)|\mathcal{F}_n)\label{eq:subeq7}\\
&\quad -\frac{2}{N}\sum_{i\neq j:(Z_i,Z_j)\in\emgn} \frac{1}{d_i}{\rm Cov}(K(\Delta_i,\Delta_j),g_N(\Delta_j)|\mathcal{F}_n)\label{eq:subeq8}\\
&\quad +\frac{1}{N}\sum_{i=1}^n {\rm Var}(g_N(\Delta_i)|\mathcal{F}_n).\label{eq:subeq9}
\end{align}
\label{eq:decom_var}
\end{subequations}
We will show the convergence of each term above.
Let $\phi_N(z)=\sum_{i=1}^M \frac{N_i}{N} f_i(z)$ be the marginal density of $Z$, and
$\phi(z):=\sum_{i=1}^M \pi_i f_i(z)$ be its limit.
Then, for $p\in\{1,\ldots,M\}$, 
$$\begin{aligned}
\left|\frac{\frac{N_p}{N}f_p(z)}{\phi_N(z)} - \frac{\pi_p f_p(z)}{\phi(z)}\right| &=\left| \frac{\frac{N_p}{N}f_p(z) \phi(z) - \pi_p f_p(z) \phi_N(z) }{\phi(z)\phi_N(z)} \right|\\
&=\left|\sum_{i=1}^M \frac{\frac{N_p}{N}f_p(z) \pi_i f_i(z) - \pi_p f_p(z) \frac{N_i}{N}f_i(z) }{\phi(z)\phi_N(z)} \right| \\
& \leq \sum_{i=1}^M \frac{ |\frac{N_p}{N} - \pi_p|f_p(z) \pi_i f_i(z) + \pi_pf_p(z)|\pi_i - \frac{N_i}{N}| f_i(z) }{\phi(z)\phi_N(z)} \\
& = \sum_{i=1}^M  \left( \frac{|\frac{N_p}{N} - \pi_p|}{\frac{N_p}{N}} \frac{\frac{N_p}{N}f_p(z)}{\phi_N(z)}
\frac{\pi_i f_i(z)}{\phi(z)} +
\frac{|\pi_i-\frac{N_i}{N}|}{\frac{N_i}{N}}\frac{\pi_pf_p(z)}{\phi(z)}\frac{\frac{N_i}{N} f_i(z)}{\phi_N(z)}
\right)\\
&\leq \sum_{i=1}^M \left(\frac{|\frac{N_p}{N}-\pi_p|}{\frac{N_p}{N}} + \frac{|\pi_i-\frac{N_i}{N}|}{\frac{N_i}{N}} \right)\to 0
\end{aligned}$$
uniformly over $z\in\{z:\phi(z) > 0\}$. Observe that
$${\rm Var}(K(\Delta_i,\Delta_j)|\mathcal{F}_n)=\underbrace{\sum_{p,q=1}^M K^2(p,q)\frac{\frac{N_p}{N}f_p(Z_i)}{\phi_N(Z_i)}\frac{\frac{N_q}{N}f_q(Z_j)}{\phi_N(Z_j)} - \left(\sum_{p,q=1}^M K(p,q)\frac{\frac{N_p}{N}f_p(Z_i)}{\phi_N(Z_i)}\frac{\frac{N_q}{N}f_q(Z_j)}{\phi_N(Z_j)}\right)^2}_{{\rm denoted\ by\ } h_N(Z_i,Z_j)}.$$
By the uniform convergence of $\frac{\frac{N_i}{N}f_i(z)}{\phi_N(z)}$ to $\frac{\pi_i f_i(z)}{\phi(z)}$, we know that $h_N$ converges uniformly to $h$, a bounded function defined by: 
$$h(Z_i,Z_j):= \sum_{p,q=1}^M K^2(p,q)\frac{\pi_pf_p(Z_i)}{\phi(Z_i)}\frac{\pi_qf_q(Z_j)}{\phi(Z_j)} - \left(\sum_{p,q=1}^M K(p,q)\frac{\pi_pf_p(Z_i)}{\phi(Z_i)}\frac{\pi_qf_q(Z_j)}{\phi(Z_j)}\right)^2.$$
To show the convergence of the first term in the expansion of ${\rm Var}(H_n|\mathcal{F}_n)$, i.e.,  $$\frac{1}{N}\sum_{i\neq j:(Z_i,Z_j)\in\emgn}\frac{1}{d_i^2} \cdot h_N(Z_i,Z_j)=: f(Z),$$
where $Z:=(Z_1,\ldots,Z_n)$.
We will show that there exists a constant $\mu$ such that $\E f(Z) \to\mu$ and $\E f^2(Z) \to \mu^2$, which implies $f(Z)\overset{L^2}{\longrightarrow} \mu$.


By Palm theory (Lemma~\ref{lem:palm}),
\begin{equation}\label{eq:exp_first_term}
\E f(Z)=\int \E\left[\frac{1}{d^2(z,\mathcal{G}(\mathcal{P}_{N \phi_N}^z))}\sum_{y:(z,y)\in\mathcal{E}(\mathcal{G}(\mathcal{P}_{N \phi_N}^z))} h_N(z,y) \right]\phi_N(z){\rm d}z.
\end{equation}
To show the convergence of the integrand, we can first replace $h_N$ by its uniform limit $h$.
$$\begin{aligned}
&\quad \lim_{N\to\infty} \E\left[\frac{1}{d^2(z,\mathcal{G}(\mathcal{P}_{N \phi_N}^z))}\sum_{y:(z,y)\in\mathcal{E}(\mathcal{G}(\mathcal{P}_{N \phi_N}^z))} h_N(z,y) \right]\\
&=\lim_{N\to\infty} \E\left[\frac{1}{d^2(z,\mathcal{G}(\mathcal{P}_{N \phi_N}^z))}\sum_{y:(z,y)\in\mathcal{E}(\mathcal{G}(\mathcal{P}_{N \phi_N}^z))} h(z,y) \right].\\
\end{aligned}$$
Note that $(z,y)\in\mathcal{E}(\mathcal{G}(\mathcal{P}_{N \phi_N}^z))$ implies that $y$ is a neighbor of $z$, and Lemma~\ref{lem:coupl} together with the graph being stabilizing on $\mathcal{P}_{\phi(z)}$ implies  $\sup_{y:(z,y)\in\mathcal{E}(\mathcal{G}(\mathcal{P}_{N\phi_N}^z))}\|y-z\|=O_p\left( N^{-1/d} \right)$. 
If $z$ is a continuity point of $f_1,\cdots,f_M$, then it is also a continuity point of $h(z,\cdot )$. Hence we can further replace $y$ by $z$ to get
$$\begin{aligned}
&\quad\lim_{N\to\infty} \E\left[\frac{1}{d^2(z,\mathcal{G}(\mathcal{P}_{N \phi_N}^z))}\sum_{y:(z,y)\in\mathcal{E}(\mathcal{G}(\mathcal{P}_{N \phi_N}^z))} h(z,y) \right]\\
&=\lim_{N\to\infty} \E\left[\frac{1}{d^2(z,\mathcal{G}(\mathcal{P}_{N \phi_N}^z))}\sum_{y:(z,y)\in\mathcal{E}(\mathcal{G}(\mathcal{P}_{N \phi_N}^z))} h(z,z) \right]\\
&=\lim_{N\to\infty} \E\left[\frac{1}{d(z,\mathcal{G}(\mathcal{P}_{N \phi_N}^z))} \right]h(z,z) \\
&=\E\left[\frac{1}{d(\mathbf{0},\mathcal{G}(\mathcal{P}_{\phi(z)}^\mathbf{0}))} \right]h(z,z) = \E\left[\frac{1}{d(\mathbf{0},\mathcal{G}(\mathcal{P}_{1}^\mathbf{0}))} \right]h(z,z).
\end{aligned}$$
The last line follows from Lemma~\ref{lem:local_prop} and $\mathcal{G}$ being translation and scale invariant.
Now, by dominated convergence,
$$\E f(Z)\to\int \E\left[\frac{1}{d(\mathbf{0},\mathcal{G}(\mathcal{P}_{1}^\mathbf{0}))} \right]h(z,z) \phi(z){\rm d}z.$$

Now we consider $\E f^2(Z)$. Write $\xi(x,\X) := \sum_{y:(x,y)\in \mathcal{E}(\mathcal{G}(\X\cup\{x\}))}\frac{1}{d_x^2}h_N (x,y)$.
Then $f(Z) = \frac{1}{N}\sum_{x\in \mathcal{P}_{N \phi_N}}\xi(x,\mathcal{P}_{N \phi_N})$, and
$$\begin{aligned}
\E f^2(Z)&=\frac{1}{N^2} \E \left[ \sum_{x\in \mathcal{P}_{N \phi_N}}\xi^2 (x,\mathcal{P}_{N \phi_N}) \right]+\frac{2}{N^2} \E \left[ \sum_{ \{x,y\}\subset \mathcal{P}_{N \phi_N} } \xi(x,\mathcal{P}_{N \phi_N}) \xi(y,\mathcal{P}_{N \phi_N})\right].
\end{aligned}$$ 
Here $\sum_{\{x,y\}\subset\mathcal{P}_{N \phi_N}}$ means summing over all $\frac{n(n-1)}{2}$ size-2 subsets of $\mathcal{P}_{N \phi_N}$. The first term converges to 0 since $\xi$ is bounded and the number of points in $\mathcal{P}_{N \phi_N}$ follows ${\rm Poisson}(N)$. The second term, by Palm theory, equals
$$\int \E\left[ \xi(x,\mathcal{P}_{N \phi_N}^y)\xi(y,\mathcal{P}_{N \phi_N}^x)\right]\phi_N(x)\phi_N(y){\rm d}x{\rm d}y.$$
By the uniform convergence of $h_N$ to $h$, assuming $x,y$ are continuity point of $f_1,\ldots,f_M$, the integrand has a limit given by
$$\begin{aligned}
&\quad \lim_{N\to\infty}\E\left[ \xi(x,\mathcal{P}_{N \phi_N}^y)\xi(y,\mathcal{P}_{N \phi_N}^x)\right]\\
&=\lim_{N\to\infty}\E\left[  \sum_{w:(x,w)\in \mathcal{E}(\mathcal{G}(\mathcal{P}_{N \phi_N}^{x,y}))}\frac{1}{d_x^2}h_N (x,w)  \sum_{w:(y,w)\in \mathcal{E}(\mathcal{G}(\mathcal{P}_{N \phi_N}^{x,y}))}\frac{1}{d_y^2}h_N (y,w)\right]\\
&=\lim_{N\to\infty}\E\left[  \sum_{w:(x,w)\in \mathcal{E}(\mathcal{G}(\mathcal{P}_{N \phi_N}^{x,y}))}\frac{1}{d_x^2}h (x,x)  \sum_{w:(y,w)\in \mathcal{E}(\mathcal{G}(\mathcal{P}_{N \phi_N}^{x,y}))}\frac{1}{d_y^2}h (y,y)\right]\\
&=\lim_{N\to\infty}\E\left[ \frac{1}{d(x,\mathcal{G}(\mathcal{P}_{N\phi_N}^{x,y}))}  \frac{1}{d(y,\mathcal{G}(\mathcal{P}_{N\phi_N}^{x,y}))} \right]h (x,x)h (y,y) \\
&=\E\left[\frac{1}{d(\mathbf{0},\mathcal{G}(\mathcal{P}_{\phi(x)}^\mathbf{0}))} \frac{1}{d(\mathbf{0},\mathcal{G}(\mathcal{P}_{\phi(y)}^\mathbf{0}))}\right] h (x,x)h (y,y) = \E\left[\frac{1}{d(\mathbf{0},\mathcal{G}(\mathcal{P}_{1}^\mathbf{0}))}\right]^2h(x,x)h(y,y).\\
\end{aligned}$$
The last line follows from Lemma~\ref{lem:second_order_conv} and $\mathcal{G}$ being translation and scale invariant.
Now, by dominated convergence,
$$\E f^2(Z)\to \int  \E\left[\frac{1}{d(\mathbf{0},\mathcal{G}(\mathcal{P}_{1}^\mathbf{0}))}\right]^2h(x,x)h(y,y) \phi(x)\phi(y){\rm d}x{\rm d}y.$$
Hence $f(Z)$ converges in $L^2$ to $\int \E\left[\frac{1}{d(\mathbf{0},\mathcal{G}(\mathcal{P}_{1}^\mathbf{0}))} \right]h(z,z) \phi(z){\rm d}z$.

With the distribution of $\tilde{\Delta},\tilde{Z}$ defined in Section~\ref{subsec:def_eta},
if $\tilde{\Delta},\tilde{\Delta}'$ are independently drawn from the conditional distribution of $\tilde{\Delta} \mid\tilde{Z}=z$, then $h(z,z) = {\rm Var}[K(\tilde{\Delta},\tilde{\Delta}')|\tilde{Z}=z]$. Hence the first term in ${\rm Var}(H_n|\mathcal{F}_n)$ converges to:
$$\begin{aligned}
\frac{1}{N}\sum_{i\neq j:(Z_i,Z_j)\in\emgn}\frac{1}{d_i^2}{\rm Var}(K(\Delta_i,\Delta_j)|\mathcal{F}_n) \overset{L^2}{\longrightarrow}\E\left[\frac{1}{d(\mathbf{0},\mathcal{G}(\mathcal{P}_{1}^\mathbf{0}))} \right]\int {\rm Var}[K(\tilde{\Delta},\tilde{\Delta}')|\tilde{Z}=z]\phi(z){\rm d}z.
\end{aligned}$$
The same strategy applied to the second term \eqref{eq:subeq2} yields:
$$\begin{aligned}
&\quad \frac{1}{N}\sum_{i\neq j:(Z_i,Z_j),(Z_j,Z_i)\in\emgn}\frac{1}{d_id_j} {\rm Var}(K(\Delta_i,\Delta_j)|\mathcal{F}_n)\\
&\overset{L^2}{\longrightarrow}\mathbb{E}\left[ \sum_{y:(y,\mathbf{0}),(\mathbf{0},y)\in\mathcal{E}(\mathcal{G}(\mathcal{P}_1^\mathbf{0}))}\frac{1}{d\left(\mathbf{0},\mathcal{G}(\mathcal{P}_1^\mathbf{0})\right) d\left(y,\mathcal{G}(\mathcal{P}_1^\mathbf{0})\right)} \right]\int {\rm Var}[K(\tilde{\Delta},\tilde{\Delta}')|\tilde{Z}=z]\phi(z){\rm d}z.
\end{aligned}$$
The convergence of the third term \eqref{eq:subeq3} can be handled similarly. Note that
$$\begin{aligned}
\quad &{\rm Cov}(K(\Delta_i,\Delta_j),K(\Delta_i,\Delta_l)|\mathcal{F}_n)=\sum_{p,q,r=1}^M K(p,q)K(p,r)\frac{\frac{N_p}{N}f_p(Z_i)}{\phi_N(Z_i)}\frac{\frac{N_q}{N}f_q(Z_j)}{\phi_N(Z_j)}\frac{\frac{N_r}{N}f_r(Z_l)}{\phi_N(Z_l)}\\
&-\sum_{p,q=1}^M K(p,q)\frac{\frac{N_p}{N}f_p(Z_i)}{\phi_N(Z_i)}\frac{\frac{N_q}{N}f_q(Z_j)}{\phi_N(Z_j)}\sum_{p,r=1}^M K(p,r)\frac{\frac{N_p}{N}f_p(Z_i)}{\phi_N(Z_i)}\frac{\frac{N_r}{N}f_r(Z_l)}{\phi_N(Z_l)}.
\end{aligned}$$
Denote the right-hand side of the above display by $h_N(Z_i,Z_j,Z_l)$.
Then $h_N$ converges uniformly to $h$ defined as:
$$\begin{aligned}
\quad &h(Z_i,Z_j,Z_l)=\sum_{p,q,r=1}^M K(p,q)K(p,r)\frac{\pi_pf_p(Z_i)}{\phi(Z_i)}\frac{\pi_qf_q(Z_j)}{\phi(Z_j)}\frac{\pi_rf_r(Z_l)}{\phi(Z_l)}\\
&-\sum_{p,q=1}^M K(p,q)\frac{\pi_pf_p(Z_i)}{\phi(Z_i)}\frac{\pi_qf_q(Z_j)}{\phi(Z_j)}\sum_{p,r=1}^M K(p,r)\frac{\pi_pf_p(Z_i)}{\phi(Z_i)}\frac{\pi_rf_r(Z_l)}{\phi(Z_l)}.
\end{aligned}$$
With $\xi(z,\mathcal{P}_{N \phi_N}) :=  \frac{1}{d_x^2} \sum_{x\neq y:(z,x),(z,y)\in \mathcal{E}(\mathcal{G}(\mathcal{P}_{N \phi_N}))}h_N(z,x,y)$, \eqref{eq:subeq3} can be written as $$f(Z) = \frac{1}{N} \sum_{z\in \mathcal{P}_{N \phi_N}}\xi(z,\mathcal{P}_{N \phi_N}).$$
By Palm theory,
\begin{equation}\label{eq:third_term}
\E f(Z) = \int \E\left[\frac{1}{d^2(z,\mathcal{G}(\mathcal{P}_{N \phi_N}^z))}\sum_{x\neq y:(z,x),(z,y)\in\mathcal{E}(\mathcal{G}(\mathcal{P}_{N \phi_N}^z))} h_N(z,x,y) \right]\phi_N(z){\rm d}z.
\end{equation}
Since $h_N$ converges uniformly to $h$, using the same argument as before, if $z$ is a continuity point of $f_1,\ldots,f_M$, then
$$\begin{aligned}
&\quad \lim_{N\to\infty }\E\left[\frac{1}{d^2(z,\mathcal{G}(\mathcal{P}_{N \phi_N}^z))}\sum_{x\neq y:(z,x),(z,y)\in\mathcal{E}(\mathcal{G}(\mathcal{P}_{N \phi_N}^z))} h_N(z,x,y) \right]\\
&=\lim_{N\to\infty }\E\left[\frac{1}{d^2(z,\mathcal{G}(\mathcal{P}_{N \phi_N}^z))}\sum_{x\neq y:(z,x),(z,y)\in\mathcal{E}(\mathcal{G}(\mathcal{P}_{N \phi_N}^z))} h(z,x,y) \right]\\
&=\lim_{N\to\infty }\E\left[\frac{1}{d^2(z,\mathcal{G}(\mathcal{P}_{N \phi_N}^z))}\sum_{x\neq y:(z,x),(z,y)\in\mathcal{E}(\mathcal{G}(\mathcal{P}_{N \phi_N}^z))} h(z,z,z) \right]\\
&=\lim_{N\to\infty }\E\left[\frac{d(z,\mathcal{G}(\mathcal{P}_{N \phi_N}^z))(d(z,\mathcal{G}(\mathcal{P}_{N \phi_N}^z))-1)}{d^2(z,\mathcal{G}(\mathcal{P}_{N \phi_N}^z))}\right] h(z,z,z) \\
&=\E\left[1-\frac{1}{d(\mathbf{0},\mathcal{G}(\mathcal{P}_{\phi(z)}^\mathbf{0}))} \right]h(z,z,z)=\E\left[1-\frac{1}{d(\mathbf{0},\mathcal{G}(\mathcal{P}_{1}^\mathbf{0}))} \right]h(z,z,z).
\end{aligned}$$
Hence by dominated convergence theorem applied to \eqref{eq:third_term},
$$\E f(Z) \to \int \E\left[1-\frac{1}{d(\mathbf{0},\mathcal{G}(\mathcal{P}_{1}^\mathbf{0}))} \right]h(z,z,z)\phi(z){\rm d}z.$$ Note that
$$\begin{aligned}
\E f^2(Z)&=\frac{1}{N^2} \E \left[ \sum_{x\in \mathcal{P}_{N \phi_N}}\xi^2 (x,\mathcal{P}_{N \phi_N}) \right]+\frac{2}{N^2} \E \left[ \sum_{\{x,y\}\subset \mathcal{P}_{N \phi_N} } \xi(x,\mathcal{P}_{N \phi_N}) \xi(y,\mathcal{P}_{N \phi_N})\right].
\end{aligned}$$
The first term converges to 0 again as $\xi$ is bounded and the number of points in $\mathcal{P}_{N \phi_N}$ follows ${\rm Poisson}(N)$. The second term equals
$\int \E\left[ \xi(x,\mathcal{P}_{N \phi_N}^y)\xi(y,\mathcal{P}_{N \phi_N}^x)\right]\phi_N(x)\phi_N(y){\rm d}x{\rm d}y$ by Palm theory.
By the uniform convergence of $h_N$ to $h$, assuming $x,y$ are continuity point of $f_1,\ldots,f_M$, the integrand has a limit given by
$$\begin{aligned}
&\quad \lim_{N\to\infty}\E\left[ \xi(x,\mathcal{P}_{N \phi_N}^y)\xi(y,\mathcal{P}_{N \phi_N}^x)\right]\\
&=\lim_{N\to\infty}\E\left[  \sum_{u\neq v:(x,u),(x,v)\in \mathcal{E}(\mathcal{G}(\mathcal{P}_{N \phi_N}^{x,y}))}\frac{1}{d_x^2}h_N (x,u,v)  \sum_{u\neq v:(y,u),(y,v)\in \mathcal{E}(\mathcal{G}(\mathcal{P}_{N \phi_N}^{x,y}))}\frac{1}{d_y^2}h_N (y,u,v)\right]\\
&=\lim_{N\to\infty}\E\left[  \sum_{u\neq v:(x,u),(x,v)\in \mathcal{E}(\mathcal{G}(\mathcal{P}_{N \phi_N}^{x,y}))}\frac{1}{d_x^2}h (x,u,v)  \sum_{u\neq v:(y,u),(y,v)\in \mathcal{E}(\mathcal{G}(\mathcal{P}_{N \phi_N}^{x,y}))}\frac{1}{d_y^2}h (y,u,v)\right]\\
&=\lim_{N\to\infty}\E\left[  \sum_{u\neq v:(x,u),(x,v)\in \mathcal{E}(\mathcal{G}(\mathcal{P}_{N \phi_N}^{x,y}))}\frac{1}{d_x^2}h (x,x,x)  \sum_{u\neq v:(y,u),(y,v)\in \mathcal{E}(\mathcal{G}(\mathcal{P}_{N \phi_N}^{x,y}))}\frac{1}{d_y^2}h (y,y,y)\right]\\
&=\lim_{N\to\infty}\E\left[ \frac{d(x,\mathcal{G}(\mathcal{P}_{N\phi_N}^{x,y}))-1}{d(x,\mathcal{G}(\mathcal{P}_{N\phi_N}^{x,y}))}  \frac{d(y,\mathcal{G}(\mathcal{P}_{N\phi_N}^{x,y}))-1}{d(y,\mathcal{G}(\mathcal{P}_{N\phi_N}^{x,y}))} \right]h (x,x,x)h (y,y,y) \\
&=\E\left[\frac{d(\mathbf{0},\mathcal{G}(\mathcal{P}_{\phi(x)}^\mathbf{0}))-1}{d(\mathbf{0},\mathcal{G}(\mathcal{P}_{\phi(x)}^\mathbf{0}))} \frac{d(\mathbf{0},\mathcal{G}(\mathcal{P}_{\phi(y)}^\mathbf{0}))-1}{d(\mathbf{0},\mathcal{G}(\mathcal{P}_{\phi(y)}^\mathbf{0}))}\right] h (x,x,x)h (y,y,y) \\
&= \E\left[\frac{d(\mathbf{0},\mathcal{G}(\mathcal{P}_{1}^\mathbf{0}))-1}{d(\mathbf{0},\mathcal{G}(\mathcal{P}_{1}^\mathbf{0}))}\right]^2h(x,x,x)h(y,y,y).\\
\end{aligned}$$
The last line again follows from Lemma~\ref{lem:second_order_conv} and $\mathcal{G}$ being translation and scale invariant.
Now, by dominated convergence,
$$\E f^2(Z)\to \int  \E\left[\frac{d(\mathbf{0},\mathcal{G}(\mathcal{P}_{1}^\mathbf{0}))-1}{d(\mathbf{0},\mathcal{G}(\mathcal{P}_{1}^\mathbf{0}))}\right]^2h(x,x,x)h(y,y,y)\phi(x)\phi(y){\rm d}x{\rm d}y.$$
Hence $f(Z)$ converges in $L^2$ to $\int \E\left[\frac{d(\mathbf{0},\mathcal{G}(\mathcal{P}_{1}^\mathbf{0}))-1}{d(\mathbf{0},\mathcal{G}(\mathcal{P}_{1}^\mathbf{0}))}\right]h(x,x,x)\phi(x){\rm d}x$.

Therefore, \eqref{eq:subeq3} converges in $L^2$ to
$$\begin{aligned}
&\quad \frac{1}{N} \sum_{\substack{i,j,l{\rm\ distinct:}\\ (Z_i,Z_j),(Z_i,Z_l)\in\emgn}} \frac{1}{d_i^2}{\rm Cov}(K(\Delta_i,\Delta_j),K(\Delta_i,\Delta_l)|\mathcal{F}_n)\\
&\overset{L^2}{\longrightarrow} \E\left[1-\frac{1}{d(\mathbf{0},\mathcal{G}(\mathcal{P}_{1}^\mathbf{0}))} \right]\int h(z,z,z)\phi(z){\rm d}z\\
&=\E\left[1-\frac{1}{d(\mathbf{0},\mathcal{G}(\mathcal{P}_{1}^\mathbf{0}))} \right]\int {\rm Cov}(K(\tilde{\Delta},\tilde{\Delta}'),K(\tilde{\Delta},\tilde{\Delta}'')|\tilde{Z}=z) \phi(z){\rm d}z,
\end{aligned}$$
where $\tilde{\Delta},\tilde{\Delta}',\tilde{\Delta}''$ are independently drawn from the conditional distribution of $\tilde{\Delta} \mid \tilde{Z}=z$.

The other terms in the right side of~\eqref{eq:decom_var} converge similarly:
{\footnotesize $$\begin{aligned}
&\quad \eqref{eq:subeq4}\overset{L^2}{\longrightarrow} \mathbb{E}\left[ \sum_{x\neq y:(x,\mathbf{0}),(\mathbf{0},y)\in\mathcal{E}(\mathcal{G}(\mathcal{P}_1^\mathbf{0}))}\frac{1}{d\left(x,\mathcal{G}(\mathcal{P}_1^\mathbf{0})\right) d\left(\mathbf{0},\mathcal{G}(\mathcal{P}_1^\mathbf{0})\right)} \right]\int {\rm Cov}(K(\tilde{\Delta},\tilde{\Delta}'),K(\tilde{\Delta}',\tilde{\Delta}'')|\tilde{Z}=z) \phi(z){\rm d}z,
\end{aligned}$$
$$\begin{aligned}
&\quad\eqref{eq:subeq5}\overset{L^2}{\longrightarrow}  \mathbb{E}\left[ \sum_{x\neq y:(\mathbf{0},x),(y,\mathbf{0})\in\mathcal{E}(\mathcal{G}(\mathcal{P}_1^\mathbf{0}))}\frac{1}{d\left(\mathbf{0},\mathcal{G}(\mathcal{P}_1^\mathbf{0})\right) d\left(y,\mathcal{G}(\mathcal{P}_1^\mathbf{0})\right)} \right]\int {\rm Cov}(K(\tilde{\Delta},\tilde{\Delta}'),K(\tilde{\Delta}'',\tilde{\Delta})|\tilde{Z}=z) \phi(z){\rm d}z,
\end{aligned}$$
$$\begin{aligned}
&\quad\eqref{eq:subeq6}\overset{L^2}{\longrightarrow} \mathbb{E}\left[ \sum_{x\neq y:(x,\mathbf{0}),(y,\mathbf{0})\in\mathcal{E}(\mathcal{G}(\mathcal{P}_1^\mathbf{0}))}\frac{1}{d\left(x,\mathcal{G}(\mathcal{P}_1^\mathbf{0})\right) d\left(y,\mathcal{G}(\mathcal{P}_1^\mathbf{0})\right)} \right]\int {\rm Cov}(K(\tilde{\Delta},\tilde{\Delta}'),K(\tilde{\Delta}'',\tilde{\Delta}')|\tilde{Z}=z) \phi(z){\rm d}z,
\end{aligned}$$
$$\eqref{eq:subeq7}\overset{L^2}{\longrightarrow}-2\int {\rm Cov}(K(\tilde{\Delta},\tilde{\Delta}'),g(\tilde{\Delta})|\tilde{Z}=z) \phi(z){\rm d}z,$$
$$\eqref{eq:subeq8}\overset{L^2}{\longrightarrow}
-2\int {\rm Cov}(K(\tilde{\Delta},\tilde{\Delta}'),g(\tilde{\Delta}')|\tilde{Z}=z) \phi(z){\rm d}z,$$
$$\eqref{eq:subeq9} \overset{L^2}{\longrightarrow} \E\left[{\rm Var}(g(\tilde{\Delta})|\tilde{Z}) \right].$$}
where $g (\Delta)=2\sum_{p=1}^M \pi_p K(\Delta,p) - \sum_{p,q=1}^M \pi_p \pi_qK(p,q)$ is the uniform limit in $\Delta$ of $g_N(\Delta)$.

Combining the results above, we have ${\rm Var}(H_n|\mathcal{F}_n)\overset{L^2}{\longrightarrow}\kappa_1^2$, where
$$\begin{aligned}
\kappa_1^2&=(g_1+g_3)\int {\rm Var}[K(\tilde{\Delta},\tilde{\Delta}')|\tilde{Z}=z]\phi(z){\rm d}z\\
&\quad +(3-2g_1 - 2g_3 + g_2)\int {\rm Cov}(K(\tilde{\Delta},\tilde{\Delta}'),K(\tilde{\Delta},\tilde{\Delta}'')|\tilde{Z}=z) \phi(z){\rm d}z\\
&\quad -4\int {\rm Cov}(K(\tilde{\Delta},\tilde{\Delta}'),g(\tilde{\Delta})|\tilde{Z}=z) \phi(z){\rm d}z \; + \; \E\left[{\rm Var}(g(\tilde{\Delta})|\tilde{Z}) \right],
\end{aligned}$$
where when simplifying the coefficient of $\int {\rm Cov}(K(\tilde{\Delta},\tilde{\Delta}'),K(\tilde{\Delta},\tilde{\Delta}'')|\tilde{Z}=z) \phi(z){\rm d}z$, we have used that
$$\begin{aligned}
&\quad \mathbb{E}\left[ \sum_{x\neq y:(x,\mathbf{0}),(\mathbf{0},y)\in\mathcal{E}(\mathcal{G}(\mathcal{P}_1^\mathbf{0}))}\frac{1}{d\left(x,\mathcal{G}(\mathcal{P}_1^\mathbf{0})\right) d\left(\mathbf{0},\mathcal{G}(\mathcal{P}_1^\mathbf{0})\right)} \right]\\
&=\mathbb{E}\left[ \sum_{x:(x,\mathbf{0})\in\mathcal{E}(\mathcal{G}(\mathcal{P}_1^\mathbf{0}))}\frac{d\left(\mathbf{0},\mathcal{G}(\mathcal{P}_1^\mathbf{0})\right) - 1_{(\mathbf{0},x),(x,\mathbf{0})\in \mathcal{E}(\mathcal{G}(\mathcal{P}_1^\mathbf{0}))}}{d\left(x,\mathcal{G}(\mathcal{P}_1^\mathbf{0})\right) d\left(\mathbf{0},\mathcal{G}(\mathcal{P}_1^\mathbf{0})\right)} \right],
\end{aligned}$$
and
$$\mathbb{E}\left[ \sum_{x:(x,\mathbf{0})\in\mathcal{E}(\mathcal{G}(\mathcal{P}_1^\mathbf{0}))} \frac{1}{d\left(x,\mathcal{G}(\mathcal{P}_1^\mathbf{0})\right) } \right] = 1$$
since it is the limit of the expectation of $\frac{1}{n}\sum_{x\in\X_n} \sum_{y:(y,x)\in \emgn}\frac{1}{d_y}\equiv1$ \citep[Lemma 3.2]{penrose2003weak}. In particular, if the null hypothesis holds true (in which case $\tilde{\Delta}\indep \tilde{Z}$), then $\kappa_1^2$ reduces exactly to the asymptotic variance of the numerator derived in Theorem~\ref{thm:DistFree}, which is distribution-free.

If $\kappa_1^2 = 0$, then $H_n - \mathbb{E}(H_n|\mathcal{F}_n)\mid \mathcal{F}_n \overset{d}{\to}N(0,\kappa_1^2)$ follows trivially from the convergence of  the conditional variance of $H_n - \mathbb{E}(H_n|\mathcal{F}_n)$ to 0. We will assume $\kappa_1^2 >0$ and prove the convergence using Stein's
method based on dependency graphs (see Theorem~\ref{thm:depGraph} below). First note that $H_n - \mathbb{E}(H_n|\mathcal{F}_n) = \sum_{i=1}^n V_i$, where
$$\begin{aligned}
V_i &= \frac{1}{\sqrt{N}} \frac{1}{d_i} \sum \limits_{j:(Z_i,Z_j) \in \emgn} K(\Delta_i,\Delta_j)-\frac{1}{\sqrt{N}} g_N(\Delta_i) \\
&- \E\left[\frac{1}{\sqrt{N}} \frac{1}{d_i} \sum \limits_{j:(Z_i,Z_j) \in \emgn} K(\Delta_i,\Delta_j)-\frac{1}{\sqrt{N}} g_N(\Delta_i)\Big|\mathcal{F}_n \right].
\end{aligned}$$
Construct a graph $\mathcal{D}(\mathcal{G}_n)$ on $\{V_1,\ldots, V_n\}$ as follows: for $i\neq j$, there is an edge between $V_i$ and $V_j$ in $\mathcal{D}(\mathcal{G}_n)$ if and only if there is a path of length $\leq 2$ joining $Z_i$ and $Z_j$ in $\mathcal{G}_n$ (ignoring the direction of edges in $\mathcal{G}_n$).
Then for any pair of disjoint sets $\Gamma_1,\Gamma_2\subset\{V_1,\ldots,V_n\}$ such that no edge in $\mathcal{D}(\mathcal{G}_n)$ has one endpoint in $\Gamma_1$ and the other in $\Gamma_2$, $\{V_i\}_{i\in \Gamma_1}$ is independent of $\{V_i\}_{i\in \Gamma_2}$ conditioned on $\mathcal{F}_n$. This implies that $\mathcal{D}(\mathcal{G}_n)$ is a dependency graph \citep{chen2004normal}.
Suppose $t_n$ is the maximum degree of $\mathcal{G}_n$.
Then the maximal degree in $\mathcal{D}(\mathcal{G}_n)$ has an upper bound of $t_n^2$.
We will use the following CLT for dependency graph (with $p=3$).

\begin{theorem}[{\citealp[Theorem 2.7]{chen2004normal}}]\label{thm:depGraph}
Suppose $\{X_i\}_{i\in\mathcal{V}}$ are random variables indexed by vertices of a dependency graph, whose maximal degree is $D$. Set $W=\sum_{i\in\mathcal{V}}X_i$. If ${\rm Var}(W)=1$, $\E X_i=0$, and $\E |X_i|^p\leq \theta^p$ for $i\in\mathcal{V}$, then we have
$$\sup_{z\in\R}|\mathbb{P}(W\leq z) - \Phi(z)|\leq 75D^{5(p-1)}|\mathcal{V}|\theta^p.$$
\end{theorem}
Note that
$$\begin{aligned}
&\quad \left|\frac{1}{\sqrt{N}} \frac{1}{d_i} \sum \limits_{j:(Z_i,Z_j) \in \emgn} K(\Delta_i,\Delta_j)-\frac{1}{\sqrt{N}} g_N(\Delta_i)\right|^3\\
&\leq \left( \left|\frac{1}{\sqrt{N}} \frac{1}{d_i} \sum \limits_{j:(Z_i,Z_j) \in \emgn} K(\Delta_i,\Delta_j)\right|^3 + \left|\frac{1}{\sqrt{N}} g_N(\Delta_i)\right|^3 \right)(1+1)(1+1)\\
&\leq 4 \left( \left|\frac{1}{\sqrt{N}}  \|K\|_\infty \right|^3 + \left|\frac{1}{\sqrt{N}} \|g_N\|_\infty\right|^3 \right)\;\lesssim\; \frac{1}{N^{3/2}}.
\end{aligned}$$
Hence
$$\begin{aligned}
\sup_{z\in\R}\left|\mathbb{P}\left(\frac{H_n - \mathbb{E}(H_n|\mathcal{F}_n)}{{\rm Var}(H_n|\mathcal{F}_n)}\leq z\Big|\mathcal{F}_n\right) - \Phi(z)\right| \;& =\; \sup_{z\in\R}\left|\mathbb{P}\left(\sum_{i=1}^n \frac{V_i}{{\rm Var}(H_n|\mathcal{F}_n)} \leq z\Big|\mathcal{F}_n\right) - \Phi(z)\right| \\
&\lesssim 75 D^{5(3-1)} \cdot n\cdot \frac{1}{N^{3/2}({\rm Var}(H_n|\mathcal{F}_n))^3}\\
&\leq \frac{75 t_n^{20}}{n^{1/2}({\rm Var}(H_n|\mathcal{F}_n))^3}\cdot \frac{n^{3/2}}{N^{3/2}}\overset{p}{\to}0.
\end{aligned}$$

\noindent \textbf{Step 2.} We will now show that  $\mathbb{E}(H_n|\mathcal{F}_n) - \mathbb{E}(H_n)\overset{d}{\to}N(0,\kappa_2^2)$.

Note that $\mathbb{E}(H_n|\mathcal{F}_n)$ can be written as:
$$\begin{aligned}
\frac{1}{\sqrt{N}}\sum_{i=1}^n \frac{1}{d_i}\sum_{j:(Z_i,Z_j)\in\emgn}\sum_{p,q=1}^M K(p,q)\frac{\frac{N_p}{N}f_p(Z_i)}{\phi_N(Z_i)}\frac{\frac{N_q}{N}f_q(Z_j)}{\phi_N(Z_j)} - \frac{1}{\sqrt{N}} \sum_{i=1}^n\sum_{p=1}^M g_N(p)\frac{\frac{N_p}{N}f_p(Z_i)}{\phi_N(Z_i)}.
\end{aligned}$$
Since $\sqrt{N}\left(\frac{N_p}{N} - \pi_p\right)\to 0$, $p=1,\ldots,M$, the difference between 
$\frac{\frac{N_p}{N}f_p(z)}{\phi_N(z)}$ and $\frac{\pi_pf_p(z)}{\phi(z)}$, and the difference between $$g_N (z) :=2\sum_{p=1}^M \frac{N_p}{N}K(z,p) - \sum_{p,q=1}^M \frac{N_pN_q}{N^2} K(p,q)$$ and $$g (z) :=2\sum_{p=1}^M \pi_p K(z,p) - \sum_{p,q=1}^M \pi_p \pi_qK(p,q)$$ are $o(\frac{1}{\sqrt{N}})$, uniformly in $z$.
Hence the above statistic has the same limit in probability as:
$$\begin{aligned}
\frac{1}{\sqrt{N}}\sum_{i=1}^n \frac{1}{d_i}\sum_{j:(Z_i,Z_j)\in\emgn}\sum_{p,q=1}^M K(p,q)\frac{\pi_pf_p(Z_i)}{\phi(Z_i)}\frac{\pi_qf_q(Z_j)}{\phi(Z_j)} - \frac{1}{\sqrt{N}}  \sum_{i=1}^n\sum_{p=1}^M g(p)\frac{\pi_pf_p(Z_i)}{\phi(Z_i)},
\end{aligned}$$
We can write the above statistic as $\frac{1}{\sqrt{N}}\sum_{x\in\mathcal{P}_{N \phi_N} }\xi(x,\mathcal{P}_{N \phi_N})$, where
\begin{equation}\label{eq:xi}
\xi(x,\mathcal{P}_{N \phi_N}) :=   \frac{1}{d_x}\sum_{y:(x,y)\in\mathcal{E}(\mathcal{P}_{N \phi_N})}\sum_{p,q=1}^M K(p,q)\frac{\pi_pf_p(x)}{\phi(x)}\frac{\pi_qf_q(y)}{\phi(y)} -  \sum_{p=1}^M g(p)\frac{\pi_pf_p(x)}{\phi(x)}.
\end{equation}

From the radius of stabilization for $\mathcal{G}$, we also have the same {\it radius of stabilization} \citep{Penrose2005normal,Penrose2007measures} $R=R(N, x)$ for $\xi(x,\mathcal{P}_{N\phi_N})$ for which the points outside $B(x,N^{-1/d} R)$ cannot impact the value of $\xi(x,\mathcal{P}_{N\phi_N})$, i.e., 
\begin{equation}\label{eq:Xi-RNx}\xi\left(x,[\mathcal{P}_{N\phi_N}\cap B(x,N^{-1/d} R)]\cup\X  \right)=\xi\left(x,\mathcal{P}_{N\phi_N}\cap B(x,N^{-1/d} R) \right),
\end{equation}
for all finite $\X\subset A\backslash B(x,N^{-1/d} R)$.

The power-law stabilization of a function $\xi: (x,\X) \mapsto \R$, where $x \in \R^d$ and $\X \subset \R^d$ is a finite set, can be defined in a similar fashion as the power-law stabilization of a graph (see \eqref{eq:power_law_graph}):
\begin{defn}[Power-law stabilizing functions \citep{Penrose2005normal}]\label{defn:stabilizing_function}
Let $\xi(x,\X)$ be a real-valued function defined for any $x\in\R^d$ and any finite set $\X\subset\R^d$ such that $R \equiv R(N,x)$ is a radius of stabilization as defined in~\eqref{eq:Xi-RNx}. Let
$$\tau(t):=\sup_{N\geq 1,x\in A}\mathbb{P}[R(N,x) > t], \qquad \mbox{for}\;\;t >0.$$
Then, $\xi$ is said to be:
\begin{enumerate}
\item {\it power-law stabilizing} of order $q$ with respect to $\phi_N$ if $\sup_{t\geq 1} t^q \tau(t) < \infty$,
\item {\it exponentially stabilizing} with respect to $\phi_N$ if $\limsup_{t\to\infty} t^{-1}\log \tau(t) <0$.
\end{enumerate}
\end{defn}

Let $$T:=\frac{1}{\sqrt{N}}\sum_{x\in \mathcal{P}_{N\phi_N}} \xi(x,\mathcal{P}_{N\phi_N})$$ with $ \xi(x,\mathcal{P}_{N\phi_N})$ defined in \eqref{eq:xi}.
The goal is to show the variance of $T$ converges. By Palm theory,
$$\E [T] =\sqrt{N} \int \E\xi(x,\mathcal{P}_{N\phi_N})\phi(x){\rm d}x,$$
$$\begin{aligned}
\E[T^2]&=\frac{1}{N}\E\left(\sum_{x\in \mathcal{P}_{N\phi_N}}\xi^2(x,\mathcal{P}_{N\phi_N})\right) + \frac{2}{N}\E\left[\sum_{\{x,y\}\subset \mathcal{P}_{N\phi_N}}\xi(x,\mathcal{P}_{N\phi_N})\xi(y,\mathcal{P}_{N\phi_N}) \right]\\
&=\int \E \xi^2(x,\mathcal{P}_{N\phi_N})\phi_N(x){\rm d}x +N \int \E\left[\xi(x,\mathcal{P}_{N\phi_N}^y)\xi(y,\mathcal{P}_{N\phi_N}^x) \right]\phi_N(x)\phi_N(y){\rm d}x{\rm d}y
\end{aligned}$$
Hence,
\begin{equation}\label{eq:var_T}\begin{aligned}
{\rm Var}[T]&=\int \E\xi^2(x,\mathcal{P}_{N\phi_N})\phi_N(x){\rm d}x \\
& + N \int \left( \E\left[\xi(x,\mathcal{P}_{N\phi_N}^y)\xi(y,\mathcal{P}_{N\phi_N}^x) \right] - \E \xi(x,\mathcal{P}_{N\phi_N}) \E\xi(y,\mathcal{P}_{N\phi_N}) \right)\phi_N(x)\phi_N(y){\rm d}x{\rm d}y.
\end{aligned}\end{equation}
The convergence of the first term above is easy to see.
Similar to previous arguments, by the coupling lemma \ref{lem:coupl} and the $\mathcal{P}_\lambda$ stabilization assumption, $(x,y)\in\mathcal{E}(\mathcal{G}(\mathcal{P}_{N \phi_N}^z))$ implies that $y$ is a neighbor of $x$, and lies within $O_p\left(N^{-1/d}\right)$ distance to $x$.
If $x$ is a continuity point of $f_q$, for $q=1,\ldots,M$, then together with boundedness of
$$\xi^2(x,\mathcal{P}_{N\phi_N}) = \left(\frac{1}{d_x}\sum_{y:(x,y)\in\mathcal{E}(\mathcal{P}_{N \phi_N}^x)}\sum_{p,q=1}^M K(p,q)\frac{\pi_pf_p(x)}{\phi(x)}\frac{\pi_qf_q(y)}{\phi(y)} -  \sum_{p=1}^M g(p)\frac{\pi_pf_p(x)}{\phi(x)} \right)^2,$$
we can 
replace $f_q(y)$ and $\phi(y)$ in the expression above by $f_q(x)$ and $\phi(x)$.
Thus,
{\small $$\begin{aligned}
\lim_{N\to\infty}\E\xi^2(x,\mathcal{P}_{N\phi_N})&=\lim_{N\to\infty}\E\left[\left(\frac{1}{d_x}\sum_{y:(x,y)\in\mathcal{E}(\mathcal{P}_{N \phi_N}^x)}\sum_{p,q=1}^M K(p,q)\frac{\pi_pf_p(x)}{\phi(x)}\frac{\pi_qf_q(y)}{\phi(y)} -  \sum_{p=1}^M g(p)\frac{\pi_pf_p(x)}{\phi(x)} \right)^2\right]\\
&=\lim_{N\to\infty}\E\left[ \left(\frac{1}{d_x}\sum_{y:(x,y)\in\mathcal{E}(\mathcal{P}_{N \phi_N}^x)}\sum_{p,q=1}^M K(p,q)\frac{\pi_pf_p(x)}{\phi(x)}\frac{\pi_qf_q(x)}{\phi(x)} -  \sum_{p=1}^M g(p)\frac{\pi_pf_p(x)}{\phi(x)} \right)^2\right] \\
&=\left(\sum_{p,q=1}^M K(p,q)\frac{\pi_pf_p(x)}{\phi(x)}\frac{\pi_qf_q(x)}{\phi(x)} -  \sum_{p=1}^M g(p)\frac{\pi_pf_p(x)}{\phi(x)} \right)^2.
\end{aligned}$$}
To show the convergence of the second term in \eqref{eq:var_T}, we first perform a change of variable $y=x+N^{-1/d}z$:
{\footnotesize \begin{equation}\label{eq:second_term}\begin{aligned}
&\quad N \int \left( \E\Big[\xi(x,\mathcal{P}_{N\phi_N}^y)\xi(y,\mathcal{P}_{N\phi_N}^x) \right] - \E \xi(x,\mathcal{P}_{N\phi_N}) \E\xi(y,\mathcal{P}_{N\phi_N}) \Big)\phi_N(x)\phi_N(y){\rm d}x{\rm d}y=\\
&\int \Big( \E\big[\underbrace{\xi(x,\mathcal{P}_{N\phi_N}^{x+N^{-\frac{1}{d}}z})}_{{\rm denoted\ by\ }X} \underbrace{\xi(x+N^{-\frac{1}{d}}z,\mathcal{P}_{N\phi_N}^x)}_{Z} \big] - \E \underbrace{\xi(x,\mathcal{P}_{N\phi_N})}_{X'} \E\underbrace{\xi(x+N^{-\frac{1}{d}}z,\mathcal{P}_{N\phi_N})}_{Z'} \Big)\phi_N(x)\phi_N(x+N^{-\frac{1}{d}}z){\rm d}x{\rm d}z.
\end{aligned}\end{equation}}
When taking the limit as $N\to\infty$, we want to pass the limit inside the integral. We will show that the integrand is small for large $|z|$, uniformly in $N$. More specifically, there exists $C_1>0$ free of $N,x,z$ such that
\begin{equation}\label{eq:XZ_XZ}
|\E[XZ]-\E[X']\E[Z']|\leq C_1(|z|^{-d-1/C_1}\wedge 1),
\end{equation}
for all $N\geq 1,x\in\R^d,z\in\R^d$, where $X,Z,X',Z'$ are defined in \eqref{eq:second_term}.
Let $\tilde{X} :=X1_{R(x,N)\leq |z|/3}$, $\tilde{Z} := Z1_{R(x+N^{-1/d}z,N)\leq |z|/3}$.
Then $\tilde{X}$ and $\tilde{Z}$ are independent because they are determined by the points of $\mathcal{P}_{N\phi_N}$
in $B(x,N^{-\frac{1}{d}}|z|/3)$, $B(x+N^{-\frac{1}{d}}z,N^{-\frac{1}{d}}|z|/3)$ respectively.
So $\E[\tilde{X}\tilde{Z}]=\E[\tilde{X}]\E[\tilde{Z}]$, and
$$\E[XZ] = \E[\tilde{X}]\E[\tilde{Z}] + \E[\tilde{X}(Z-\tilde{Z})] + \E[(X-\tilde{X})Z]$$
while
$$\E[X']\E[Z'] = \E[\tilde{X}]\E[\tilde{Z}] + \E[\tilde{X}]\E[Z'-\tilde{Z}] + \E[X'-\tilde{X}]\E[Z'].$$
Note that the absolute values of $X,Z,\tilde{X},\tilde{Z},X',Z'$ are all bounded by a constant $C>0$ depending only on the kernel $K$, $\pi$, and $\sup_n  \{\frac{t_n}{r_n}\}$. Hence
$$|\E[(X-\tilde{X})Z]|\leq C\E |X-\tilde{X}| =C\E|X|1_{R(x,N)>|z|/3}\leq C^2 \mathbb{P}(R(x,N)>|z|/3).$$
If we have power-law stabilization of order $q>d$, then there exists $C_2>0$ such that $C^2 \mathbb{P}(R(x,N)>|z|/3)\leq C_2(|z|^{-d+1/C_2}\wedge 1)$. Similarly, $\E[\tilde{X}(Z-\tilde{Z})]$, $ \E[\tilde{X}]\E[Z'-\tilde{Z}] $, and $\E[(X-\tilde{X})Z]$ are all controlled by $C_2(|z|^{-d+1/C_2}\wedge 1)$. Hence \eqref{eq:XZ_XZ} is proved.
This implies that the contribution of $|z|>K$ to the integral in \eqref{eq:second_term} is sufficiently small for large $K$, and hence we can pass the limit into the integral by dominated convergence. It remains to compute the limit of the integrand:
$$ \Big( \E\big[\xi(x,\mathcal{P}_{N\phi_N}^{x+N^{-\frac{1}{d}}z})\xi(x+N^{-\frac{1}{d}}z,\mathcal{P}_{N\phi_N}^x)\big] - \E \xi(x,\mathcal{P}_{N\phi_N}) \E\xi(x+N^{-\frac{1}{d}}z,\mathcal{P}_{N\phi_N})\Big)\phi_N(x)\phi_N(x+N^{-\frac{1}{d}}z)$$
as $N\to\infty$.
We show each term above converges. Suppose $x$ is a continuity point of $f_p$, $p=1,\ldots,M$.
With the same coupling technique used previously, we have:
$$\begin{aligned}
\lim_{N\to\infty}\E \xi(x,\mathcal{P}_{N\phi_N}) &= \lim_{N\to\infty}\E \left[\frac{1}{d_x}\sum_{y:(x,y)\in\mathcal{E}(\mathcal{P}_{N \phi_N}^x)}\sum_{p,q=1}^M K(p,q)\frac{\pi_pf_p(x)}{\phi(x)}\frac{\pi_qf_q(y)}{\phi(y)} -  \sum_{p=1}^M g(p)\frac{\pi_pf_p(x)}{\phi(x)} \right]\\
&=\lim_{N\to\infty}\E \left[\frac{1}{d_x}\sum_{y:(x,y)\in\mathcal{E}(\mathcal{P}_{N \phi_N}^x)}\sum_{p,q=1}^M K(p,q)\frac{\pi_pf_p(x)}{\phi(x)}\frac{\pi_qf_q(x)}{\phi(x)} -  \sum_{p=1}^M g(p)\frac{\pi_pf_p(x)}{\phi(x)} \right] \\
&= \sum_{p,q=1}^M K(p,q)\frac{\pi_pf_p(x)}{\phi(x)}\frac{\pi_qf_q(x)}{\phi(x)} -  \sum_{p=1}^M g(p)\frac{\pi_pf_p(x)}{\phi(x)}.
\end{aligned}$$
Lemma~\ref{lem:coupl} shows that $N^{1/d}(\mathcal{P}_{N\phi_N}-x)$ can be locally approximated by $\mathcal{P}_{\phi(x)}$,
and thus $N^{1/d}(\mathcal{P}_{N \phi_N}^{x+N^{-1/d}z}-x)$ can be approximated by $\mathcal{P}_{\phi(x)}\cup\{z\}$.
Since the graph is translation invariant and stabilizing on $\mathcal{P}_{\phi(x)}^\mathbf{0}$, it is also stabilizing on $\mathcal{P}_{\phi(x)}^z$ \citep[Lemma 3.3]{penrose2003weak}.
Therefore,
$$\begin{aligned}
&\quad \lim_{N\to\infty}\E \xi(x+N^{-1/d}z,\mathcal{P}_{N\phi_N}) \\
&= \lim_{N\to\infty}\E \Bigg[\frac{1}{d_{x+N^{-1/d}z}}\sum_{y:(x+N^{-1/d}z,y)\in\mathcal{E}(\mathcal{P}_{N \phi_N}^{x+N^{-1/d}z})}\sum_{p,q=1}^M K(p,q)\frac{\pi_pf_p(x+N^{-1/d}z)}{\phi(x+N^{-1/d}z)}\frac{\pi_qf_q(y)}{\phi(y)} \\
&\quad -  \sum_{p=1}^M g(p)\frac{\pi_pf_p(x+N^{-1/d}z)}{\phi(x+N^{-1/d}z)} \Bigg]\\
&=\lim_{N\to\infty}\E \left[\frac{1}{d_{x+N^{-1/d}z}}\sum_{y:(x+N^{-1/d}z,y)\in\mathcal{E}(\mathcal{P}_{N \phi_N}^{x+N^{-1/d}z})}\sum_{p,q=1}^M K(p,q)\frac{\pi_pf_p(x)}{\phi(x)}\frac{\pi_qf_q(x)}{\phi(x)} -  \sum_{p=1}^M g(p)\frac{\pi_pf_p(x)}{\phi(x)} \right] \\
&= \sum_{p,q=1}^M K(p,q)\frac{\pi_pf_p(x)}{\phi(x)}\frac{\pi_qf_q(x)}{\phi(x)} -  \sum_{p=1}^M g(p)\frac{\pi_pf_p(x)}{\phi(x)}.
\end{aligned}$$
Further, as shown in the proof of \citet[Lemma 3.3]{penrose2003weak}, if a geometric graph is translation invariant and stabilizing on $\mathcal{P}_1^\mathbf{0}$, then it is also stabilizing on $\mathcal{P}_1\cup\{\mathbf{0},z\}$. Hence,
\
$$\begin{aligned}
&\quad \lim_{N\to\infty}\E\big[\xi(x,\mathcal{P}_{N\phi_N}^{x+N^{-1/d}z})\xi(x+N^{-1/d}z,\mathcal{P}_{N\phi_N}^x)\big] \\
&= \lim_{N\to\infty}\E \Bigg[ \Big(\frac{1}{d_x}\sum_{y:(x,y)\in\mathcal{E}(\mathcal{P}_{N \phi_N}^{x,x+N^{-1/d}z})}\sum_{p,q=1}^M K(p,q)\frac{\pi_pf_p(x)}{\phi(x)}\frac{\pi_qf_q(x)}{\phi(x)} -  \sum_{p=1}^M g(p)\frac{\pi_pf_p(x)}{\phi(x)} \Big)\times\\
&\Big( \frac{1}{d_{x+N^{-1/d}z}}\sum_{y:(x+N^{-1/d}z,y)\in\mathcal{E}(\mathcal{P}_{N \phi_N}^{x,x+N^{-1/d}z})}\sum_{p,q=1}^M K(p,q)\frac{\pi_pf_p(x)}{\phi(x)}\frac{\pi_qf_q(x)}{\phi(x)} -  \sum_{p=1}^M g(p)\frac{\pi_pf_p(x)}{\phi(x)}\Big)\Bigg]\\
&=\left(\sum_{p,q=1}^M K(p,q)\frac{\pi_pf_p(x)}{\phi(x)}\frac{\pi_qf_q(x)}{\phi(x)} -  \sum_{p=1}^M g(p)\frac{\pi_pf_p(x)}{\phi(x)}\right)^2.
\end{aligned}$$
Hence the integrand in \eqref{eq:second_term} converges pointwise to 0. Combining the above convergence results we have:
$${\rm Var}[T]\to \int \left(\sum_{p,q=1}^M K(p,q)\frac{\pi_pf_p(x)}{\phi(x)}\frac{\pi_qf_q(x)}{\phi(x)} -  \sum_{p=1}^M g(p)\frac{\pi_pf_p(x)}{\phi(x)} \right)^2\phi(x){\rm d}x=:\kappa_2^2.$$
If $\kappa_2^2=0$, then $T-\E T\overset{d}{\to} N(0,\kappa_2^2)$ is trivial. Suppose $\kappa_2^2 > 0$. We will use the following proposition.

\begin{proposition}\label{prop:second_normal_positive_var}
Suppose that $\xi(x,\X)$ is a bounded function, defined for any $x\in\R^d$ and any finite set $\X\subset\R^d$, which is power-law stabilizing with respect to $\phi_N$ with order $q_0>16d$ (see Definition~\ref{defn:stabilizing_function}).
Suppose $\phi(\cdot)$ is bounded and has bounded support.
Let $T=\frac{1}{\sqrt{N}}\sum_{x\in \mathcal{P}_{N\phi_N}} \xi(x,\mathcal{P}_{N\phi_N})$.
Suppose ${\rm Var}(T)\to \sigma^2>0$. Then $T-\E T\overset{d}{\to} N(0,\sigma^2)$.
\end{proposition}

\begin{proof}
We first cover $ {\rm supp}(\phi)$ by  $V=V(N)$ cubes of the form $Q=\prod_{i=1}^d [j_i s_N,(j_i +1) s_N]$,
where $s_N := N^{-1/d}\rho_N$. Here $\rho_N$ will be taken as $N^{a}$ such that
$a<\frac{1}{8d}$ and $(aq_0-1)>1$.
Since $\phi$ has bounded support, $V(N) = O\left( \frac{1}{s_N^d}\right)=O\left(N \rho_N^{-d}\right)$.
The points in $\mathcal{P}_{N\phi_N}$ can be labelled as:
$$X_{i,j},\ 1\leq i\leq V,\ 1\leq j\leq M_i,\ M_i\sim {\rm Poisson}\left( N \int_{Q_i}\phi_N(x){\rm d}x\right)$$ where $Q_1,\ldots,Q_V$ is an enumeration of the cubes that cover ${\rm supp}(\phi)$ as described above. As $\phi$ is bounded, each $f_i\ (i=1,\ldots,M)$ is also bounded, and consequently $\|\phi_N\|_\infty$ is bounded by a constant for all $N$, and $N \int_{Q_i}\phi_N (x){\rm d}x\leq N \|\phi_N\|_\infty \int_{Q_i}{\rm d}x = N s_N^d \|\phi_N\|_\infty=  \rho_N^d \|\phi_N\|_\infty$. If we write $\xi_{i,j}:=\xi(X_{i,j},\mathcal{P}_{N\phi_N})$, then we have the following bound on the $q$-th moment (i.e., for any random variable $X$, $ \|X\|_q := (\E |X|^q)^{1/q}$):
$$\left\| \sum_{j=1}^{M_i}|\xi_{i,j}| \right\|_q\leq \|\xi\|_\infty \|M_i\|_q\leq  \|\xi\|_\infty \left\|{\rm Poisson}\left( \rho_N^d \|\phi_N\|_\infty\right)\right\|_q.$$
Using the fact that $\|{\rm Poisson}(\lambda)\|_q \leq \lambda \exp(q/(2\lambda))$ (see~\cite{ahle2021sharp}),
and $ \|\phi_N\|_\infty$ being bounded,
we have:
\begin{equation}\label{eq:q-moment-bound}
\left\| \sum_{j=1}^{M_i}|\xi_{i,j}| \right\|_q\lesssim  \rho_N^d.
\end{equation}
Let $R_{i,j}$ denote the radius of stabilization of $\xi$ at $X_{i,j}$, $E_{i,j}:=\{R_{i,j}\leq \rho_N\}$, $E:=\cap_{i=1}^V \cap_{j=1}^{M_i} E_{i.j}$. Then by Palm theory \cite[Theorem 1.6]{Penrose2003graph},
$$\mathbb{P}[E^c] \leq \E\left[\sum_{i=1}^V \sum_{j=1}^{M_i} 1_{E_{i,j}^c}\right] = N \int \mathbb{P}[R(x,N)>\rho_N]\phi_N(x){\rm d}x\leq N \tau(\rho_N).$$
Since $\xi$ is power law stabilizing of order $q_0$, $\tau(t) \leq \frac{C}{t^{q_0}}$ for some $C>0$.
Since $\rho_N = N^a$, we have
$\mathbb{P}[E^c]\leq \frac{C}{N^{aq_0-1}}$.
Consider
$$T' := \frac{1}{\sqrt{N}}\sum_{i=1}^{V(N)}\sum_{j=1}^{M_i}\xi(X_{i,j},\mathcal{P}_{N\phi_N})\cdot 1_{E_{i,j}}$$
which equals $T$ on $E$, and is a sum of $V(N)$ ``near independent" random variables. We show that ${\rm Var}[T']$ is close to ${\rm Var}[T]$. By H\"older's inequality, for $q>2$:
$$\begin{aligned}
\|T-T'\|_2 &\leq \|T-T'\|_q\mathbb{P}[E^c]^{1-\frac{2}{q}}\\
&\leq \left(\|T\|_q + \|T'\|_q \right) \left( \frac{C}{N^{aq_0 - 1}} \right)^{1-\frac{2}{q}}.
\end{aligned}$$ 
To bound $\|T\|_q$ and $\|T'\|_q$, we use \eqref{eq:q-moment-bound} and $V(N) =O\left( N \rho_N^{-d}\right)$ to get, for $q\geq 2$,
$$\|T'\|_q = \frac{1}{\sqrt{N}}\left\| \sum_{i=1}^{V(N)}\sum_{j=1}^{M_i}\xi(X_{i,j},\mathcal{P}_{N\phi_N})\cdot 1_{E_{i,j}}\right\|_q\lesssim 
\frac{1}{\sqrt{N}}\cdot N \rho_N^{-d}\cdot \rho_N^d=\sqrt{N}.$$
Similarly $\|T\|_q\lesssim \sqrt{N}$. Hence
$$\|T-T'\|_2\lesssim \sqrt{N} \cdot \left( \frac{C}{N^{aq_0 - 1}} \right)^{1-\frac{2}{q}}.$$
Hence,
$$\begin{aligned}
|{\rm Var}(T) - {\rm Var}(T')|&=|{\rm Var}(T-T') + 2{\rm Cov}(T-T',T')|\\
&\leq \|T-T'\|_2^2 + 2\|T-T'\|_2\|T'\|_2\\
&\lesssim N   \left( \frac{C}{N^{ap - 1}} \right)^{2-\frac{4}{q}} + 2 \cdot \sqrt{N}  \left( \frac{C}{N^{ap - 1}} \right)^{1-\frac{2}{q}}\cdot \sqrt{N}.
\end{aligned}$$
We can choose $q$ large enough such that $(aq_0-1)(1-2/q)-1>0$. Then we have $|{\rm Var}(T) - {\rm Var}(T')|\to 0$.
In particular, ${\rm Var}(T')\to\sigma^2 > 0$. Also, $|\E T - \E T'|\leq \E|T-T'|\leq \|T-T'\|_2 \to 0$.

Write $\left( T' - \E[T']\right)/\sqrt{{\rm Var}(T')}=\sum_{i=1}^V S_i$, where
$$S_i = \frac{1}{\sqrt{N {\rm Var}(T')}}\left(\sum_{j=1}^{M_i}\xi(X_{i,j},\mathcal{P}_{N\phi_N})\cdot 1_{E_{i,j}} - \E\left[\sum_{j=1}^{M_i}\xi(X_{i,j},\mathcal{P}_{N\phi_N})\cdot 1_{E_{i,j}} \right] \right).$$
Construct a dependency graph on $\{S_1,\ldots, S_V\}$ as follows: for $i\neq j$, there is an edge between $S_i$ and $S_j$ if $d(Q_i,Q_j)\leq 2N^{-1/d}\rho_N$, where $d(Q_i,Q_j):=\inf\{|x-y|:x\in Q_i,y\in Q_j\}$.
By definition of the radius of stabilization $R(x,N)$, the value of $S_i$ is determined by the restriction of $\mathcal{P}_{N\phi_N}$ to the $N^{-1/d}\rho_N$-neighborhood of the cube $Q_i$, i.e., $\{x\in \R^d:|x-y|\leq N^{-1/d}\rho_N\textrm{ for some }y\in Q_i\}$.
By the independence property of the Poisson process, 
for any pair of disjoint sets  $\Gamma_1,\Gamma_2\subset\{S_1,\ldots,S_V\}$ such that no edge has one endpoint in $\Gamma_1$ and the other in $\Gamma_2$, $\Gamma_1$ is independent of $\Gamma_2$. Since the number of cubes in $Q_1,\ldots,Q_V$ that are at most $2N^{-1/d}\rho_N$ distant from a given cube is bounded by $5^d$, it follows that the maximal degree of the dependency graph is bounded by $D\leq 5^d$, a constant.

By \eqref{eq:q-moment-bound} and ${\rm Var}(T')\to\sigma^2 > 0$,
$$\E[|S_i|^3]\lesssim \frac{1}{N^{3/2}}\rho_N^{3d} = \frac{1}{N^{3(0.5-ad)}}.$$
Hence we can take $p=3$ and $\theta=\frac{C'}{N^{0.5-ad}}$ for some constant $C'$ in Theorem~\ref{thm:depGraph} to yield
$$\begin{aligned}
\sup_{z\in\R}\left|\mathbb{P}\left(\frac{ T' - \E[T']}{\sqrt{{\rm Var}(T')}} \leq z\right) - \Phi(z)\right| &\leq 75D^{5(p-1)}|\mathcal{V}|\theta^p\\
&\lesssim  N \rho_N^{-d} \cdot  \frac{1}{N^{3(0.5-ad)}}  = \frac{1}{N^{0.5-4ad}}\to 0,
\end{aligned}$$
as $N\to\infty$, which implies $\frac{ T' - \E[T']}{\sqrt{{\rm Var}(T')}} \overset{d}{\to} N(0,1)$.
Since $|{\rm Var}(T) - {\rm Var}(T')|\to 0$, $|\E T - \E T'|\to 0$, and $T,T'$ coincide on the set $E$ whose probability converges to 1, we have $\frac{ T - \E[T]}{\sqrt{{\rm Var}(T)}} \overset{d}{\to} N(0,1)$.
\end{proof}

\subsection{Proof of Theorem~\ref{thm:shrinking_alternative}}\label{pf:shrinking_alternative}
Note that Lemmas~\ref{lem:coupl} and \ref{lem:coupl2} still hold if $\mathcal{P}_{N\phi_N}$ is replaced by $\mathcal{P}_{N\phi_N^N}$ and $\phi(x_0)$ is replaced by $f(x_0)$, assuming $f_1^N,\ldots,f_M^N$ are equicontinuous at $x_0$ and $f_i^N(x_0)\to f(x_0)$ as $N\to\infty$.
The slight change in the proof is that \eqref{eq:coupl} now becomes:
{\small
$$
N \int_{B(x_0,N^{-1/d}K)}|\phi_N^N (x)- f(x_0)|{\rm d}x\leq N \int_{B(x_0,N^{-1/d}K)}(|\phi_N^N(x)-\phi_N^N(x_0)|+|\phi_N^N(x_0)-f(x_0)|){\rm d}x\to 0,
$$}
since $\phi_N^N$ is equicontinuous at $x_0$ and $\phi_N^N(x_0)\to f(x_0)$.
The rest of the proof follows almost verbatim from the proof of Theorem~\ref{thm:normal_alternative},
by replacing every $f_i$ by $f_i^N$, every $\phi_N$ by $\phi_N^N$, every $\phi$ by $\phi^N:= \sum_{i=1}^M \pi_i f_i^N$, and applying an extra limit (as $N \to \infty$):
$$\left\{\begin{aligned}
&\int |\phi^N(z)-f(z)|{\rm d}z \to 0,\\
&\frac{\pi_if_i^N(z)}{\phi^N(z)} \to \pi_i,\qquad \mbox{for a.e.}\;  z \in {\rm supp}(f), \qquad \mbox{for } \; i=1,\ldots,M.
\end{aligned}\right.$$
The limiting variance $\kappa_1^2$ reduces to the numerator of the asymptotic null variance in \eqref{eq:null_variance}, which is distribution-free. The limiting variance
$$\kappa_2^2=\int \left(\sum_{p,q=1}^M K(p,q)\pi_p \pi_q -  \sum_{p=1}^M g(p)\pi_p \right)^2 f(x){\rm d}x=0.$$ Recall that $g (\Delta)=2\sum_{i=1}^M \pi_i K(\Delta,i) - \sum_{i,j=1}^M \pi_i \pi_j K(i,j)$ was defined in Theorem~\ref{thm:normal_alternative}.
Hence, $\mathbb{E}(H_n|\mathcal{F}_n) - \mathbb{E}(H_n)\overset{d}{\to}N(0,\kappa_2^2) \equiv \delta_0$ as the variance of $\mathbb{E}(H_n|\mathcal{F}_n)$ converges to 0. Note that we need $q_0>d$ here to establish the convergence of ${\rm Var}\left( \E[H_n|\mathcal{F}_n]\right)$; see \eqref{eq:XZ_XZ}. Recall that the power-law stabilization with order $q_0>16d$ was used for proving Proposition~\ref{prop:second_normal_positive_var} which establishes the CLT assuming the variance converges to a non-zero quantity, and is not needed in this proof.

\subsection{Proof of Theorem~\ref{thm:detection_threshold}}\label{sec:detection_threshold}

Since
$$\frac{n\tilde{H}_n - \E[n\tilde{H}_n]}{\sqrt{n}\kappa_{1,{\rm null}}}\overset{d}{\to}N(0,1),$$
the local power of the test in \eqref{eq:test} is
$$\Phi\left( z_\alpha + \lim \frac{ \E[n\tilde{H}_n]}{\kappa_{1,{\rm null}}\sqrt{n}} \right).$$
As $\frac{n}{N} \overset{p}{\to} 1$, we know that $\frac{ \E[n\tilde{H}_n]}{\sqrt{n}}$ has the same limit in probability as the expectation of
$$\frac{1}{\sqrt{N}} \sum \limits_{i=1}^n \frac{1}{d_i} \sum \limits_{j:(Z_i,Z_j) \in \emgn} K(\Delta_i,\Delta_j)-\frac{1}{\sqrt{N}}\sum_{i=1}^n  g_N(\Delta_i).$$
Note that the expectation of the second term $\frac{1}{\sqrt{N}}\sum_{i=1}^n  g_N(\Delta_i)$ equals the expectation of the first term $\frac{1}{\sqrt{N}} \sum \limits_{i=1}^n \frac{1}{d_i} \sum \limits_{j:(Z_i,Z_j) \in \emgn} K(\Delta_i,\Delta_j)$ under the null, i.e., $\theta_1 = \theta_2$. Hence we study the first term as follows. With a $k$-NN graph and the discrete kernel, the expectation of the first term can be written as
$$\begin{aligned}
&\quad \frac{1}{k\sqrt{N}}\E \left[ \sum_{i,j:(Z_i,Z_j) \in \emgn} \left(1 - I(\Delta_i\neq \Delta_j)\right)  \right]\\
&=\sqrt{N} -  \frac{1}{k\sqrt{N}}\E \left[ \sum_{i,j:(Z_i,Z_j) \in \emgn} I(\Delta_i\neq \Delta_j)  \right]\\
&=\sqrt{N} -  \frac{1}{k\sqrt{N}}\E \left[ \sum_{x,y:(x,y) \in \mathcal{E}(\mathcal{P}_{N\phi_N^N})}
\left( \frac{\frac{N_1}{N}f(x|\theta_1)}{\phi_N^N(x)}\frac{\frac{N_2}{N}f(y|\theta_2)}{\phi_N^N(y)}
+ \frac{\frac{N_2}{N}f(x|\theta_2)}{\phi_N^N(x)}\frac{\frac{N_1}{N}f(y|\theta_1)}{\phi_N^N(y)} \right)  \right]\\
&=\sqrt{N} -  \frac{1}{k\sqrt{N}}N_1N_2\int_{S\times S} (f(x|\theta_1)f(y|\theta_2) + f(x|\theta_2)f(y|\theta_1))\rho_K^{\theta_1,\theta_2}(x,y){\rm d}x{\rm d}y,
\end{aligned}$$
where $\phi_N^N(x) = \frac{N_1}{N}f(x|\theta_1) + \frac{N_2}{N}f(x|\theta_2) $,
and $\rho_K^{\theta_1,\theta_2}(x,y) = \mathbb{P}\left((x,y)\in \mathcal{E}(\mathcal{P}_{N\phi_N^N}^{x,y})\right)$.
The last equality above follows from Palm theory.
Define $$\mu_N^{(S)}(\theta_1,\theta_2) := N^2 \int_{S\times S}(f(x|\theta_1)f(y|\theta_2) + f(x|\theta_2)f(y|\theta_1))\rho_K^{\theta_1,\theta_2}(x,y){\rm d}x{\rm d}y.$$
Then,
$$\begin{aligned}
\lim_{N \to \infty} \frac{ \E[n\tilde{H}_n]}{\sqrt{n}} &= \lim_{N \to \infty}  -\frac{N_1N_2}{kN^2} \frac{\mu_N^{(S)}(\theta_1,\theta_2) - \mu_N^{(S)}(\theta_1,\theta_1)}{\sqrt{N}}\\
&=\lim_{N \to \infty}  -\frac{N_1N_2}{kN^2}
\left( \frac{\varepsilon_N^\top \nabla \mu_N^{(S)}(\theta_1,\theta_1)}{\sqrt{N}} + \frac{1}{2}\frac{\varepsilon_N^\top {\rm H}\mu_N^{(S)}(\theta_1,\theta_1)\varepsilon_N}{\sqrt{N}} + \mathcal{R}_N \right),\\
\end{aligned}$$
where the gradient and Hessian are taken with respect to $\theta_2$.
The limit of the gradient and Hessian follows the argument in Appendix E in the supplementary file of \citet{BB20detection}.
Note that there is a typo in the limit of the Hessian term in \citet[Appendix E]{BB20detection}. In fact, with $\varepsilon_N = h N^{-1/4}$, we can show that
$$\frac{\varepsilon_N^\top {\rm H}\mu_N^{(S)}(\theta_1,\theta_1)\varepsilon_N}{\sqrt{N}}\to -2rk\E  \left[\frac{h^\top \nabla_{\theta_1}f(X|\theta_1)}{f(X|\theta_1)}\right]^2,$$
where $r=2\pi_1\pi_2$. The coefficient above is 2 instead of $3/2$ in the original version of \citet{BB20detection}.
The rest of the proof follows from the arguments at the beginning of \citet[Appendix B]{BB20detection}.
\qed

\subsection{Proof of Theorem~\ref{thm:extPham}}\label{sec:pfextPham}
The proof is similar to that in \citet{Pham1989permu}. 
But for completeness, we provide the entire proof here. The proof applies the method of moments, an idea that dates back to  Pafnutii Lvovich Chebyshev (1821–1894) \citep{Fisher2011HistoryCLT}. More specifically, we will show that for all $r\in\mathbb{N}$,
$$\E\left[ \left(\frac{\sum a_{ij}B_{ij}}{{\rm Var}(\sum a_{ij}B_{ij})} \right)^r \right]\to \E\left[ Z^r\right],$$
where $Z\sim N(0,1)$. This implies that for any subsequence that converges in distribution, the limiting distribution has the same moments as $N(0,1)$,
and therefore must be $N(0,1)$ as these moments uniquely determine the distribution $N(0,1)$ (see e.g.,~\cite[Theorem 30.1]{Billingsley2012Prob}).
Since weak convergence to $N(0,1)$ holds for any subsequence, it further implies that the entire sequence converges in distribution as weak convergence is metrizable.

In this proof, we will decompose $\left(\sum a_{ij}B_{ij} \right)^r$ into ``sums corresponding to different graphs" as in \citet{Bloemena1964graph}. The sums corresponding to most of the graphs will be negligible when compared to $\left[{\rm Var}(\sum a_{ij}B_{ij})\right]^r$. The remaining dominating terms will lead to the moments of $N(0,1)$.

We first define equivalent graphs. In this proof, a graph will be considered as a collection of edges (so it does not contain isolated vertices), and we will consider the edges in a graph to be different, each having a label (from $1,\ldots,r$) and a direction.
Two graphs are equivalent if they can be mapped to each other while keeping the direction and the labeling of the edges.
Multiple edges are allowed to exist between 2 vertices. For a graph $G$, denote by $|G|$ the number of its vertices, and arbitrarily label the vertices of $G$ as $1,\ldots,|G|$.
Then an edge in $G$ can be given by $(\mu,\nu)$, where $\mu,\nu\in\{1,\ldots,|G|\}$ are integers. Define
\begin{equation}\label{eq:def_Sigma_a_G}\Sigma (a,G) := {{\sum}'_{i_1,\ldots, i_{|G|}}} \prod_{(\mu,\nu)\in G} a_{i_\mu i_\nu},\end{equation}
where $\sum'$ means the summation is over distinct $i_1,\ldots,i_{|G|} \in \{1,\ldots, n\}$.

Denote by $G(r)$ the set of all graphs with $r$ edges (equivalent graphs will only be counted once). Then
\begin{equation}\label{eq:decomp_r_moment}(\sum a_{ij})^r = \sum_{G\in G(r)} \Sigma(a,G).\end{equation}
It can be seen from \eqref{eq:def_Sigma_a_G} that if $G$ has a self-loop, then $\Sigma (a,G)=0$ as $a_{ii}=b_{ii}=0$, $\forall i$.
Hence, we can ignore graphs in $G(r)$ that contain self-loops.
To understand \eqref{eq:decomp_r_moment},
for example, when $r=2$,
$$
\begin{aligned}
\left(\sum a_{ij}\right)^2 &= {\sum}' a_{ij} a_{kl}\\
& + {\sum}' a_{ij}a_{ik} + {\sum}'a_{ij}a_{ki} + {\sum}' a_{ij}a_{jk} + {\sum}' a_{ij}a_{kj}\\
&+ {\sum}' a_{ij}^2 + {\sum}' a_{ij}a_{ji}.
\end{aligned}$$
The sum in the first line corresponds to a graph with two isolated edges.
The sums in the second line correspond to graphs with two edges sharing exactly one vertex, but the edges can have different orientations.
The sums in the third line correspond to graphs with two edges sharing two vertices (which forms a cycle), with two possible orientations.

Now,
\begin{equation}\label{eq:expmoment}\begin{aligned}
\E\left(\sum a_{ij}B_{ij} \right)^r &=\sum_{G\in G(r)} \dsum_{i_1,\ldots, i_{|G|}} \E \prod_{(\mu,\nu)\in G} a_{i_\mu i_\nu} B_{i_\mu i_\nu}\\
&=\sum_{G\in G(r)} \E \left[\prod_{(\mu,\nu)\in G} B_{i_\mu i_\nu} \right] \Sigma(a,G)\\
&=\sum_{G\in G(r)}\frac{1}{n(n-1)\cdots(n-|G|+1)} \Sigma(b,G)\Sigma(a,G).
\end{aligned}\end{equation}
The second equality follows from the fact that each term $\prod_{(\mu,\nu)\in G}  B_{i_\mu i_\nu}$ has the same expectation, which is $ \frac{1}{n(n-1)\cdots(n-|G|+1)} \Sigma(b,G)$ as there are $n(n-1)\cdots(n-|G|+1)$ many terms in $\Sigma(b,G)$.

We first estimate the order of $\Sigma(b,G)$. Suppose $G$ has an isolated edge, say $(1,2)$.
By definition, $\Sigma (b,G) = {{\sum}'_{i_1,\ldots, i_{|G|}}} \prod_{(\mu,\nu)\in G} b_{i_\mu i_\nu}$ with the summation indices $i_1,\ldots,i_{|G|}$ required to be distinct.
If we relax the constraint for $i_1,i_2$, i.e., $i_1,i_2$ can freely take values from $1,\ldots,n$ while $i_3,\ldots,i_{|G|}$ are still required to be distinct, then the sum will be 0 as $\sum b_{ij} = 0$. But then we need to subtract back the terms which were originally not in $\Sigma (b,G)$ and came in because of relaxing the constraints on $i_1,i_2$. These terms are sum of a number of $\Sigma(b,G')$, where $G'$ is the graph obtained from $G$ by identifying vertex 1 or 2 or both with some vertex in $3,\ldots,|G|$.
In general, if $G$ has multiple isolated edges, $\Sigma (b,G)$ can be written as a linear combination of $\Sigma(b,G')$, where $G'$ can be obtained by sequentially identifying the vertices of isolated edges with other vertices in the graph, and $G'$ no longer has any isolated edge. Note that
$$|\Sigma(b,G')| = \left|  \dsum_{ i_1,\ldots, i_{|G'|}} \prod_{(\mu,\nu)\in G'} b_{i_\mu i_\nu} \right|\leq 
\prod_{(\mu,\nu)\in G'} \left(\sum_{i_1,\ldots, i_{|G'|}} |b_{i_\mu i_\nu}|^r \right)^{\frac{1}{r}},$$
where we have used the fact that for $C_{ij}\geq 0$, $1\leq i\leq I$, $1\leq j\leq J$,
$$\prod_{i=1}^I \left( \sum_{j=1}^J C_{ij}^r \right) \geq \left(\sum_{j=1}^J\prod_{i=1}^I C_{ij}\right)^r.$$
Since $\sum_{i_1,\ldots, i_{|G'|}} |b_{i_\mu i_\nu}|^r = n^{|G'|-2}\sum |b_{ij}|^r$ and by (B2), it is $O\left( n^{|G'|-r}\|b\|^r\right)$,
where $\|b\|$ is defined as $(\sum b_{ij}^2)^{1/2}$,
we have $|\Sigma(b,G')| = O\left(n^{|G'|-r}\|b\|^r\right)$.

Let $s$ be the number of isolated edges in $G$. Note that if $(1,2)$ is an isolated edge, identifying $1$ or $2$ with some vertex in $3,\ldots,|G|$ reduces the number of vertices by 1, and the number of isolated edges by at most 2. Identifying both $1$ and $2$ with some vertices in $3,\ldots,|G|$ reduces the number of vertices by 2, and the number of isolated edges by at most 3. Hence in either way, when the number of vertices is reduced by 1, the number of isolated edges is reduced by at most 2. Since $G'$ has no isolated edge, $|G'|\leq |G| - \lceil\frac{s}{2}\rceil$.
Combining the discussions above:
\begin{equation}\label{eq:orderb}
\Sigma (b,G) = O\left(n^{|G|- \lceil\frac{s}{2}\rceil-r}\cdot \|b\|^r\right).
\end{equation}

Next, we estimate the order of $\Sigma(a,G)$. Without loss of generality, we can suppose $\max |a_{ij}| = 1$.
Then, if $G$ is connected, by suppressing redundant edges in $G$ to a spanning tree $T$, we have:
$$\Sigma (|a|,G) := \dsum_{i_1,\ldots, i_{|G|}} \prod_{(\mu,\nu)\in G} |a_{i_\mu i_\nu}|\leq \dsum_{i_1,\ldots, i_{|G|}} \prod_{(\mu,\nu)\in T} |a_{i_\mu i_\nu}|.$$
Here $|a|$ means taking the absolute value of each $a_{ij}$.
Recall assumption (A1): $\max_i\sum_j |a_{ij}| = O(\max|a_{ij}|\cdot t_n)$.
Hence by sequentially summing over the index corresponding to a leaf,
$$\Sigma (|a|,G) \leq \dsum_{i_1,\ldots, i_{|G|}} \prod_{(\mu,\nu)\in T} |a_{i_\mu i_\nu}|\leq n t_n^{|G|-1}.$$
If $G$ has $c$ connected components,
we can relax the summation constraint ``$i_1,\ldots,i_{|G|}$ are distinct" to ``the indices within each component are distinct", which implies
$$\Sigma (|a|,G)\leq n^c t_n^{|G|-c}\leq  n^c t_n^{|G|}.$$
If among these $c$ connected components, $s$ of them are isolated edges,
then similar to the previous argument, using $\sum a_{ij}=0$, $\Sigma (a,G)$ can be written
as a linear combination of $\Sigma(a,G')$, where $G'$ can be obtained by sequentially identifying the vertices of isolated edges with other vertices in the graph, and $G'$ no longer has any isolated edge.
Note that if $(1,2)$ is an isolated edge, identifying $1$ or $2$ with some vertex in $3,\ldots,|G|$
reduces the number of connected components by 1 and the number of isolated edges by 1 or 2. Identifying both $1$ and $2$ with some vertices in $3,\ldots,|G|$ may reduce the number of isolated edges by 3, but in such a case the number of connected components is reduced by 2. Hence in either case, when the number of connected components is reduced by 1, the number of isolated edges is reduced by at most 2.
Because $G'$ no longer has any isolated edge, it has at most $c-\lceil\frac{s}{2}\rceil$ connected components.
 Therefore, 
$$\Sigma (a,G) = O\left( n^{c-\lceil\frac{s}{2}\rceil} t_n^{|G|} \right).$$
Recall assumption (A2): $\liminf \sum a_{ij}^2/(n\max a_{ij}^2 )>0$, which implies
\begin{equation}\label{eq:ordera}
\Sigma (a,G) = O\left(  n^{c-\lceil\frac{s}{2}\rceil-\frac{r}{2}} t_n^{ |G|}\cdot \|a\|^r \right).
\end{equation}
By considering the number of edges: $2(c-s)+s\leq r\Leftrightarrow c\leq \frac{r}{2} + \frac{s}{2}$,
we know
$$\Sigma (a,G)\leq n^{-1/2} t_n^{ |G|}\cdot \|a\|^r = o\left( \|a\|^r\right),$$
unless every connected component has at most 2 edges and $\lceil\frac{s}{2}\rceil =\frac{s}{2}$.
Together with $\Sigma (b,G) = O\left(n^{|G|- \lceil\frac{s}{2}\rceil-r}\cdot \|b\|^r\right)$ deducted previously, we have
$$ \frac{1}{n(n-1)\cdots(n-|G|+1)} \Sigma(b,G)\Sigma(a,G) = o\left( \frac{\|a\|^r  \|b\|^r}{n^r} \right),$$
unless $s=0$ and every connected component has at most 2 edges, which is then equivalent to all connected components having exactly 2 edges.

We now claim that $w_n \gtrsim \frac{\|a\| \|b\|}{n}$. Hence in the expansion \eqref{eq:expmoment} of $\E\left(\sum a_{ij}B_{ij}\right)^r$, the sum corresponding to these graphs are negligible compared to $w_n^r$.
To see the claim, note that $\sum a_{ij} = 0$ implies $\sum' a_{ij}a_{ik} = \sum a_{i+}^2 - \sum a_{ij}^2$, and similarly $\sum' b_{ij}b_{ik} = \sum b_{i+}^2 - \sum b_{ij}^2$. Hence,
\begin{equation}\label{eq:w_n}\begin{aligned}
w_n^2 &= \frac{4(n-2)}{n^4}\left(\sum a_{i+}^2\right)\left(\sum b_{i+}^2\right)\\
&\quad + \frac{2}{n^2} \left( \sum a_{ij}^2 - \frac{2}{n} \sum a_{i+}^2\right)\left( \sum b_{ij}^2 - \frac{2}{n} \sum b_{i+}^2\right)
+ \frac{4}{n^3} \left(\sum a_{ij}^2\right)\left(\sum b_{ij}^2\right).
\end{aligned}\end{equation}
From (A1), (A2), and (B1), we know that the second term in the right-hand side of \eqref{eq:w_n} is $ \gtrsim \frac{\|a\|^2 \|b\|^2}{n^2}$.
Hence the third term in \eqref{eq:w_n} is negligible compared to the second term.
Moreover, since ${\rm Var}\left(\sum a_{ij}B_{ij}\right)$ has the following expression \cite[Equation (2.6)]{Pham1989permu}:
\begin{equation*}\label{eq:Vara_ijB_ij}\begin{aligned}
{\rm Var}\left(\sum a_{ij}B_{ij}\right) &= \frac{4}{(n-1)(n-2)^2}\left(\sum a_{i+}^2\right)\left(\sum b_{i+}^2\right)\\
&\quad +\frac{2}{n(n-3)}\left( \sum a_{ij}^2 - \frac{2n-5}{(n-2)^2} \sum a_{i+}^2\right)\left( \sum b_{ij}^2 - \frac{2n-5}{(n-2)^2} \sum b_{i+}^2\right),
\end{aligned}\end{equation*}
we have $w_n^2/{\rm Var}(\sum a_{ij}B_{ij}) \to 1$.

Now consider the remaining graphs $G$ consisting of connected components $G_1,\ldots,G_p$, each of which has exactly 2 edges.
Hence necessarily, $r=2p$ is even in order that $\E\left[ \left(\frac{\sum a_{ij}B_{ij}}{{\rm Var}(\sum a_{ij}B_{ij})} \right)^r \right]$ has a non-zero limit.
Each $G_q$ has two possible configurations\footnote{the configuration of a graph is the blank graph obtained by disregarding the direction of all edges}, either having 2 vertices and 2 edges forming a loop, or having 3 vertices and 2 edges forming a tree.
For the loop configuration, $\Sigma(a,G_q) = \sum'a_{ij}^2 = \sum a_{ij}^2$.
For the tree configuration, $\Sigma(a,G_q) = \sum' a_{ij}a_{ik} =  \sum a_{ij}a_{ik}$.
Since $\Sigma(a,G) = \sum_{i_1,\ldots, i_{|G|}}' \prod_{(\mu,\nu)\in G} a_{i_\mu i_\nu}$,
if we relax the constraint that $i_1,\ldots, i_{|G|}$ are distinct,
then we get $\Sigma(a,G_1\cup \cdots\cup G_p) = \prod_{q=1}^p \Sigma(a,G_q) + R_1$.
The difference $R_1$ is the sum of a number of $\Sigma(a,G')$, where $G'$ is obtained from identifying vertices of $G_1\cup \cdots\cup G_p$.
Since $G'$ has at most $p-1$ connected components, by \eqref{eq:ordera}, $\Sigma(a,G')=O\left(n^{p-1-\frac{r}{2}}t^{|G|}_n \cdot \|a\|^r\right)=O\left(\frac{t^{|G|}_n}{n} \cdot \|a\|^r\right)$.
Similarly, we can relax the constraint that $i_1,\ldots, i_{|G|}$ are distinct in $\Sigma(b,G_1\cup \cdots\cup G_p)$,
which leads to $\Sigma(b,G) = \prod_{q=1}^p \Sigma(b,G_q) + R_2$ and the difference $R_2$ is a sum of a number of $\Sigma(b,G')$.
Since $|G'|\leq |G|-1$, by \eqref{eq:orderb}, $\Sigma(b,G')=O\left(n^{|G|-r-1}\cdot \|b\|^r\right)$.
Note that \eqref{eq:orderb} and \eqref{eq:ordera} also give
$\Sigma(a,G)=O\left(t^{|G|}_n \cdot \|a\|^r\right)$, $\Sigma(b,G)=O\left(n^{|G|-r}\cdot \|b\|^r\right)$.
Hence together with $w_n\gtrsim \frac{\|a\|\cdot  \|b\|}{n}$,
$$\begin{aligned}
\frac{1}{n(n-1)\cdots(n-|G|+1)} \Sigma(b,G)\Sigma(a,G) &= \frac{(n-|G|)!}{n!}\prod_{q=1}^p \Sigma(a,G_q)\Sigma(b,G_q) + o(w_n^r)\\
&=\frac{1}{n^{|G|}}\prod_{q=1}^p \Sigma(a,G_q)\Sigma(b,G_q) + o(w_n^r).
\end{aligned}$$
The second equality in the above display follows as $\frac{(n-|G|)!}{n!} - \frac{1}{n^{|G|}} = O\left(\frac{1}{n^{|G|+1}}\right)$ and \\ $\prod_{q=1}^p \Sigma(a,G_q)\Sigma(b,G_q) = O\left(n^{|G|-r}t^{|G|}_n \cdot \|a\|^r\|b\|^r \right)$.
Therefore, \eqref{eq:expmoment} reduces to
\begin{equation*}\begin{aligned}
\E\left(\frac{\sum a_{ij}B_{ij}}{w_n} \right)^r &=\frac{1}{w_n^r}\sum_{G\in \tilde{G}(r)}\frac{1}{n^{|G|}}\prod_{q=1}^p \Sigma(a,G_q)\Sigma(b,G_q) + o(1),
\end{aligned}\end{equation*}
where $\tilde{G}(r)$ contains the graphs that consists of connected components $G_1,\ldots,G_{p}$, each of which has exactly 2 edges.

Suppose $i$ of $G_1,\ldots,G_{p}$ have a loop configuration and $p-i$ of them have a tree configuration. Then $|G| = 2i+3(p-i)$ and
$$\begin{aligned}
\frac{1}{n^{|G|}}\prod_{q=1}^p \Sigma(a,G_q)\Sigma(b,G_q)&= \frac{1}{n^{3p-i}} \left(\sum a_{ij}^2 \right)^i\left(\sum b_{ij}^2 \right)^i
\left( {\sum}'a_{ij} a_{ik}\right)^{p-i}\left( {\sum}'b_{ij} b_{ik}\right)^{p-i}\\
&=\frac{1}{n^{3p-i}} \xi^i \eta^{p-i},
\end{aligned}$$
where $\xi :=  \left(\sum a_{ij}^2 \right) \left(\sum b_{ij}^2 \right)$ and $\eta := \left({\sum}'a_{ij} a_{ik}\right)\left({\sum}'b_{ij} b_{ik}\right)$.

The number of graphs having $i$ loops and $p-i$ trees is given by:
$$\frac{\tbinom{2p}{2}\tbinom{2p-2}{2}\cdots \tbinom{2}{2}}{i!(p-i)!}\cdot 2^i \cdot 4^{p-i} = \frac{(2p)!2^{p-i}}{i!(p-i)!},$$
where we first divide the $2p$ different edges into $i$ loops and $p-i$ trees, and then determine the orientation of each loop and tree.
Therefore,
$$\begin{aligned}
\sum_{G\in \tilde{G}(r)}\frac{1}{n^{|G|}}\prod_{q=1}^p \Sigma(a,G_q)\Sigma(b,G_q)&=\sum_{i=0}^p \frac{(2p)!2^{p-i}}{i!(p-i)!} \cdot \frac{1}{n^{3p-i}} \xi^i \eta^{p-i}\\
&=\frac{(2p)!(2\eta)^p}{n^{3p}p!}\sum_{i=0}^p \frac{p!}{i!(p-i)!} \left(\frac{\xi}{2\eta n}\right)^i\\
&=\frac{(2p)!(2\eta)^p}{n^{3p}p!} \left( 1+\frac{\xi}{2\eta n}\right)^p\\
&=\frac{(2p)!}{p!} \left( \frac{2\eta}{n^3}+\frac{\xi}{n^4}\right)^p\\
&=\frac{(2p)!}{2^p p!} \left( \frac{4\left({\sum}'a_{ij} a_{ik}\right)\left({\sum}'b_{ij} b_{ik}\right)}{n^3}+\frac{2\left(\sum a_{ij}^2 \right) \left(\sum b_{ij}^2 \right)}{n^4}\right)^p\\
&=\frac{(2p)!}{2^p p!} w_n^{2p}.
\end{aligned}$$
Note that $\frac{(2p)!}{2^p p!} $ is the $2p$-th moment of $N(0,1)$. Hence the proof is completed.\qed

\section{Further Simulations}\label{sec:further_simu}
In this section we provide further simulation studies to support the major results provided in the main paper.
\subsection{Validity of Theorem~\ref{thm:asympnull}}
Figure~\ref{nullclt} shows the histogram of $\frac{\hat{\eta}}{\sqrt{{\rm Var}(\hat{\eta}|\mathcal{F}_n)}}$ from 20000 independent replications constructed using the directed 1-NN graph and the discrete kernel.
Here we take $M=3$ and the three distributions are $P_1=P_2=P_3\equiv N(\mathbf{0},I_d)$ with equal sample sizes $n_i$ and $d=2$.
The red curve is the standard normal density function.
It can be seen from the plots that the empirical distribution is already close to the standard normal distribution for $n_i=100$. The approximation gets even better when $n_i=1000$, and when $n_i=10000$, the empirical distribution is almost identical to the standard normal.

\begin{figure}
    \centering
    \includegraphics[width = 1\textwidth]{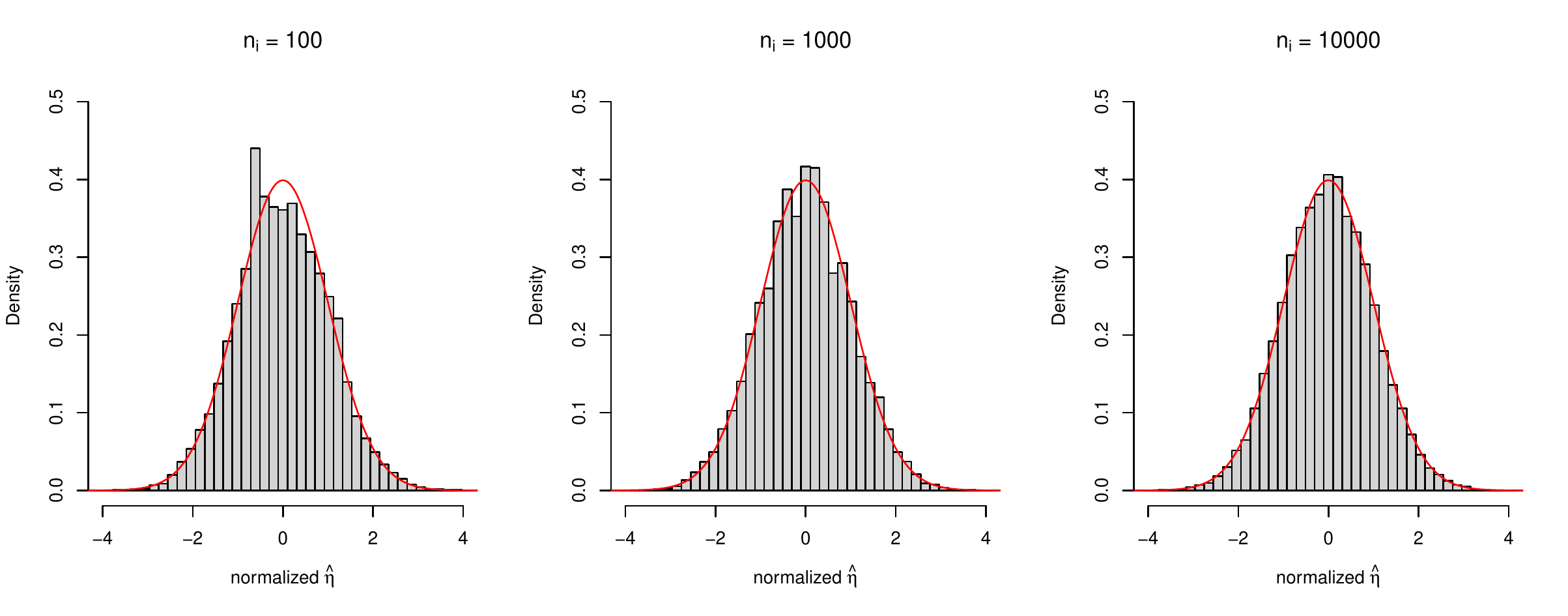}
    \caption{Histograms of $\frac{\hat{\eta}}{\sqrt{{\rm Var}(\hat{\eta}|\mathcal{F}_n)}}$ as $n_1 = n_2 = n_3$ vary. The red curves are the standard normal density.}
    \label{nullclt}
\end{figure}

\subsection{Validation of Theorem~\ref{thm:detection_threshold}}\label{sec:check_detection_threshold}
Here we empirically check the validity of Theorem~\ref{thm:detection_threshold} on the detection threshold of our method.
\begin{figure}
    \centering
    \includegraphics[width = 1\textwidth]{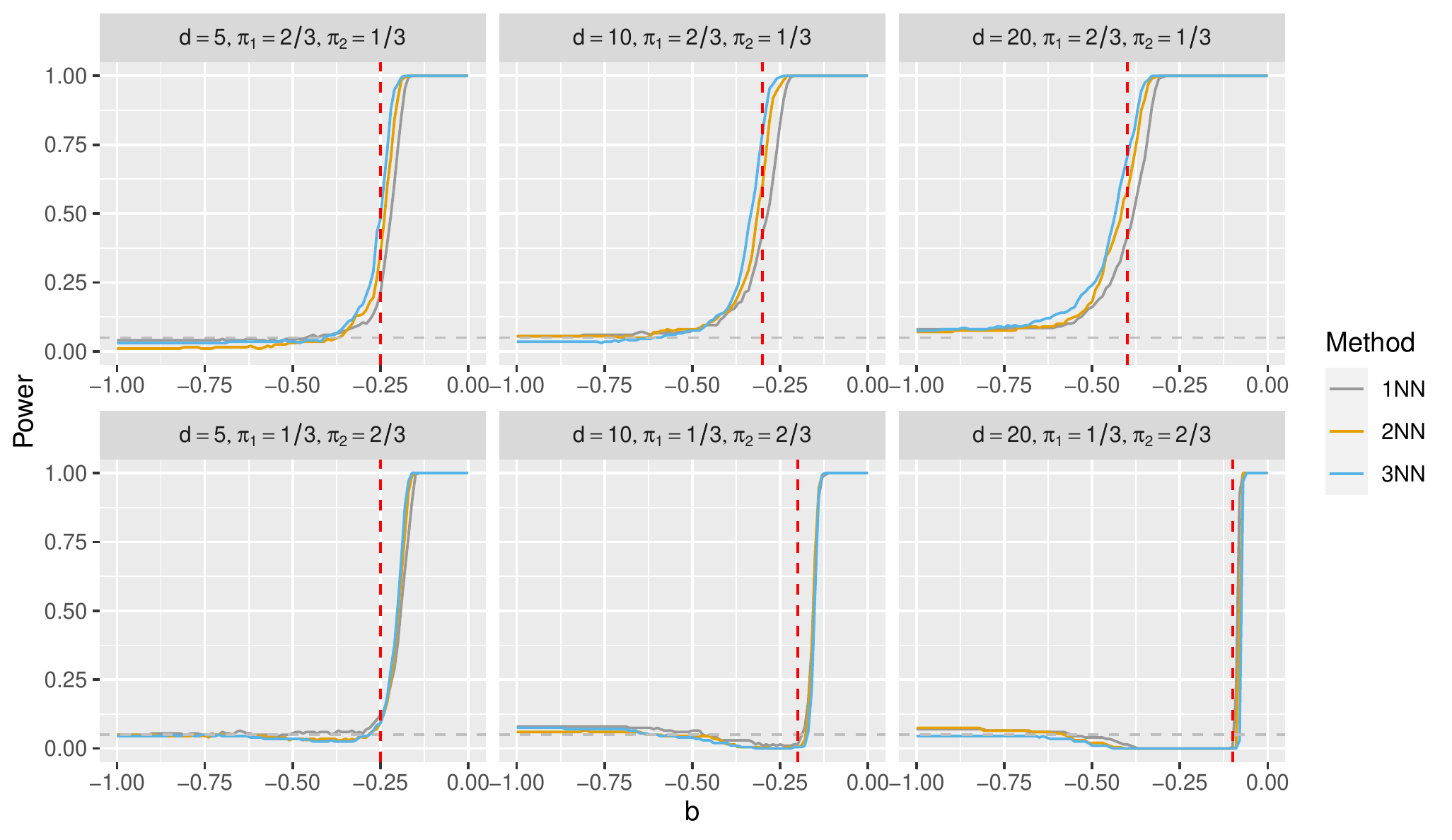}
    \caption{Empirical power of our method with $n_1$ samples from $P_1=N\left(\mathbf{0},3^2\cdot I_d\right)$ and $n_2$ samples from $P_2=N\left(\mathbf{0},(3+2 n^b)^2\cdot I_d\right)$. $(n_1,n_2)$ is taken as either $(2\cdot 10^4,10^4)$ or $(10^4,2\cdot 10^4)$, corresponding to $(\pi_1,\pi_2)=(\frac{2}{3},\frac{1}{3})$ or $(\frac{1}{3},\frac{2}{3})$. The red vertical line shows the detection threshold predicted by Theorem~\ref{thm:detection_threshold}.}
    \label{fig:threshold}
\end{figure}
Although Theorem~\ref{thm:detection_threshold} was proved in the Poissonized setting, our simulation results suggest that the conclusions also hold true in the usual non-Poissonized regime. We consider $M=2$ with $n_1$ samples from $P_1=N\left(\mathbf{0},3^2\cdot I_d\right)$ and $n_2$ samples from
$P_2=N\left(\mathbf{0},(3+2 n^b)^2\cdot I_d\right)$ for $b\in (-1,0)$, $n=n_1+n_2$. Here
$(n_1,n_2)$ is taken as either $(2\cdot 10^4,10^4)$ or $(10^4,2\cdot 10^4)$, corresponding to $(\pi_1,\pi_2)=(\frac{2}{3},\frac{1}{3})$ or $(\frac{1}{3},\frac{2}{3})$.
Figure~\ref{fig:threshold} shows the empirical power of the asymptotic test \eqref{eq:rej_region} averaged over 200 replications.
The level of the test is set as $\alpha=0.05$ and  $k$-NN graphs are used with $k=1,2,3$.
The red vertical line shows the detection threshold predicted by Theorem~\ref{thm:detection_threshold}.
When $d=5$, the detection threshold is at $n^{-\frac{1}{4}}$. It can be seen that for either choice of $(\pi_1,\pi_2)$,
the power increases from 0.05 to 1 around $b=-\frac{1}{4}$.
For $d=10$ and $d=20$, the detection threshold depends on the sign of $a_{k,\theta_1}$ (see \eqref{eq:a}).
When $(\pi_1,\pi_2)=(\frac{2}{3},\frac{1}{3})$, $a_{k,\theta_1} > 0$, and the detection threshold is $n^{-\frac{1}{2}+\frac{2}{d}}$,
so the power increases from 0.05 to 1 around $b=-\frac{1}{2}+\frac{2}{d}$ (see the middle and right plots in the first row of Figure~\ref{fig:threshold}).
When $d$ is large, this threshold gets closer and closer to the parametric threshold $-\frac{1}{2}$ (see the top right plot in Figure~\ref{fig:threshold}).
When $(\pi_1,\pi_2)=(\frac{1}{3},\frac{2}{3})$, $a_{k,\theta_1} < 0$, and the detection threshold is $n^{-\frac{2}{d}}$,
so there is a rapid increase of the power from 0 to 1 around $b=-\frac{2}{d}$. In such a case, the limiting power for $b\in(-\frac{1}{2}+\frac{2}{d}, -\frac{2}{d})$ is 0 as predicted by Theorem~\ref{thm:detection_threshold}, and this is also supported by the last two plots in the second row of Figure~\ref{fig:threshold}, where the empirical power for $b\in(-\frac{1}{2}+\frac{2}{d}, -\frac{2}{d})$ is close to 0.

\subsection{Choice of $k$ for the $k$-NN Graph}\label{sec:choose_k}
The choice of $k$ for the $k$-NN graph may depend on the task at hand. From the following experiments we see that
for testing the equality of the $M$ distributions, the empirical criterion $k= 0.10n$, for samples with up to a few hundred observations~\citep{Petrie2016}, seems to provide a good choice. However, for estimating $\eta$ using its empirical version $\hat{\eta}$, a much smaller $k$ is recommended --- often $k=1$ works best. 

\noindent \textbf{Choice of $k$ in testing}: In Figure~\ref{diffK} we show the empirical power (over 1000 replications) of our test statistics, when the level is set at 0.05, using different $k$-NN graphs.

\noindent For $M=3$, the three distributions we consider are:
\begin{enumerate}
\item Normal location problem: $P_1 = N(\mathbf{0},I_d)$, $P_2 = N(0.1\cdot \mathbf{1},I_d)$, $P_3=N(0.2\cdot \mathbf{1},I_d)$.
\item Normal scale problem: $P_1=N(\mathbf{0},I_d)$, $P_2=N(\mathbf{0},1.5\cdot I_d)$, $N(\mathbf{0},2\cdot I_d)$.
\item Non-Gaussian problem with $t$-distribution: $P_1$ has each entry following independent $t(1)$ with noncentrality parameter $\delta$; $P_2$ and $P_3$ have each entry following independent $t(1)$. The dimension $d$ is set to be 16.
\end{enumerate}
For $M=5$, the distributions we consider are:
\begin{enumerate}
\item Normal location problem: $P_1 = N(\mathbf{0},I_d)$, $P_2 = N(0.05\cdot \mathbf{1},I_d)$, $P_3=N(0.1\cdot \mathbf{1},I_d)$,
$P_4=N(0.15\cdot \mathbf{1},I_d)$, $P_5=N(0.2\cdot \mathbf{1},I_d)$.
\item Normal scale problem: $P_1=N(\mathbf{0},I_d)$, $P_2=N(\mathbf{0},1.25\cdot I_d)$, $P_3=N(\mathbf{0},1.5\cdot I_d)$, $P_4=N(\mathbf{0},1.75\cdot I_d)$, $P_5=N(\mathbf{0},2\cdot I_d)$.
\item Non-Gaussian problem with $t$-distribution: $P_1$ has each entry following independent $t(1)$ with noncentrality parameter $\delta$; $P_i$, $i=2,3,4,5$ have each entry following $t(1)$. The dimension is set to be 16.
\end{enumerate}
Here $n_i$, for $i=1,\ldots,M$, are set to be equal. Different combinations of $n$ and $d$ are considered in our simulation experiments.
It can be seen from Figure~\ref{diffK} that the empirical criterion $k=0.1n$ indeed provides reasonably good performance: For the normal scale problem with discrete kernel $k=0.1n$ is close to the optimal $k$,
while for the problem with $t$-distribution the power continues to increase until $k\approx 0.2n$. For the normal location problem and the normal scale problem with kernel $K_2$ (recall that $K_2(1,1) = 10$, $K_2(2,2) = K_2(3,3) = 1$, $K_2(i,j)=0$,for $i\neq j$), the power still increases as $k$ increases even after $k\approx 0.5n$ (see the plots corresponding to $n=300$ in Figure~\ref{diffK}). \newline

\begin{figure}
    \centering
    \includegraphics[width = 1\textwidth]{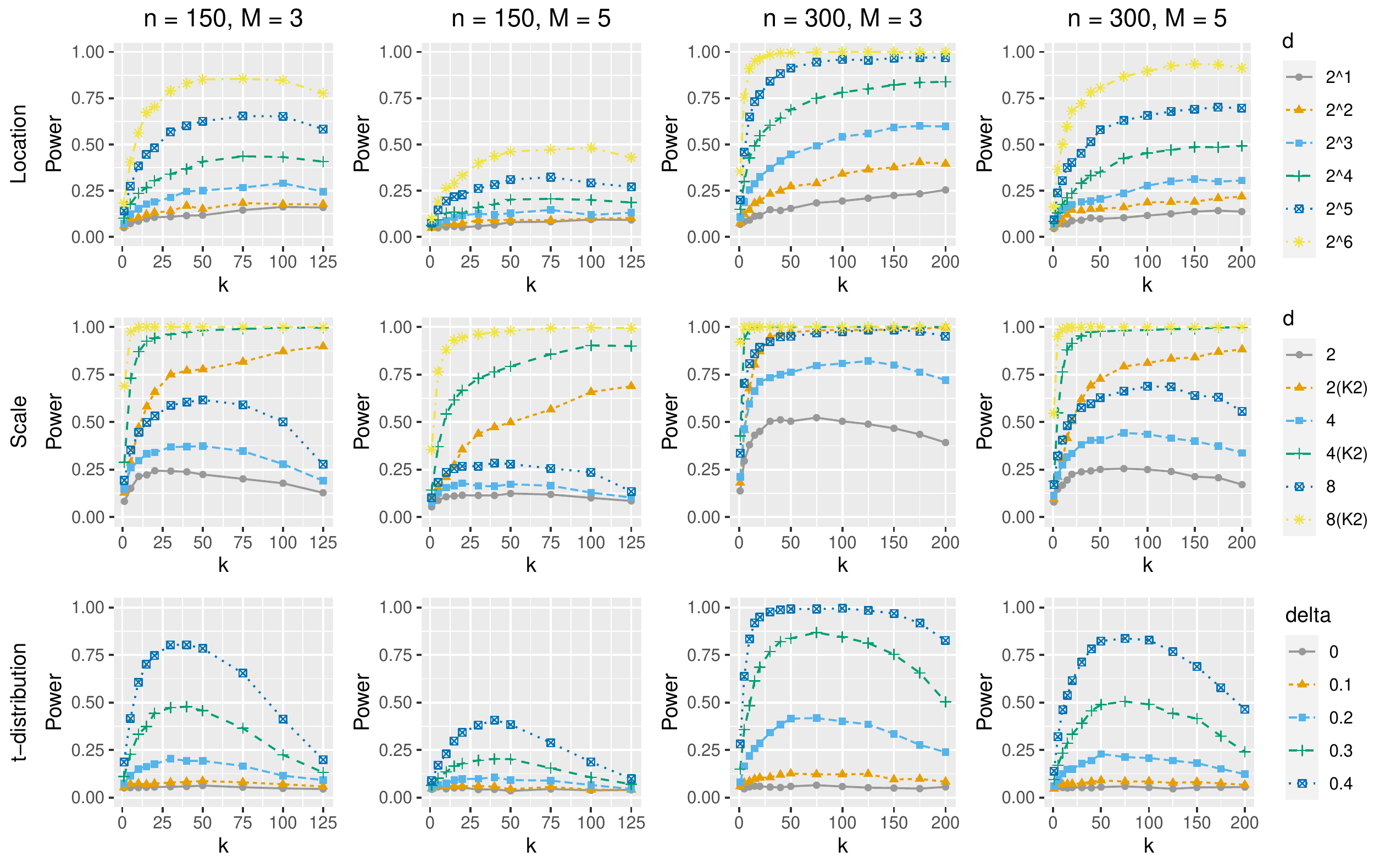}
    \caption{Power of the tests using different $k$-NN graphs.}
    \label{diffK}
\end{figure}

\noindent \textbf{Choice of $k$ in estimating $\eta$}: Although power increases as $k$ increases until or even after $k=0.1n$, the estimate of $\eta$ may no longer be accurate with a large $k$, since the $k$-NNs of a point may no longer be close to that point.
When estimating $\eta$, we face a bias-variance trade-off. When $\eta=0$, $\hat{\eta}$ is unbiased, so only the variance needs to be controlled.
When $\eta$ is large, the variance is of the order $n^{-1}$, but the bias can be of the order $n^{-1/d}$, a typical distance between a point and its nearest neighbor, and is the dominating term when the dimension gets large. This issue of the bias could be alleviated by choosing a small $k$ --- as in other $k$-NN applications, a smaller $k$ usually has less bias (see Figure~\ref{eta_diffK}).

Another observation is that, in both simulated and real data, as $k$ increases, the actual value of $\hat{\eta}$ often decreases. In the most extreme case where $k=n-1$, $\hat{\eta}$ is exactly 0.
For the above reasons, we would suggest that when $\eta$ or the dimension $d$ are suitably large, $k$ should be chosen to be as small as possible; often $k=1$ would be the best choice. However, when $\eta$ is close to 0 and the dimension $d$ is small, $k$ can be suitably increased to reduce the variance of $\hat \eta$.

Figure~\ref{eta_diffK} shows the mean of $\hat{\eta}$ under different settings, where we consider
\begin{enumerate}
\item $\eta = 1$:
\begin{enumerate}
\item $M =2$, $P_1 = {\rm uniform}[0,1]^d$, $P_2 = {\rm uniform}[0,1]^d + e_1$,
\item $M=3$, $P_1 = {\rm uniform}[0,1]^d$, $P_2 = {\rm uniform}[0,1]^d + e_1$, $P_3 = {\rm uniform}[0,1]^d + e_2$,
\end{enumerate}
\item $\eta = 0.5$:
\begin{enumerate}
\item $M =2$, $P_1 = {\rm uniform}[0,1]^d$, $P_2 = {\rm uniform}[0,1]^d + \frac{1}{2}e_1$,
\item $M=3$, $P_1 =P_2= {\rm uniform}[0,1]^d$, $P_3 = {\rm uniform}[0,1]^d + e_1$,
\end{enumerate}
\item $\eta = 0.1$:
\begin{enumerate}
\item $M=2$, $P_1 = {\rm uniform}[0,1]^d$, $P_2 = {\rm uniform}[0,1]^d + 0.1e_1$,
\item $M=3$, $P_1 =P_2= {\rm uniform}[0,1]^d$, $P_3 = {\rm uniform}[0,1]^d + 0.2e_1$.
\end{enumerate}
\end{enumerate}
Here $e_i$ denotes an all-zero vector except a 1 at the $i$-th entry.
\begin{figure}
    \centering
    \includegraphics[width = 1\textwidth]{diffKd20.pdf}
    \caption{The mean of $\hat{\eta}$ for different values of $\eta$ as $k$ (in the $k$-NN graph) varies (for different choices of $M,n$, and $d$), computed using $10^4$ independent replications.}
    \label{eta_diffK}
\end{figure}
For $M=2$, $\pi_1=\pi_2=1/2$, and for $M=3$, $\pi_1=\pi_2=\pi_3=1/3$.
We consider $(n,d)\in\{60,300\}\times\{2,200\}$. A larger $d$ indicates a higher noise level. In each of our $10^4$ replications, we have $n_i=n\pi_i$ samples from $P_i$, and the mean of $\hat{\eta}$ is reported.

We observe in Table \ref{eta_diffK} that with a discrete kernel, $\hat{\eta}$ tends to underestimate $\eta$, and the bias increases as $k$ increases. Here $k=1$ gives the most accurate estimate in the sense that it has the least bias in all scenarios.

The empirical mean squared error of $(\hat{\eta}-\eta)^2$ is given in Figure~\ref{diffKmse}.
\begin{figure}
    \centering
    \includegraphics[width = 1\textwidth]{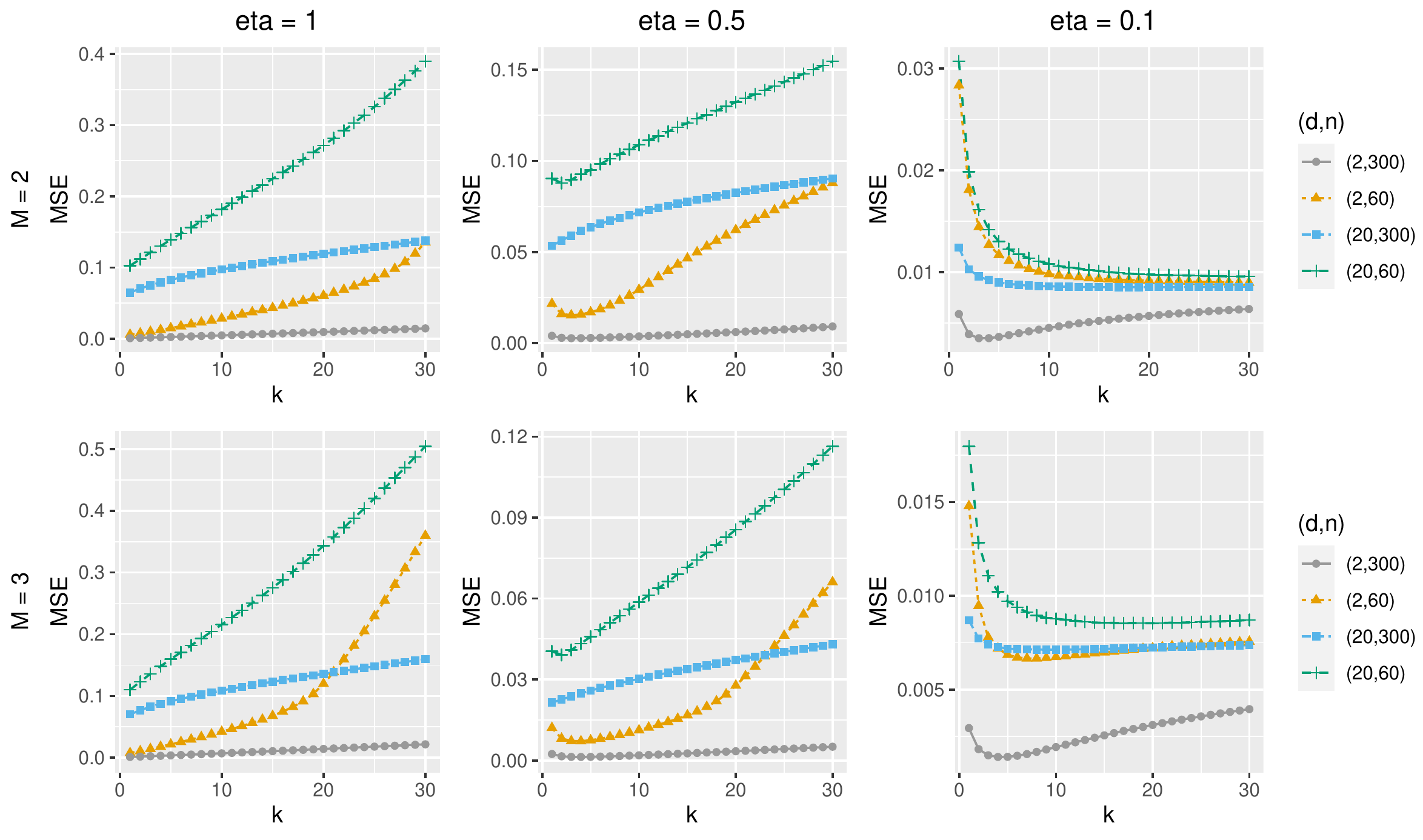}
    \caption{The mean squared error of $\hat{\eta}$ for different values of $\eta$ as $k$ (in the $k$-NN graph) varies (for different choices of $M,n$, and $d$), computed using $10^4$ independent replications.}
    \label{diffKmse}
\end{figure}
It can be seen that when $\eta=1$, the mean squared error (MSE) strictly increases as $k$ increases; when $\eta=0.5$, $k=1$ is also close to the optimal $k$. When $\eta=0.1$, which is close to 0, the MSE first decreases and then increases as we increase $k$. Moreover, when $\eta = 0.1$, the MSE is small for all choices of $k$ compared to the previous examples where $\eta=0.5$ or $1$. When $d=20$, it is clear from Figures~\ref{eta_diffK} and \ref{diffKmse} that the bias dominates the variance in its contribution to the MSE. 
When we get closer to the null $\eta=0$ (where $\hat{\eta}$ is unbiased), the bias decreases to 0, so the MSE also decreases dramatically. For the case $\eta=0.1$, one may slightly increase $k$ to 4 or 5 to reduce the variance and obtain the optimal MSE.

\subsection{Validation of Theorem~\ref{thm:DistFree}}\label{sec:validation_distfree}
Here we empirically verify our theoretical results in Theorem~\ref{thm:DistFree}, which states that under ${\rm H}_0$, the (conditional) variance of $\hat{\eta}$ converges to a distribution-free constant not depending on the common distribution, thus yielding an asymptotic test using this distribution-free variance. We first examine the convergence of $\tilde{g}_i$ to $g_i$.
\begin{figure} 
    \centering
    \includegraphics[width = 1\textwidth]{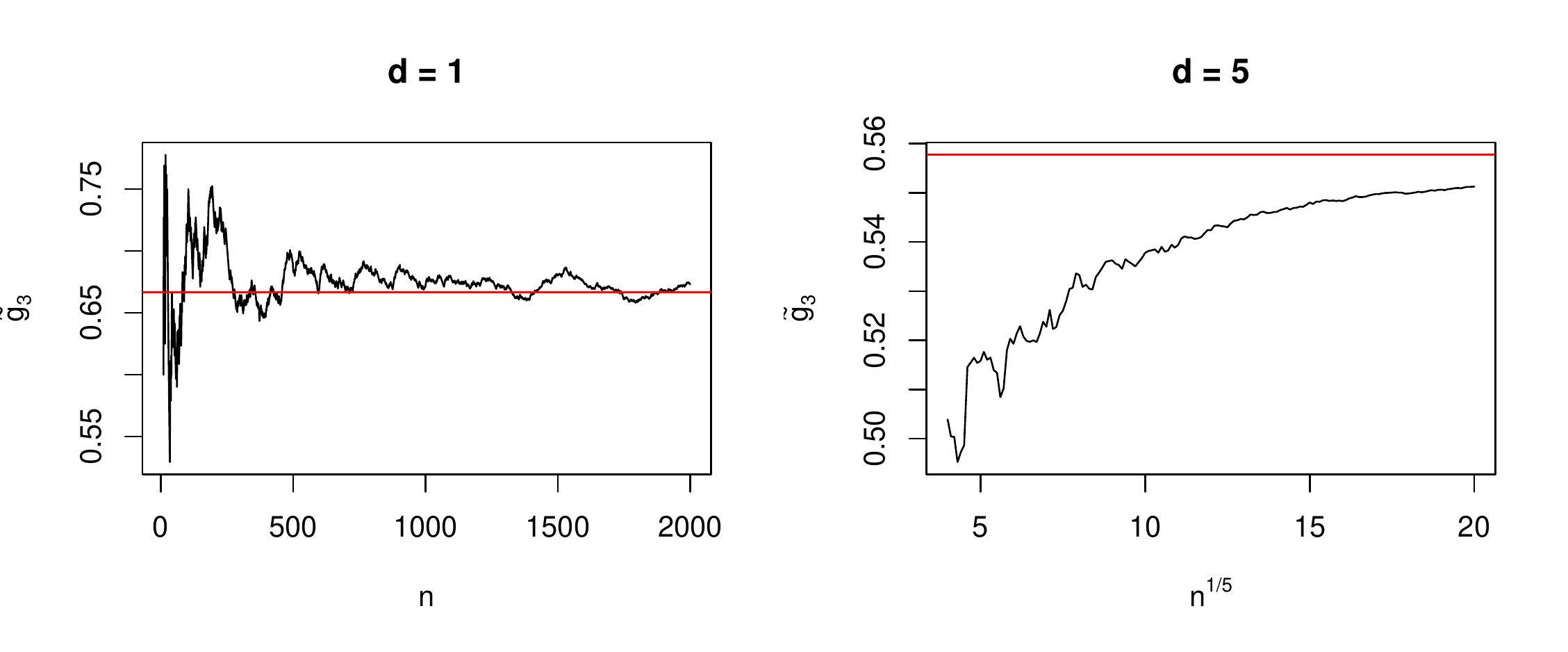}
    \caption{The computed $\tilde{g}_3$  as the sample size $n$ varies. The red line shows $g_3$.}
    \label{tilde_g3}
\end{figure}
Consider $n$ observations from three identical distributions $P_1=P_2=P_3 \equiv N(\mathbf{0},I_d)$. Each observation has probability $1/3$ to be sampled from $P_i$, $i=1,2,3$; so $\pi_1=\pi_2=\pi_3=1/3$.
With a directed 1-NN graph,
$g_1=g_2=1$, and $g_3$ depends on $d$. The exact value of $g_3$ can be obtained from \citet{Henze1986mutualNN}.
When $d=1$, it can be seen from the left panel of Figure~\ref{tilde_g3} that $\tilde{g}_3$ converges to $g_3$ as $n$ increases. When $n=2000$, the difference between $\tilde{g}_3$ and $g_3$ is only 0.006. However, the convergence becomes slower as the dimension $d$ becomes higher, as is seen from the right panel of Figure~\ref{tilde_g3}.
When $d=5$, the difference between $\tilde{g}_3$ and $g_3$ is 0.043 when $n=2000$, and it takes $n=3.2\times 10^6$ observations to bring the error down to 0.006.

Next, since both Theorems~\ref{thm:asympnull} and \ref{thm:DistFree} provide a CLT for $\hat{\eta}$, we compare the quality of the approximation given by the two theorems.
\begin{figure} 
    \centering
    \includegraphics[width = 1\textwidth]{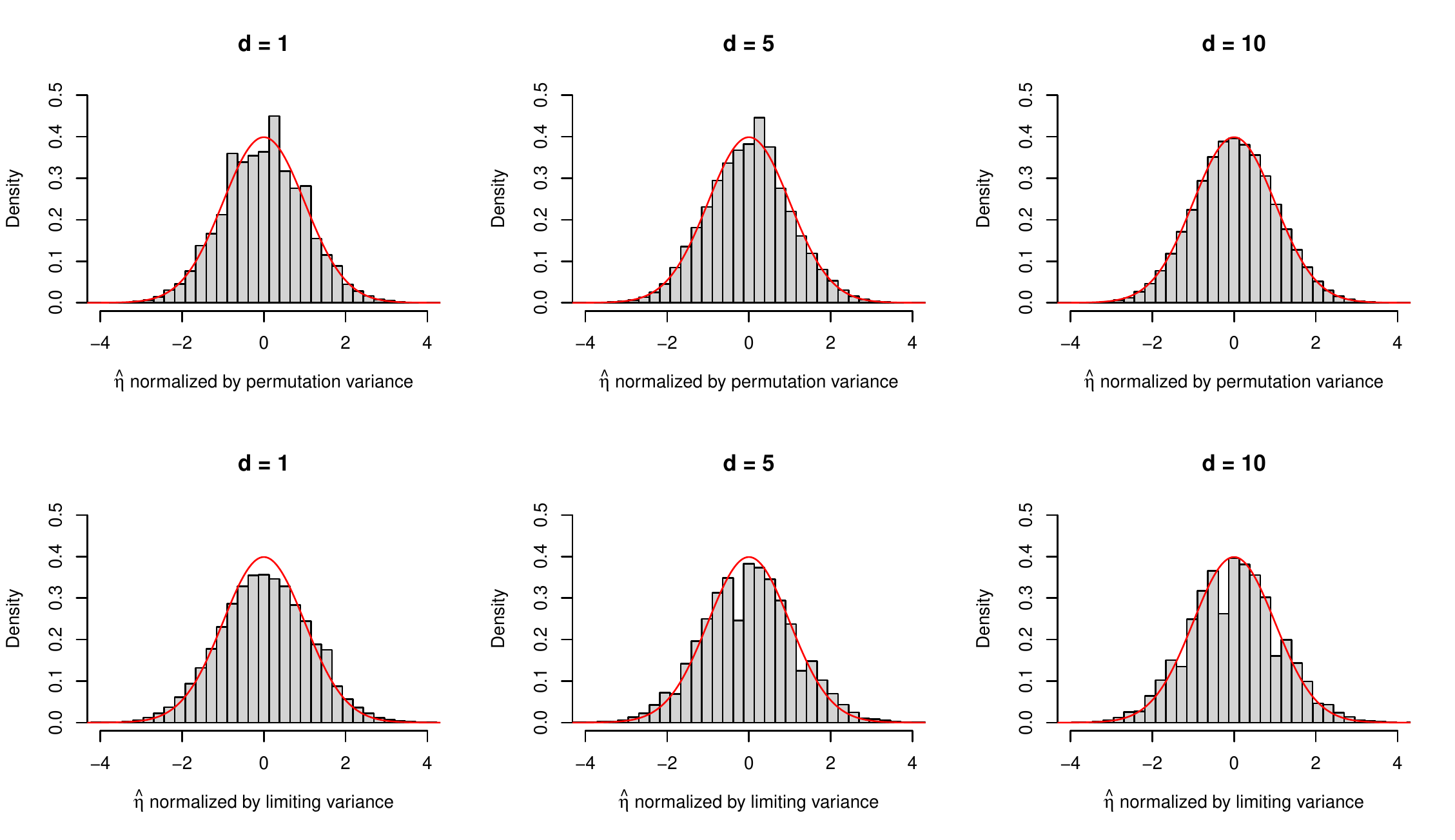}
    \caption{Histograms of normalized $\hat{\eta}$, either by the permutation variance in Theorem~\ref{thm:asympnull} or by the limiting variance in Theorem~\ref{thm:DistFree}. The red curves show the standard normal densities.}
    \label{Thm3vs4}
\end{figure}
We consider $n_1=n_2=n_3=150$, $P_1=P_2=P_3 \equiv N(\mathbf{0},I_d)$ and vary $d \in \{1,5,10\}$, and use a directed 1-NN graph to compute $\hat \eta$. Figure~\ref{Thm3vs4} shows the histograms of $\hat{\eta}$ normalized its limiting variance, i.e., $\frac{\sqrt{n}\hat{\eta}} {\sigma_{\mathcal{G},K,d,\pi}}$,
along with the histograms of $\hat{\eta}$ normalized by its permutation variance (as in Theorem~\ref{thm:asympnull}), i.e.,
$\frac{\hat{\eta}}{\sqrt{{\rm Var}(\hat{\eta}|\mathcal{F}_n)}}$.
The red curves are the standard normal densities. It can be seen that all the histograms are close to the standard normal density, and using the permutation variance seems to yield a better approximation than using the limiting variance. Therefore, we recommend using Theorem~\ref{thm:asympnull} over Theorem~\ref{thm:DistFree} for testing equality of distributions in practice.

\section*{Acknowledgement}
The authors would like to thank Bhaswar Bhattacharya and Nabarun Deb for a number of useful comments and references.

\bibliographystyle{chicago}
\bibliography{OT}

\end{document}